\documentclass[11pt,reqno]{amsart}
\usepackage[margin=2.5cm]{geometry}
\allowdisplaybreaks
\usepackage[english]{babel}
\usepackage[latin1]{inputenc}
\usepackage[T1]{fontenc}
\usepackage{graphicx}
\usepackage{subfig}
\usepackage[all]{xy}
\usepackage{caption}
\usepackage{url}
\usepackage{xcolor}
\usepackage{amssymb,amsmath,amstext,amsfonts,amsthm}
\usepackage{mathrsfs}
\usepackage{mathtools}
\usepackage{verbatim}
\usepackage{hhline}
\usepackage{relsize}
\usepackage[colorlinks=true, urlcolor=blue, linkcolor=blue, citecolor=blue]{hyperref}
\usepackage{fancyvrb}
\usepackage{tikz-cd}
\usepackage{overpic}

\numberwithin{equation}{section}
\theoremstyle{plain}
\newtheorem*{theorem*}{Theorem}

\newtheorem{theorem}{Theorem}
\numberwithin{theorem}{section}
\newtheorem{proposition}[theorem]{Proposition}
\newtheorem{lemma}[theorem]{Lemma}
\newtheorem{corollary}[theorem]{Corollary}
\newtheorem{conjecture}[theorem]{Conjecture}

\theoremstyle{definition}
\newtheorem{definition}[theorem]{Definition}
\newtheorem{remark}[theorem]{Remark}
\newtheorem{example}[theorem]{Example}

\theoremstyle{definition}
\newtheorem*{defn*}{Definition}
\theoremstyle{plain}
\newtheorem*{thm*}{Theorem}
\theoremstyle{plain}
\newtheorem*{prop*}{Proposition}
\theoremstyle{plain}
\newtheorem*{conj*}{Conjecture}
\theoremstyle{plain}
\newtheorem*{ex*}{Example}

\newcommand{\pr}{\mathrm{pr}}
\newcommand{\C}{\mathbb{C}}
\newcommand{\G}{\mathbb{G}}
\newcommand{\K}{\mathbb{K}}

\newcommand{\Z}{\mathbb{Z}}
\newcommand{\PP}{\mathbb{P}}

\newcommand{\R}{\mathbb{R}}

\newcommand{\sI}{\mathcal{I}}
\newcommand{\sE}{\mathcal{E}}

\newcommand{\sG}{\mathcal{G}}
\newcommand{\sK}{\mathcal{K}}

\newcommand{\sP}{\mathcal{P}}

\newcommand{\sQ}{\mathcal{Q}}

\newcommand{\sN}{\mathcal{N}}

\newcommand{\sT}{\mathcal{T}}
\newcommand{\sO}{\mathcal{O}}
\newcommand{\sU}{\mathcal{U}}
\newcommand{\sV}{\mathcal{V}}
\newcommand{\sZ}{\mathcal{Z}}
\newcommand{\sPE}{\mathcal{PE}}
\newcommand{\sPG}{\mathcal{PG}}
\newcommand{\sPN}{\mathcal{PN}}
\newcommand{\mR}{\mathsmaller{\mathbb{R}}}

\def\sm{{\rm sm}}
\def\sing{{\rm sing}}
\def\codim{{\rm codim}}
\def\rank{{\rm rank}}
\def\image{{\rm im}}
\newcommand{\mT}{\mathsmaller{\mathsf{T}}}

\newcommand{\red}[1]{{\color{red}#1}}
\newcommand{\blue}[1]{{\color{blue}#1}}

\makeatletter
\newcommand*{\rom}[1]{\expandafter\@slowromancap\romannumeral #1@}
\makeatother

\makeatletter
\@namedef{subjclassname@2020}{%
	\textup{2020} Mathematics Subject Classification}
\makeatother

\title[Conditional Euclidean distance optimization via relative tangency]{Conditional Euclidean distance optimization \\ via relative tangency}

\author{Sandra Di Rocco}
\author{Lukas Gustafsson}
\author{Luca Sodomaco}
\address{Department of Mathematics, KTH Royal Institute of Technology, SE-100 44 Stockholm, Sweden}
\email{dirocco@kth.se}
\email{lukasgu@kth.se}
\email{sodomaco@kth.se}



\date{}

\begin{document}

\begin{abstract}
We introduce a theory of relative tangency for projective algebraic varieties. The dual variety $X_Z^\vee$ of a variety $X$ relative to a subvariety $Z$ is the set of hyperplanes tangent to $X$ at a point of $Z$. We also introduce the concept of polar classes of $X$ relative to $Z$.
We explore the duality of varieties of low rank matrices relative to special linear sections.
In this framework, we study the critical points of the Euclidean Distance function from a data point to $X$, lying on $Z$. The locus where the number of such conditional critical points is positive is called the ED data locus of $X$ given $Z$. The generic number of such critical points defines the conditional ED degree of $X$ given $Z$. We show the irreducibility of ED data loci, and we compute their dimensions and degrees in terms of relative characteristic classes.
\end{abstract}

\maketitle

\section{Introduction}

Duality in projective algebraic geometry is a natural notion: for an algebraic variety $X\subset\PP^N$, its {\em dual variety} $X^\vee$ is given by all points in the dual projective space $(\PP^N)^\vee$ whose associated hyperplanes in $\PP^N$ are tangent to $X$. The dual variety plays an essential role in several central problems like optimization or sampling over algebraic varieties.
For example, when optimizing a linear cost function over a real algebraic variety $X,$ the equations and the degree of $X^\vee$ provide relevant information on the complexity of the linear semidefinite program \cite[Section 5.3]{blekherman2012semidefinite}. The investigation of tangential properties of embedded varieties has also garnered extensive attention due to their pivotal role in understanding the shape and topology of these varieties. When datasets exhibit an underlying algebraic model structure, resembling points on or near an algebraic variety, a deep understanding of their geometry can lead to the discovery of significant patterns that serve as key elements in data analysis (see \cite{nicolau2011topology}, \cite{blair2021phenotipic} for examples).

The degree of a dual variety $X^\vee$ is also the degree of a specific polar class of $X$. Polar classes are invariants of projective varieties, encoding the non-generic tangential behavior with respect to linear spaces in the ambient space. The {\em $i$th polar class} of a variety $X\subset\PP^N$, denoted by $p_i(X)$, represents the Chow class of a codimension $i$ subvariety of $X$. Its degree, denoted by $\mu_i(X)$, provides valuable numerical information about the geometry of $X$. Notably, these invariants play a key role in applications such as Algebraic Statistics and Variety Sampling \cite{DHOST,dirocco2020bottleneck}.

In this paper, we lay the theoretical foundations of a relative notion of projective duality. Having fixed a projective subvariety $Z\subset X$, one may look for all hyperplanes tangent to $X$ on at least one point of $Z$. This defines the {\em dual variety of $X$ relative to $Z$}, denoted by $X_Z^\vee$. Similarly, we work with a relative version of polar classes, which we call {\em relative polar classes}, denoted by $p_i(X,Z)$. Their degrees, in this paper denoted by $\mu_i(X,Z)$, were originally introduced in \cite{piene1978polar} under the name of {\em classes of immersion}.
This paper aims to establish a relative tangency theory for nongeneric subvarieties, proving that most natural generalizations still hold.
A class of varieties for which this relative point of view is particularly central when it comes to applications is the class of determinantal varieties, explored in Section \ref{sec: relative tang determinantal}.

\begin{example}\label{ex: conics tangent to fixed conic}
Consider the space $\PP^5=\PP(S^2\C^3)$, the parameter space for conics in $\PP^2$ or equivalently of $3\times 3$ symmetric matrices. Let $X$ be the degree $4$ surface in $\PP^5$  representing symmetric matrices of rank at most one, i.e., the Veronese embedding of $\PP^2$. Geometrically, $X$ encodes non-reduced conics, counting lines with multiplicity two.
The dual variety $X^\vee\subset (\PP^5)^\vee$ parametrizes all $3\times 3$ symmetric matrices with determinant zero: these correspond to singular plane conics, or conics given by the union of two lines.
Consider a smooth conic $C\subset\PP^2$ and its Veronese embedding $Z$, a quartic curve in $X$. The subvariety $Z$ describes lines with multiplicity two, tangential to $C$. The dual variety of $Z$ is another hypersurface in $\PP^5$, representing plane conics tangent to $C$. If we consider only hyperplanes tangent to $X$ at points living in $Z$, they form the relative dual variety $X_Z^\vee$. This variety is contained in $X^\vee\cap Z^\vee$ and characterizes conics resulting from two lines intersecting at a point on $C$. Furthermore, it has codimension $2$ in $\PP^5$. We show in Remark \ref{rmk: relative dual strictly contained in intersection X dual Z dual} that $X_Z^\vee$ is strictly contained in $X^\vee\cap Z^\vee$.\hfill$\diamondsuit$
\end{example}

The primary motivation for introducing a relative tangency theory comes again from optimization.
In several applications, one seeks to optimize a polynomial objective function over a model described by polynomial equalities \cite{nie2009algebraic}. One of the most classic examples is the minimization of the Euclidean distance $d_u(x)=(x_0-u_0)^2+\cdots+(x_N-u_N)^2$ from a given data point $u=(u_0,\ldots,u_N)$, where $x=(x_0,\ldots,x_N)$ is constrained to an algebraic variety $X$.
The minimizers of $d_{X,u}\coloneqq d_u|_X$ sit among the subset of critical points $\mathrm{Crit}(d_{X,u})$, see Definition \ref{def: ED degree}. The cardinality $|\mathrm{Crit}(d_{X,u})|$ is finite and constant on a dense open subset of data points $u$ and is known as the {\em Euclidean Distance (ED) degree} of $X$.
It might happen, however, that the relevant area of focus is a subset of the given variety $X$.
In such situations, limiting the analysis to that particular subvariety of interest $Z\subset X$, rather than considering the entire variety, can significantly improve the computational efficiency of the problem. This motivates the introduction of ``conditional data invariants''.
If $u$ is sufficiently general, one expects to find no critical points of $d_{X,u}$ lying on $Z$. Indeed, there is a proper subset of data points possessing at least one critical point on $Z$. This is, by definition, the {\em ED data locus} of $X$ {\em given} $Z$, and it is denoted by $\mathrm{DL}_{X|Z}$, see Definition \ref{def: ED data locus}. This definition is slightly different than the definition of ED data locus given in  \cite{horobet2022data}, and coincides with it when $X$ is a smooth variety. In this paper, we investigate more on dimensions and degrees of ED data loci, connecting the tools of relative tangency introduced above.

\begin{example}\label{ex: Cayley intro}
Consider the Cayley's nodal cubic surface $X\subset\R^3$ of equation
\begin{equation}\label{eq: Cayley}
f(x_1,x_2,x_3)=16\,x_{1}x_{2}x_{3}+4(x_{1}^{2}+x_{2}^{2}+x_{3}^{2})-1=0\,.
\end{equation}
The variety $X$ has four non-coplanar singular points, highlighed in red in Figures \ref{fig: ED data locus Cayley circle} and \ref{fig: ED data locus Cayley line}.

First, we consider the circle $Z$ obtained by intersecting $X$ with the plane $x_1=0$. The ED data locus $\mathrm{DL}_{X|Z}$ is the sextic surface shown in green in Figure \ref{fig: ED data locus Cayley circle} together with the normal spaces $N_xX$ at the points $x\in Z$. Its equations is given in \eqref{eq: ED data locus Cayley circle}, and it coincides with the data locus studied in \cite{horobet2022data}.
Instead, consider the line $Z\subset X$ that joins the singular points $P_1=(1/2,1/2,-1/2)$ and $P_2=(-1/2,-1/2,-1/2)$ of $X$. On one hand, the ED data locus $\mathrm{DL}_{X|Z}$ is the plane $u_1-u_2=0$, shown in green on the left of Figure \ref{fig: ED data locus Cayley line}. On the other hand, the ED data locus studied in \cite{horobet2022data} consists of three irreducible components: the variety $\mathrm{DL}_{X|Z}$ and the two normal cones of $X$ at the two singularities lying on $Z$, depicted in orange in Figure \ref{fig: ED data locus Cayley line}. Their equations are given in \eqref{eq: components ED data locus Cayley}.
More details are given in Example \ref{ex: Cayley}.\hfill$\diamondsuit$
\end{example}
\begin{figure}[ht]
\begin{overpic}[width=2.2in]{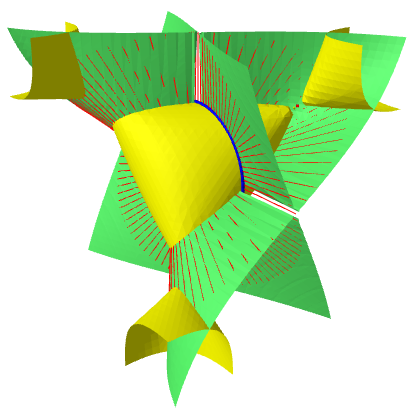}
\put (40,60) {{\scriptsize $X$}}
\put (53,58) {{\scriptsize\blue{$Z$}}}
\put (15,40) {{\small $\mathrm{DL}_{X|Z}$}}
\end{overpic}
\caption{The Cayley's nodal cubic surface $X$ (in yellow) and the green ED data locus $\mathrm{DL}_{X|Z}$ with respect to the blue circle $Z\subset X$.}\label{fig: ED data locus Cayley circle}
\end{figure}
\begin{figure}[ht]
\begin{overpic}[width=2.1in]{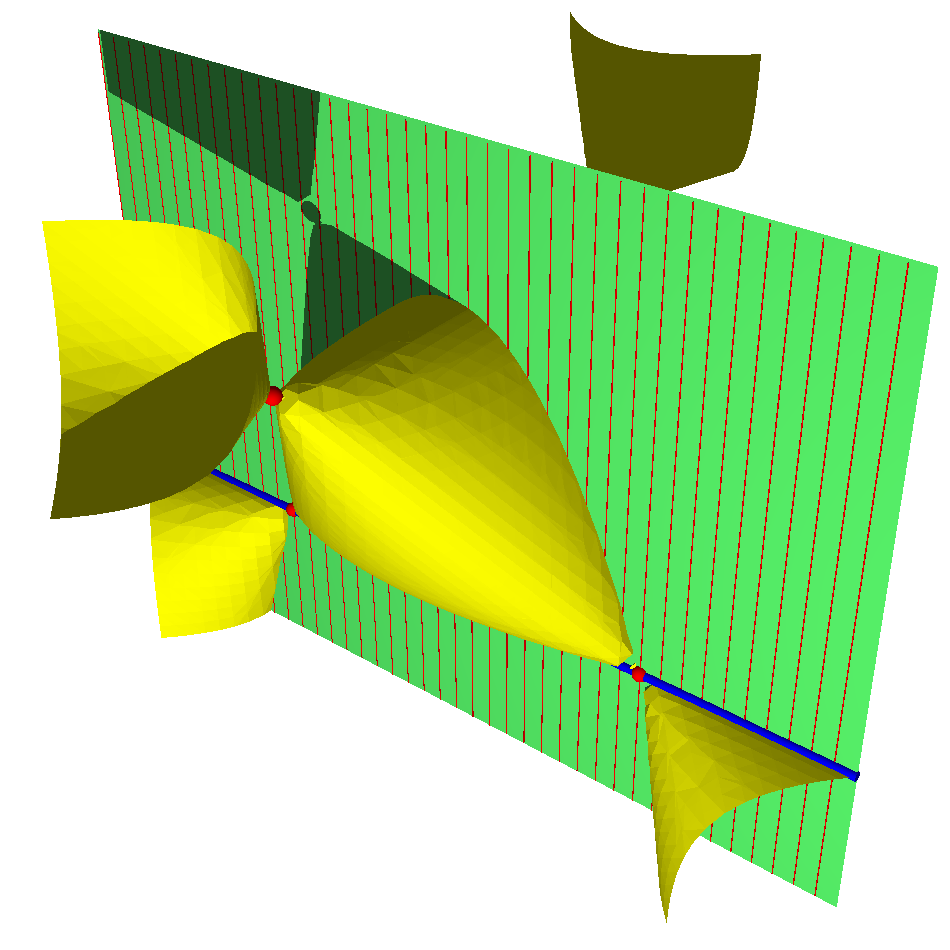}
\put (40,50) {{\scriptsize $X$}}
\put (80,25) {{\scriptsize\blue{$Z$}}}
\put (75,67) {{\small $\mathrm{DL}_{X|Z}$}}
\put (60,24) {{\scriptsize\red{$P_2$}}}
\put (25,41) {{\scriptsize\red{$P_1$}}}
\end{overpic}
\begin{overpic}[width=2.1in]{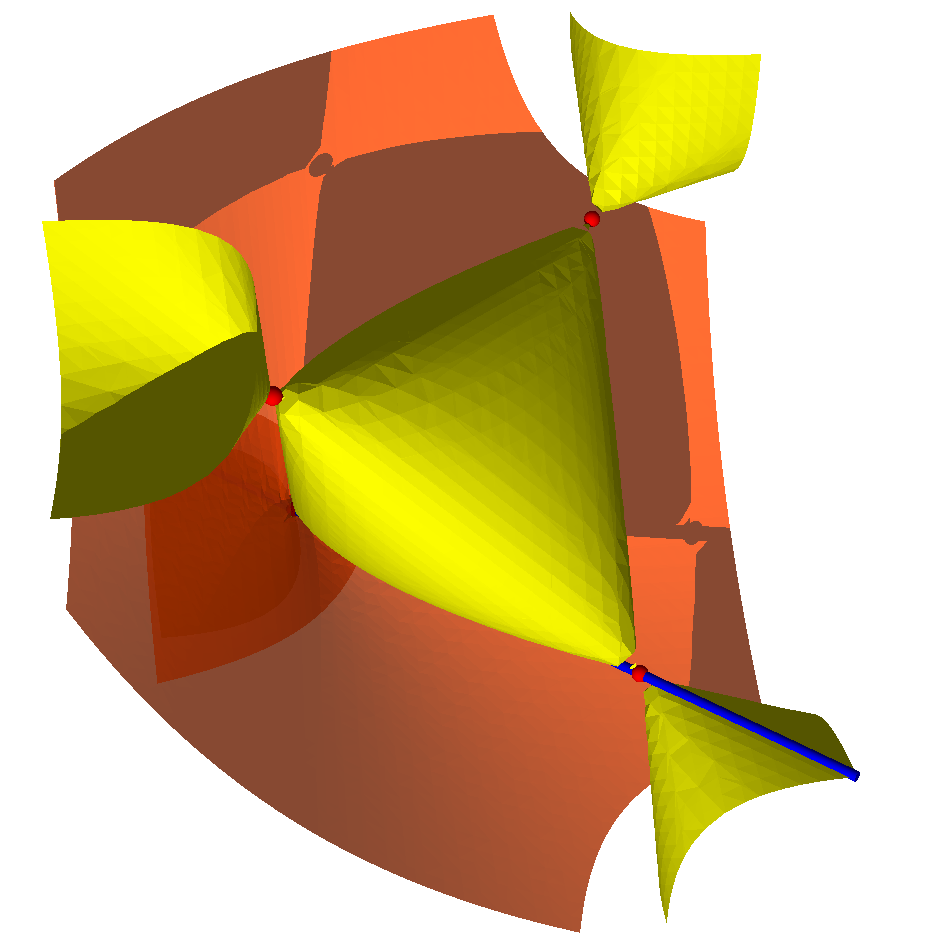}
\put (40,50) {{\scriptsize $X$}}
\put (82,26) {{\scriptsize\blue{$Z$}}}
\put (60,24) {{\scriptsize\red{$P_2$}}}
\put (25,41) {{\scriptsize\red{$P_1$}}}
\end{overpic}
\begin{overpic}[width=2.1in]{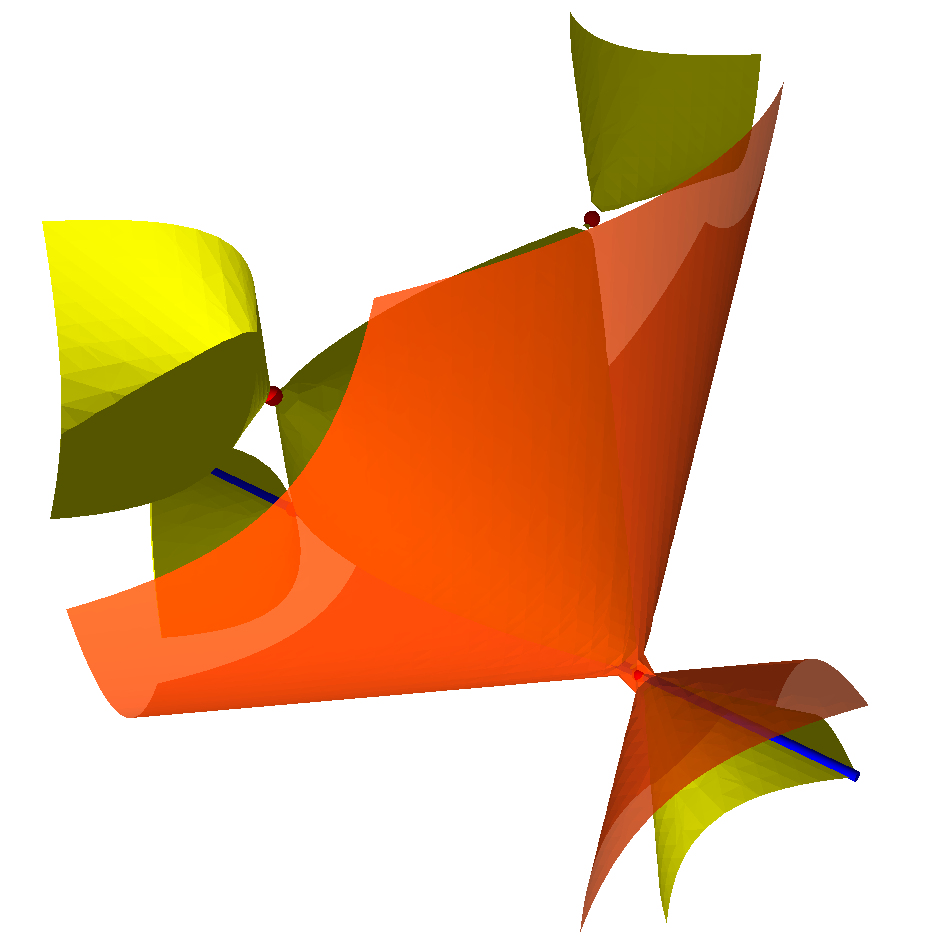}
\put (10,66) {{\scriptsize $X$}}
\put (85,12) {{\scriptsize\blue{$Z$}}}
\put (60,24) {{\scriptsize\red{$P_2$}}}
\put (25,41) {{\scriptsize\red{$P_1$}}}
\end{overpic}
\caption{On the left, in green the ED data locus $\mathrm{DL}_{X|Z}$ given the line $Z\subset X$ joining the singular points $P_1$ and $P_2$ of $X$.
In the center and on the right, in orange the normal cones of $X$ at $P_1$ and $P_2$.}\label{fig: ED data locus Cayley line}
\end{figure}
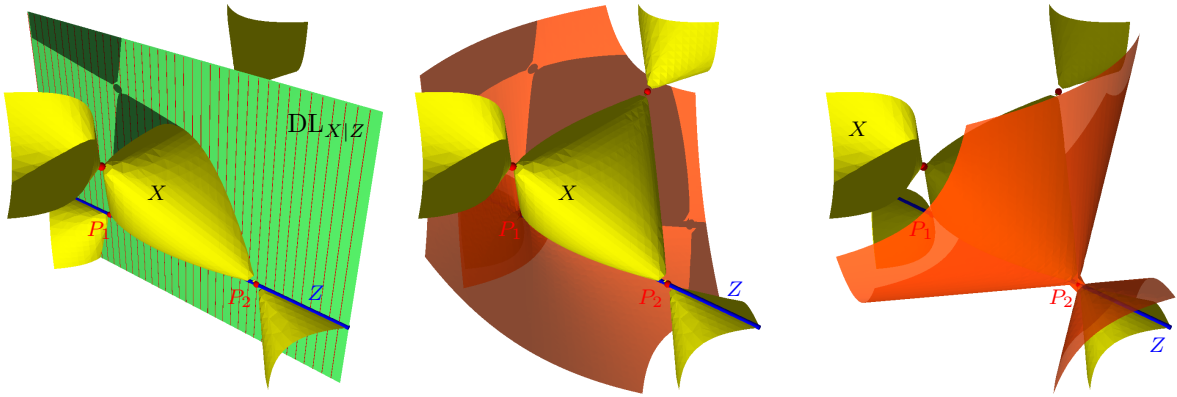

Furthermore, in our paper we consider the following question: {\em how many of the ED critical points of $d_{X,u}$ belong to $Z$?} Equivalently, {\em what is the ED degree of $X$ conditional to $Z$?}
These questions are also motivated by the following observation. Suppose that $X$ is covered by the family of nonempty subsets $\sZ=\{Z_i\mid i\in I\}$, and let $\mathrm{Crit}(d_{X,u}|Z_i)$ be the subset of $\mathrm{Crit}(d_{X,u})$ of critical points that belong to $Z_i$ (or conditional to $Z_i$).
By the inclusion-exclusion principle,
\begin{equation}\label{eq: inclusion-exclusion}
|\mathrm{Crit}(d_{X,u})|=\sum_{i\in I}|\mathrm{Crit}(d_{X,u}|Z_i)|-\sum_{Z_i\cap Z_j\neq\emptyset}|\mathrm{Crit}(d_{X,u}|Z_i\cap Z_j)|+\cdots\,.
\end{equation}
When the elements of $\sZ$ are pairwise disjoint, then $\sZ$ is a partition of $X$ and only the first sum in \eqref{eq: inclusion-exclusion} is nonzero. Assuming that all $Z_i$ are algebraic varieties, under a suitable condition of ED regularity (see Definition \ref{def: conditional ED regularity}), the cardinality $\mathrm{Crit}(d_{X,u}|Z)$ is finite and constant on a dense open subset of data points $u\in\mathrm{DL}_{X|Z}$. We denote it by $\mathrm{EDD}(X|Z)$ we refer to it as the {\em conditional ED degree} of $X$ {\em given $Z$}.

ED data loci and conditional ED degrees are closely related to further relevant topics in {\em metric algebraic geometry}, such as {\em Voronoi cells} \cite{cifuentes2022voronoi}, {\em medial axis}, {\em reach} and {\em bottlenecks} \cite{dirocco2020bottleneck}. We point out that, in this work, we focus on ED data loci and conditional ED degrees with respect to the Euclidean distance function, but in principle our results may be rephrased for other optimization problems, for instance considering the log-likelihood function \cite{catanese2006maximum,sturmfels2010multivariate,amendola2021maximum}, the $p$-norm \cite{kubjas2021algebraic} or the Wasserstein distance \cite{celik2021wasserstein,depaul2024degrees,meroni2024algebraic}.

In Section \ref{sec: preliminaries}, we outline the main definitions and theorems in classical tangency theory, following the references \cite{piene1978polar,kleiman1986tangency,holme1988geometric,gelfand1994discriminants}. Given a projective variety $X\subset\PP^N$, the conormal variety $W_X$ and the dual variety $X^\vee$ are recalled in Definitions \ref{def: conormal variety} and \ref{def: dual variety X}, respectively. The variety $X$ is reflexive when $W_X=W_{X^\vee}$, and is reciprocal when $(X^\vee)^\vee=X$. Reflexivity implies reciprocity in general, and over an algebraically closed field of characteristic zero, the two notions are equivalent. We further recall the connection between the rational equivalence class $[W_X]$ in the Chow ring of $\PP^N\times(\PP^N)^\vee$ and the degrees $\mu_i(X)$ of the polar classes of order $i$ of $X$.

In Section \ref{sec: relative polar classes of a projective variety}, we develop the concepts of relative duality and polarity. In Definition \ref{def: relative conormal variety} we introduce the conormal variety $W_{X,Z}$ relative to the subvariety $Z\subset X$, and the relative dual variety $X_Z^\vee$. We say that $X$ is reflexive relative to $Z$ if $W_{X,Z}=W_{X^\vee,X_Z^\vee}$, and that $X$ is reciprocal relative to $Z$ if $(X^\vee)_{X_Z^\vee}^\vee=Z$. Under a mild condition of {\em dual regularity} given in Definition \ref{def: dual regular relative Z}, we characterize relative reflexivity and reciprocity in terms of the defects $\mathrm{def}(X)$ and $\mathrm{def}(X,Z)$, which are the dimensions of the contact locus in $X$ (respectively, in $Z$) of a generic tangent hyperplane to $X$.

Let $Z\subset X$ be irreducible projective varieties defined over an algebraic closed field of characteristic zero.
\begin{theorem}\label{thm: conditions equivalent to relative biduality}
Assume that $X$ is dual regular relative to $Z$. The following properties are equivalent:
\begin{enumerate}
    \item $\mathrm{def}(X,Z)=\mathrm{def}(X)$.
    \item $X$ is reflexive relative to $Z$.
    \item $X$ is reciprocal relative to $Z$.
\end{enumerate}
\end{theorem}
Next, in Definition \ref{def: relative polar variety} we introduce relative polar varieties of order $i$ and their degrees $\mu_i(X,Z)$. We prove a relative version of the celebrated Kleiman's duality formula relating the invariants $\mu_i(X,Z)$ to the multidegrees $\delta_i(X,Z)$ of the rational equivalence class $[W_{X,Z}]$.
\begin{proposition}\label{prop: properties relative polar ranks}
Let $X\subset\PP^N$ be a smooth projective variety of dimension $n$, and $Z$ a smooth subvariety of $X$ of dimension $d$.
For all $0\le i\le d$, we have $\mu_i(X,Z)=\delta_{d-i}(X,Z)$.
\end{proposition}

In Section \ref{sec: relative tang determinantal}, we investigate relative duality and polarity of determinantal varieties, also known as varieties of low-rank matrices, relative to two special subvarieties: first, the subvariety of low-rank matrices with prescribed column and row spaces (see Proposition \ref{prop: dim relative dual Z X_r}). Then, the subvariety of low-rank symmetric matrices (see Proposition \ref{prop: relative dual determinantal wrt subspace symmetric}).

In Sections \ref{sec: affine ED data loci} and \ref{sec: projective} we provide the geometric foundations of ED data loci and conditional ED degrees, both for affine and projective varieties. The geometric foundations of these notions rely on the {\em conditional ED correspondence} $\sE_{X|Z}$ of a variety $X$ given a subvariety $Z$, see Definition \ref{def: ED correspondence Z}. In Proposition \ref{prop: rationality ED data loci} we show how to parametrize the relative ED correspondence, provided that $X$ and $Z$ are rational varieties.

In Section \ref{sec: degrees ED data loci} we characterize $\mathrm{EDD}(X|Z)$ and $\mathrm{DL}_{X|Z}$ in terms of characteristic classes. Theorem \ref{thm: degree data locus no genericity X, Z} shows that the polar class formula in \cite[Theorem 5.4]{DHOST} has a conditional counterpart.
\begin{theorem}\label{thm: degree data locus no genericity X, Z}
Let $Z$ be a projective subvariety of $X$ of dimension $d$, such that $X$ is ED regular given $Z$. If $\sPN_{X,Z}$ is disjoint from the diagonal $\Delta(\PP^N)$ of $\PP^N\times\PP^N$, then
\begin{equation}\label{eq: degree data locus no genericity X, Z}
\mathrm{EDD}(X|Z)\deg(\mathrm{DL}_{X|Z})=\sum_{i=0}^d\delta_i(X,Z)\,.
\end{equation}
\end{theorem}
The incidence variety $\sPN_{X,Z}$ is introduced in \eqref{eq: def PN X Z}.
In the remainder of the section, we also provide alternative formulas for the product on the left-hand side of \eqref{eq: degree data locus no genericity X, Z} involving Chern classes. When $X$ is a complete intersection variety, this product has an explicit formulation.
\begin{theorem}\label{thm: degree data loci complete intersection}
Let $X\subset\PP^N$ be a generic complete intersection of $s$ hypersurfaces $X_1,\ldots,X_s$ of degrees $d_1,\ldots,d_s$.
Let $Z\subset X$ be an irreducible smooth subvariety of dimension $d$ such that $X$ is ED regular given $Z$. Then
\begin{equation}\label{eq: degree data loci complete intersection}
\mathrm{EDD}(X|Z)\deg(\mathrm{DL}_{X|Z})= \deg(Z)\sum_{i_1+\cdots+i_s\le d}(d_1-1)^{i_1}\cdots(d_s-1)^{i_s}\,.
\end{equation}
\end{theorem}

Finally in Section \ref{sec: multiple}, we investigate multiple ED data loci, namely those sets of data points admitting more than $\mathrm{EDD}(X|Z)$ critical points on $Z$. We describe their relationship with singular loci of ED data loci, and we use these properties to determine when $\mathrm{EDD}(X|Z)=1$.

All figures in this paper have been realized using the software \verb|Sage| \cite{sagemath}. The \verb|Macaulay2| \cite{GS} code used in the computations of this work can be found at 
\begin{center}
    \url{https://github.com/sodomal1/conditional-ED-Optimization.git}
\end{center}

\section*{Acknowledgements}
We thank Emil Horobe\c{t}, Antonio Lerario, Giorgio Ottaviani, Ragni Piene, Jose Rodriguez, for useful discussions.
The first and second author are supported by VR grant [NT:2018-03688]. A KTH grant from the Verg Foundation and Brummer \& Partners MathDataLab currently supports the third author.

\section{Preliminaries on duality and polar classes}\label{sec: preliminaries}

Let $\K$ be an algebraically closed field. 
Let $V$ be an $(N+1)$-dimensional vector space over $\K$. We denote by $\PP^N$ the projective space $\PP(V)$, while $(\PP^N)^\vee$ denotes the dual projective space of hyperplanes in $\PP^N$.

We fix a reduced and irreducible algebraic variety $X\hookrightarrow\PP^N$ of dimension $n$.
In the following, we denote by $X_{\sm}$ the dense subvariety of smooth points of $X$. For a smooth point $x\in X_{\sm}$, we let $T_xX$ be the projective tangent space of $X$ at $x$. When $X$ is smooth, then $\sT_X$ is the tangent bundle of $X$, whose fibers are $T_xX$. Furthermore, we denote with $\sN_{X/\PP^N}$ the {\em normal bundle} of $X\hookrightarrow\PP^N$, that is the quotient bundle $(\sT_{\PP^N}|_X)/\sT_X$ whose fibers are $T_x\PP^N/T_xX$. Its dual $\sN_{X/\PP^N}^\vee$ is the {\em conormal bundle} of $X$ in $\PP^N$.

\begin{definition}\label{def: conormal variety}
We mainly follow the notations in \cite[Chapter 1]{gelfand1994discriminants}. If $x\in X_{\sm}$ and $H\in(\PP^N)^\vee$ are such that $H\supset T_xX$, we say that $H$ is {\em tangent to $X$ at $x$}. Define
\begin{equation}\label{eq: def conormal variety}
W_X^{\circ}\coloneqq\{(x,H)\in\PP^N\times(\PP^N)^\vee\mid\text{$x\in X_{\sm}$ and $H$ is tangent to $X$ at $x$}\}\,.
\end{equation}
The {\em conormal variety} of $X\hookrightarrow\PP^N$ is the Zariski closure $W_X\coloneqq\overline{W_X^{\circ}}\subset\PP^N\times(\PP^N)^\vee$.
\end{definition}

Let $\pr_1$ and $\pr_2$ be the projections from $\PP^N\times(\PP^N)^\vee$ to the factors $\PP^N$ and $(\PP^N)^\vee$, respectively.
Note that $\pr_1$ gives $W_X^{\circ}$ the structure of a projective bundle on $X_{\sm}$ of rank $N-n-1$. In particular, if $X$ is smooth, then $W_X^{\circ}=W_X=\PP(\sN_{X/\PP^N}^\vee)$. It follows that
\begin{equation}\label{eq: dim conormal}
\dim(W_X)=N-1\,.    
\end{equation}
\begin{definition}\label{def: dual variety X}
The {\em dual variety} of $X\hookrightarrow\PP^N$ is $X^\vee\coloneqq \pr_2(W_X)$.
\end{definition}

Consider the surjective morphisms $\pi_1\colon W_X\to X$ and $\pi_2\colon W_X\to X^\vee$ induced by the projections $\pr_1$ and $\pr_2$. More precisely $\pi_i=\pr_i\circ j$ for $i=1,2$, where $j\colon W_X\hookrightarrow\PP^N\times(\PP^N)^\vee$ is the inclusion.
 Notice that, for a generic choice of $H\in X^\vee$, the dimension of the {\it contact locus}
\begin{equation}\label{eq: def contact locus}
    \mathrm{Cont}(H,X)\coloneqq\pi_1(\pi_2^{-1}(H))\cong\overline{\{x\in X_{\sm}\mid H\supset T_xX\}}
\end{equation}
is equal to the dimension of a generic fiber of $\pi_2$. From \eqref{eq: dim conormal} it follows that 
\[
\text{$\dim(X^\vee)=N-1-\dim(\mathrm{Cont}(H,X))$ for a generic hyperplane $H\in X^\vee$.}
\]
Generally, one expects that $X^\vee$ is a hypersurface or equivalently that the generic contact locus is zero-dimensional, otherwise the embedding is called dual-defective.
\begin{definition}\label{def: dual defect}
The {\em dual defect} of $X$ is defined as
\begin{equation}
\mathrm{def}(X)\coloneqq\dim(\mathrm{Cont}(H,X))\text{ for a generic }H\in(\PP^N)^\vee\,.    
\end{equation}
A variety $X$ is said to be {\em dual defective} if $\mathrm{def}(X)>0$. Otherwise, it is {\em dual non-defective}.
\end{definition}

The name ``dual defect'' is motivated by the identity $\dim(X^\vee)=N-1-\mathrm{def}(X)$. In the following, we consider the dual correspondence between a linear space $L\subset\PP^N$ of dimension $k$ and its dual linear space $L^\vee\subset(\PP^N)^\vee$ of codimension $k+1$. We say that $X$ is {\em reflexive} if $W_{X^\vee}=W_X$. A classical result asserts that:

\begin{proposition}{\cite[Corollary 2.2]{holme1988geometric}}\label{prop: reflexivity contact locus linear}
A projective variety $X$ is reflexive if and only if the contact locus $\mathrm{Cont}(H,X)$ of a generic hyperplane $H$ is a linear subspace of $\PP^N$.
\end{proposition}
Reflexivity is  central in the theory of duality that holds in characteristic zero:

\begin{theorem}{\cite[Theorem 4]{kleiman1986tangency}}\label{thm: reflexivity}
Let $\K$ be a field of characteristic zero. Let $X\subset\PP^N$ be a reduced and irreducible algebraic variety. Then $X$ is reflexive. In particular, if $x$ is a smooth point of $X$ and $H$ is a smooth point of $X^\vee$, then $H$ is tangent at $x$ if and only if the hyperplane $x^\vee$ in $(\PP^N)^\vee$ is tangent to $X^\vee$ at the point $H$.
\end{theorem}

Furthermore, we say that $X$ is {\em reciprocal} (or that {\em biduality} holds for $X$) if $(X^\vee)^\vee=X$. If $X$ is reflexive, then it is also reciprocal. This is also known as the {\em Biduality Theorem}, see \cite[Theorem I.1.1]{gelfand1994discriminants}. Conversely, a reciprocal variety might not be reflexive if the ground field $\K$ has positive characteristic. See \cite[p. 168]{kleiman1986tangency} for a counterexample.
From now on we will assume that $\K$ has characteristic zero. We can conclude that:

\begin{corollary}\label{corol: contact locus linear}
The dual variety $X^\vee$ is a hypersurface if and only if the generic tangent hyperlane to $X$ is tangent at a single point.
\end{corollary}

Let $H$ and $H'$ be hyperplanes in $\PP^N$ and $(\PP^N)^\vee$, respectively, and let $[H]$ and $[H']$ be the corresponding classes in the Chow rings $A^*(\PP^N)$ and $A^*((\PP^N)^\vee)$. We use the shorthand
\begin{equation}\label{eq: notation hyperplane classes}
h=\pr_1^*([H])=[H\times(\PP^N)^\vee]\,,\quad h'=\pr_2^*([H'])=[\PP^N\times H']\,.
\end{equation}
Recalling \eqref{eq: dim conormal}, the class $[W_X]\in A^*(\PP^N\times(\PP^N)^\vee)$ can be written as
\begin{equation}\label{eq: def multidegrees conormal}
[W_X]=\delta_0(X)h^Nh'+\delta_1(X)h^{N-1}(h')^2+\cdots+\delta_{N-1}(X)h(h')^N\,.
\end{equation}
The numbers $\delta_i(X)$ are the multidegrees of $W_X$, or the classes of $X$. The class of $X$ always refers to the number $\delta_0(X)$.
Given $0\le i\le N-1$, let $L_1\subset\PP^N$ and $L_2\subset(\PP^N)^\vee$ be generic subspaces of dimensions $N-i$ and $i+1$, respectively. In particular $[L_1]=[H]^i$ and $[L_2]=[H']^{N-1-i}$. Then $\delta_i(X)$ is the cardinality of the finite intersection $W_X\cap(L_1\times L_2)$, or equivalently
\begin{equation}\label{eq: interpretation multidegrees delta_i}
    \delta_i(X)=\int \pr_1^*([L_1])\cdot \pr_2^*([L_2])\cdot[W_X]=\int h^i(h')^{N-1-i}[W_X]\,.
\end{equation}

We recall now the identification of the multidegrees $\delta_i(X)$ with the so-called polar degrees of $X$. Let $G(k+1,N+1)=\G(k,N)$ be the {\em Grassmannian} of $(k+1)$-dimensional vector subspaces of $V$, namely of $k$-dimensional projective subspaces of $\PP^N$. For a fixed projective subspace $L\subset\PP^N$ and a generic $M\in\G(k,N)$, one expects that $\dim(M\cap L)=k-\codim(L)$. Hence, it is natural to define the set of elements of $\G(k,N)$ that have a larger intersection with $L$:
\[
\Sigma_k(L)\coloneqq\{M\in \G(k,N)\mid\dim(M\cap L)>k-\codim(L)\}.
\] 
Actually, $\Sigma_k(L)$ is an example of a {\em Schubert variety}.
For example, the Grassmannian $\G(1,2)$ is the dual space $(\PP^2)^\vee$. Given a point $p\in\PP^2$, we have that
\[
\Sigma_1(p)=\{M\in(\PP^2)^\vee\mid\dim(M\cap \{p\})>-1\}=\{\text{lines in $\PP^2$ containing $p$}\}\,.
\]
The {\em Gauss map} associated with the variety $X$ is the rational map
\begin{equation}\label{eq: def Gauss map}
\gamma_X\colon X \dasharrow \G(n,N)\,,\quad x \mapsto T_xX
\end{equation}
which is defined over the smooth points of $X$.

\begin{definition}\label{def: polar variety}
The {\em polar variety} of an $n$-dimensional projective variety $X\subset\PP^N$ with respect to a subspace $L\subset\PP^N$ is
\begin{equation}
\sP(X,L) \coloneqq \overline{\gamma_X^{-1}(\Sigma_n(L))} = \overline{\{x\in X\mid\dim(T_xX\cap L)>n-\codim(L)\}}\,.
\end{equation}
\end{definition}
The class $[\sP(X,L)]\in A^*(X)$ does not depend on the generic choice of $L$, namely $[\sP(X,L)]$ depends only on the dimension of $L$. This motivates the following definition.
\begin{definition}\label{def: polar classes}
For every $0\le i\le n$, let $L\subset\PP^N$ be a generic subspace of dimension $N-n+i-2$. The {\em $i$th polar class of $X$} is $p_i(X)\coloneqq[\sP(X,L)]\in A^*(X)$. The {\em $i$th polar degree of $X$} is $\mu_i(X)\coloneqq\deg(p_i(X))$.
\end{definition}
For example, assume $i=0$ in Definition \ref{def: polar classes}, equivalently that $L\subset\PP^N$ is a generic subspace of dimension $N-n-2$. Then $\sP(X,L) = \overline{\{x\in X\mid\dim(T_xX\cap L)>-2\}}=X$.
This implies that $\mu_0(X)=\deg(X)$. Instead if $i=n+1$, then $L\subset\PP^N$ is a generic hyperplane and $\sP(X,L) = \overline{\{x\in X\mid\dim(T_xX\cap L)>n-1\}}= \overline{\{x\in X\mid T_xX\subset L\}}=\emptyset$ because $L$ is generic. More in general, the class $p_i(X)$ represents a subvariety of codimension $i$ in $X$. This motivates why the index $i$ ranges between $0$ and $n$ in Definition \ref{def: polar classes}.

\begin{example}
Let $X\subset\PP^2$ be a plane curve, so it admits two polar degrees $\mu_0(X)$ and $\mu_1(X)$. The $0$th polar degree of $X$ is $\mu_0(X)=\deg(X)$. Then $\mu_1(X)$ is the degree of $\sP(X,p)$ for a generic point $p\in\PP^2$. In this case $\sP(X,p)=\{x\in X\mid\dim(T_xX\cap \{p\})>-1\}=\{x\in X\mid p\in T_xX\}$ is equal to the degree of the dual variety $X^\vee$ of $X$.\hfill$\diamondsuit$
\end{example}

Assume that $X\hookrightarrow\PP^N$ is a smooth embedding given by the global sections of the very ample line bundle $\sO_X(1)$.   
We briefly recall the definition of the first jet bundle of the line bundle $\sO_X(1),$ and outline an important relation between this object and the polar classes of $X.$ 

Let $x\in X$ and  $\sigma\in H^0(X,\sO_X(1)).$ Recall that the global section $\sigma$ can be represented in a neighborhood of $x$ as a polynomial in local coordinates $(\xi_1,\ldots,\xi_n).$  The {\em first jet} of $\sigma$ at $x$ is the $(n+1)$-tuple:
\[
j_{1,x}(\sigma)\coloneqq\left(\sigma(x),\frac{\partial\sigma}{\partial\xi_1}(x),\ldots,\frac{\partial\sigma}{\partial\xi_n}(x)\right)\in H^0\left(X,\sO_X(1)\otimes\frac{\sO_{X,x}}{m_{x}^2}\right)\cong{\C^{n+1}}\,,
\]
where  $m_x$ denotes the maximal ideal at $x$. The vector space $H^0\left(X,\sO_X(1)\otimes\frac{\sO_{X,x}}{m_{x}^2}\right)$ is the fiber at $x$ of the vector bundle $\sP^1(\sO_X(1))$ of rank $n+1$ called the {\it first jet bundle} of the embedding or the principal part bundle.
Gluing together the maps $j_{1,x}$ yields a surjective morphism of vector bundles on $X$: $j_1\colon H^0(X,\sO_X(1))\otimes\sO_X\to\sP^1(\sO_X(1))$. 
Moreover, since $X$ is smooth, the Gauss map $\gamma_X$ is a morphism and for every $x\in X$ we can identify $\gamma_X(x)=T_xX$ with $\PP(\image(j_{1,x}))$. Consequently, we have the relations (see \cite[\S2]{piene2015polar})
\begin{equation}\label{eq: relations polar classes Chern classes first jet bundle X}
    p_i(X)=c_i(\sP^1(\sO_X(1)))\quad\forall\,i\in\{0,\ldots,n\}\,,
\end{equation}
where $c_i(\sP^1(\sO_X(1))$ is the $i$th Chern class of $\sP^1(\sO_X(1))$.
Furthermore, the kernel of $j_1$ is $\sN_{X/\PP^N}^\vee(1)=\sN_{X/\PP^N}^\vee\otimes\sO_X(1)$, a vector bundle of rank $N-n$. In particular, we obtain the short exact sequence
\begin{equation}\label{eq: first jet sequence}
0\to\sN_{X/\PP^N}^\vee(1) \to \sO_X^{\oplus(N+1)}\to\sP^1(\sO_X(1))\to 0\,.
\end{equation}

The next proposition outlines a few important properties of polar degrees. We refer to \cite{holme1988geometric} for more details.

\begin{proposition}\label{prop: properties polar degrees}
Let $X\subset\PP^N$ be an irreducible projective variety of dimension $n$, and let $\alpha(X)\coloneqq n-\mathrm{def}(X)$. Then
\begin{enumerate}
    \item $\mu_i(X)>0$ if and only if $0\le i\le\alpha(X)$
    \item $\mu_0(X)=\deg(X)$
    \item $\mu_{\alpha(X)}(X)=\deg(X^\vee)$
    \item $\mu_i(X)=\mu_{\alpha(X)-i}(X^\vee)$
\end{enumerate}
\end{proposition}

We conclude this section by recalling an important correspondence between the polar degrees of $X$ and the multidegrees of $W_X$. In Proposition \ref{prop: properties relative polar ranks}, we show a similar relation in the relative setting.

\begin{proposition}{\cite[Prop. (3), p. 187]{kleiman1986tangency}}\label{prop: multidegrees and polar classes Kleiman}
Let $X\subset\PP^N$ be a projective variety of dimension $n$. Then
\begin{equation}
    \mu_i(X)=\delta_{n-i}(X)\quad\forall\,0\le i\le n\,.
\end{equation}
In particular
\begin{enumerate}
    \item $\delta_i(X)>0$ if and only if $\mathrm{def}(X)\le i\le n$
    \item $\delta_{\mathrm{def}(X)}(X)=\deg(X^\vee)$
    \item $\delta_n(X)=\deg(X)$
    \item $\delta_i(X)=\delta_{n+\mathrm{def}(X)-i}(X^\vee)$
\end{enumerate}
\end{proposition}

\section{Relative duality and relative polar classes}\label{sec: relative polar classes of a projective variety}

Throughout the rest of the paper, $Z$ denotes a $d$-dimensional reduced and irreducible subvariety of $X\subset\PP^N$, hence $0\le d\le n\le N-1$.

\begin{definition}\label{def: relative conormal variety}
Define the incidence variety $W_{X,Z}^{\circ}\coloneqq W_X^{\circ}\cap[Z\times(\PP^N)^\vee]$, namely
\begin{align}\label{eq: def Con 0 relative}
W_{X,Z}^{\circ} = \{(x,H)\in\PP^N\times(\PP^N)^\vee\mid\text{$x\in X_{\sm}\cap Z$ and $H$ is tangent to $X$ at $x$}\}\,.
\end{align}
The {\em conormal variety} of $X\hookrightarrow\PP^N$ {\em relative to $Z$} is the Zariski closure $\overline{W_{X,Z}^{\circ}}\subset\PP^N\times(\PP^N)^\vee$. Its projection $X_Z^\vee\coloneqq\pr_2(W_{X,Z})$ is the {\em dual variety} of $X$ {\em relative to $Z$}.
\end{definition}

\begin{proposition}\label{prop: dimension relative conormal variety}
Let $X\subset\PP^N$ be a projective variety of dimension $n$, and let $Z\subset X$ be an irreducible projective subvariety of dimension $d$ not contained in $X_{\sing}$. Then $W_{X,Z}$ is irreducible of dimension $N-1-n+d$. Also $X_Z^\vee$ is irreducible.
\end{proposition}
\begin{proof}
The first projection $\pr_1$ gives $W_{X,Z}^{\circ}$ the structure of a projective bundle on $X_{\sm}\cap Z$ of rank $N-n-1$. This implies that its closure $W_{X,Z}$ is irreducible of dimension $\dim(\pr_1^{-1}(x))+\dim(Z)=N-1-n+d$. It follows immediately that $X_Z^\vee=\pr_2(W_{X,Z})$ is irreducible as well.
\end{proof}

We also consider the surjective morphisms $\pi_{1,Z}\colon W_{X,Z}\to Z$ and $\pi_{2,Z}\colon W_{X,Z}\to X_Z^\vee$ induced by $\pr_1$ and $\pr_2$, namely $\pi_{i,Z}=\pr_i\circ j_Z$, where $j_Z\colon W_{X,Z}\hookrightarrow\PP^N\times(\PP^N)^\vee$ is the inclusion.

Given $H\in X_Z^\vee$, the fiber $\mathrm{Cont}(H,X,Z)\coloneqq\pi_{1,Z}(\pi_{2,Z}^{-1}(H))$ is the {\em contact locus} of $H$ with $X$ {\em relative to} $Z$.
The {\em relative dual defect} of the nested pair $Z\subset X$ is $\mathrm{def}(X,Z)\coloneqq\dim(\mathrm{Cont}(H,X,Z))$ for a generic $H\in(\PP^N)^\vee$. A variety $X$ is {\em dual defective relative to a subvariety $Z$} if $\mathrm{def}(X,Z)>0$, otherwise it is {\em dual non-defective relative to $Z$}. The name is motivated by the identity $\dim(X_Z^\vee)=N-1-\codim_X(Z)-\mathrm{def}(X,Z)$.
The following are general properties relating $X^\vee$ and $X_Z^\vee$.

\begin{proposition}\label{prop: formulas relative codim X dual X Z dual}
Let $Z\subset X$ be two varieties such that $Z\not\subset X_{\sing}$. Then
\begin{enumerate}
    \item $W_{X,Z}=W_X\cap W_Z$.
    \item $X_Z^\vee\subset X^\vee\cap Z^\vee$.
    \item $\codim_{X^\vee}(X_Z^\vee)=\codim_{X}(Z)+\mathrm{def}(X,Z)-\mathrm{def}(X)$.
    \item $\codim_{X^\vee}(X_Z^\vee)=\codim_{X}(Z)$ if and only if $\mathrm{def}(X,Z)=\mathrm{def}(X)$.
\end{enumerate}
\end{proposition}
\begin{proof}
The inclusion $W_{X,Z}\supset W_X\cap W_Z$ in $(1)$ is immediate. The other inclusion follows by the inclusion $T_xZ\subset T_xX$ for all $x\in Z_{\sm}\cap X_{\sm}$, and by taking closures.
Furtheremore, by definition $X_Z^\vee = \pi_{2,Z}(W_{X,Z})$ and
\[
\pi_{2,Z}(W_{X,Z})=\pr_2(W_{X,Z})=\pr_2(W_X\cap W_Z)\subset \pr_2(W_X)\cap \pr_2(W_Z)\,,
\]
where the last quantity is exactly $X^\vee\cap Z^\vee$. This shows $(2)$.
Part $(3)$ follows from the identities
\begin{align*}
\dim(X^\vee) &= N-1-\mathrm{def}(X)\\
\dim(X_Z^\vee) &= N-1-\codim_{X}(Z)-\mathrm{def}(X,Z)\,.
\end{align*}
Property $(4)$ follows immediately by part $(3)$.
\end{proof}

\begin{remark}\label{rmk: relative dual strictly contained in intersection X dual Z dual}
The containment $X_Z^\vee\subset X^\vee\cap Z^\vee$ of Lemma \ref{prop: formulas relative codim X dual X Z dual}$(2)$ might be strict. Consider a fixed smooth conic $C\subset\PP^2$ and the varieties $Z\subset X$ of Example \ref{ex: conics tangent to fixed conic}. Let $Q$ be a conic in $\PP^2$, seen as a point in $X^\vee\cap Z^\vee\subset\PP(S^2\C^3)$. Since $Q\in X^\vee$, then $Q$ is singular, in particular $Q=L_1\cup L_2$ for some lines $L_1,L_2$. Moreover because $Q\in Z^\vee$, $Q$ is tangent to $C$. Define the subsets
\begin{align*}
Y_1 &\coloneqq\{Q\in\PP(S^2\C^3)\mid\text{$Q=L_1\cup L_2$ for some lines $L_1,L_2$ such that $L_1\cap L_2\in C$}\}\\
Y_2 &\coloneqq\{Q\in\PP(S^2\C^3)\mid\text{$Q=L_1\cup L_2$ for some lines $L_1,L_2$ such that $L_1=T_PC$ for some $P\in C$}\}
\end{align*}
Then $X^\vee\cap Z^\vee=Y_1\cup Y_2$. We now show that $Y_1\cup Y_2$  is an irreducible decomposition of $X^\vee\cap Z^\vee$. First of all, we have that $Y_1\not\subset Y_2$ and $Y_2\not\subset Y_1$. Define the maps
\begin{align*}
&f_1\colon (\PP^2\setminus\{P\})^{\times 2}\times C\to\PP(S^2\C^3)\,,\quad f_1(P_1,P_2,P)\coloneqq \overline{P_1P}\cup\overline{P_2P}\\
&f_2\colon (\PP^2)^\vee\times C\to\PP(S^2\C^3)\,,\quad f_2(L,P)\coloneqq L\cup T_PC\,.
\end{align*}
Then $Y_i=\image(f_i)$ for all $i\in\{1,2\}$. This implies that $Y_1$ and $Y_2$ are irreducible varieties. Using the maps $f_1,f_2$, we can also compute the dimensions of $Y_1,Y_2$. Consider the conic $Q=\overline{P_1P}\cup\overline{P_2P}\in Y_1$. Then $Q=\overline{P_1'P}\cup\overline{P_2'P}$ for any other pair $(P_1',P_2')\in\overline{P_1'P}\times \overline{P_2'P}$, with $P_i'\neq P$. This means that the fiber $f_1^{-1}(\overline{P_1P}\cup\overline{P_2P})$ is isomorphic to the the product $\PP^1\times\PP^1$, in particular $\dim(f_1^{-1}(Q))=2$ for all $Q\in Y_1$. Hence $\dim(Y_1)=\dim((\PP^2\setminus\{P\})^{\times 2}\times C)-\dim(f_1^{-1}(Q))=3$. Now suppose that $Q=L\cup T_PC\in Y_2$. In this case $f_2^{-1}(Q)=\{(L,P)\}$, namely $f_2$ is injective. This implies that $\dim(Y_2)=\dim((\PP^2)^\vee\times C)=3$. 

We have thus shown that $X^\vee\cap Z^\vee$ is the union of the two irreducible varieties $Y_1$ and $Y_2$ of dimension 3 in $\PP^5$. Now we show that $X_Z^\vee=Y_1$. This implies immediately that $X_Z^\vee\subsetneq X^\vee\cap Z^\vee$.

Recall that $X$ is the Veronese surface in $\PP^5$, image of the map $\nu_2\colon\PP^2\to\PP^5$ sending a point $P=[x_1,x_2,x_3]$ to $\nu(P)\coloneqq[x_1^2,x_1x_2,x_1x_3,x_2^2,x_2x_3,x_3^2]$.
The projectivized row span of the Jacobian matrix of $\nu$ defines the projective tangent space $T_{\nu(P)}X$:
\[
\begin{pmatrix}
	\frac{\partial\nu}{\partial x_1}\\[2pt]
	\frac{\partial\nu}{\partial x_2}\\[2pt]
	\frac{\partial\nu}{\partial x_3}
\end{pmatrix}
=
\begin{pmatrix}
	2\,x_{1} & x_{2} & x_{3} & 0 & 0 & 0\\
	0 & x_{1} & 0 & 2\,x_{2} & x_{3} & 0\\
	0 & 0& x_{1} & 0 & x_{2} & 2\,x_{3}
\end{pmatrix}\,.
\]
Up to a linear change of coordinates, we assume that $P=E_3=[0,0,1]$. Then $\nu(E_3)=[0,0,0,0,0,1]$ and $T_{\nu(E_3)}X$ is the projective subspace spanned by the rows of the matrix
\[
A=
\begin{pmatrix}
	0 & 0 & 1 & 0 & 0 & 0\\
	0 & 0 & 0 & 0 & 1 & 0\\
	0 & 0 & 0 & 0 & 0 & 2	
\end{pmatrix}\,.
\]
Let $\alpha=(\alpha_{200},\alpha_{110},\ldots,\alpha_{002})$ be the vector of coefficients of the hyperplane $H\subset\PP^5$. The preimage $\nu^{-1}(H\cap X)$ is the conic of equation $\alpha_{200}x_1^2+\alpha_{110}x_1x_2+\cdots+\alpha_{002}x_3^2=0$. We have that $T_{\nu(E_3)}X\subset H$ if and only if $\alpha^\mT\in\ker(A)$. This yields the conditions $\alpha_{101}=\alpha_{011}=\alpha_{002}=0$. Therefore $H$ induces the singular conic of equation $\alpha_{200}x_1^2+\alpha_{110}x_1x_2+\alpha_{020}x_2^2=0$, which is the union of two lines passing through $E_3$. More generally, we have that
\[
T_{\nu(P)}X\subset H \quad\Longleftrightarrow\quad\text{$\nu^{-1}(H\cap X) = L_1\cup L_2$ and $P\in L_1\cap L_2$}\,.
\]
It follows that the variety $X_Z^\vee$ corresponds to the variety of hyperplanes $H\subset\PP^5$ such that $\nu^{-1}(H\cap X)$ is a union of two lines passing through a point of $C$, or equivalently $X_Z^\vee=Y_1$.
\end{remark}

\begin{definition}\label{def: dual regular relative Z}
Let $Z\subset X$ be projective varieties. We say that $X$ is {\em dual regular relative to $Z$} if there exists a pair $(x,H)\in W_{X,Z}$ where $x\in X_{\sm}$ and $H\in(X^\vee)_{\sm}$. 
\end{definition}

Note that, if $X$ is dual regular relative to $Z$, then necessarily $Z\not\subset X_{\sing}$ and $X_Z^\vee\not\subset (X^\vee)_{\sing}$.

\begin{lemma}\label{lem: properties relative duality}
Let $Z\subset X$ be two varieties such that $X$ is dual regular relative to $Z$. Then
\begin{enumerate}
    \item $\mathrm{Cont}(H,X,Z)=\mathrm{Cont}(H,X)\cap Z$ for all $H\in X_Z^\vee\cap(X^\vee)_{\sm}$.
    \item $0\le\mathrm{def}(X,Z)\le\mathrm{def}(X)$.
    \item $0\le\codim_{X^\vee}(X_Z^\vee)\le\codim_{X}(Z)$.
\end{enumerate}
\end{lemma}
\begin{proof} 
$(1)$ By dual regularity of $X$ relative to $Z$, there exists $H\in X_Z^\vee\cap(X^\vee)_{\sm}$. Then the set $\{x\in\PP^N\mid\text{$x\in X_{\sm}$ and $H\supset T_xX$}\}$ is nonempty and open, hence dense in $\mathrm{Cont}(H,X)$. But then, since we are assuming $X_{\sm}\cap Z\neq\emptyset$, the set $\{x\in\PP^N\mid\text{$x\in X_{\sm}\cap Z$ and $H\supset T_xX$}\}$ is also nonempty and open, hence dense in $\mathrm{Cont}(H,X)\cap Z$, and
\[
\mathrm{Cont}(H,X,Z)=\pi_{1,Z}(\pi_{2,Z}^{-1}(H))=\overline{\{x\in\PP^N\mid\text{$x\in X_{\sm}\cap Z$ and $H\supset T_xX$}\}}\,,
\]
therefore we conclude that $\mathrm{Cont}(H,X,Z)=\mathrm{Cont}(H,X)\cap Z$.

$(2)$ By assumption $X_Z^\vee\cap(X^\vee)_{\sm}$ is a nonempty open dense subset of $X_Z^\vee$. By part $(1)$, for every $H\in X_Z^\vee\cap(X^\vee)_{\sm}$ we have that $\mathrm{Cont}(H,X,Z)=\mathrm{Cont}(H,X)\cap Z$, in particular $\dim(\mathrm{Cont}(H,X,Z))\le\dim(\mathrm{Cont}(H,X))$.
Since $H$ is smooth in $X^\vee$ (hence a generic tangent hyperplane to $X$), we have that $\dim(\mathrm{Cont}(H,X))=\mathrm{def}(X)$. Similarly $H$ is generic in $X_Z^\vee$, hence $\dim(\mathrm{Cont}(H,X,Z))=\mathrm{def}(X,Z)$. Therefore we conclude that $0\le\mathrm{def}(X,Z)\le\mathrm{def}(X)$.

Property $(3)$ is an immediate consequence of Proposition \ref{prop: formulas relative codim X dual X Z dual}$(3)$ and property $(2)$.
\end{proof}

In particular, Lemma \ref{lem: properties relative duality}$(1)$ tells us that, despite the generic contact loci $\mathrm{Cont}(H,X)$ being linear, relative contact loci are not necessarily linear spaces.
Furthermore, the following example shows that the assumption of relative dual regularity in Lemma \ref{lem: properties relative duality}$(2)$ cannot be dropped.

\begin{example}\label{ex: Cayley projective}
Consider the cubic surface in $X\subset\PP^3$ defined by the homogeneous polynomial
\begin{equation}\label{eq: Cayley projective}
f(x_1,x_2,x_3,x_4) = 16\,x_{1}x_{2}x_{3}+4(x_{1}^{2}+x_{2}^{2}+x_{3}^{2})x_4-x_4^3=0\,.
\end{equation}
This surface corresponds to the Zariski closure in $\PP^3$ of Cayley's cubic surface discussed in Example \ref{ex: Cayley}. The surface $X$ has four singular points. Consider the projective line $Z=V(x_1-x_2,2\,x_3+x_4)$ contained in $X$ and containing two of its singular points. We identify $(\PP^3)^\vee$ with $\PP^3$ via the quadratic form $q(x)=x_1^2+x_2^2+x_3^2+x_4^2$, see the discussion after Proposition \ref{prop: proj data locus is irreducible}.
On one hand, the dual variety of $X$ is the quartic surface
\begin{equation}
X^\vee=V(x_{1}^{2}x_{2}^{2}+x_{1}^{2}x_{3}^{2}+x_{2}^{2}x_{3}^{2}+4\,x_{1}x_{2}x_{3}x_{4})
\end{equation}
whose singular locus is a union of three lines:
\begin{equation}
(X^\vee)_{\sing}=V(x_{1},x_{2})\cup V(x_{1},x_{3})\cup V(x_{2},x_{3})\,.
\end{equation}
On the other hand, we verified symbolically that $X_Z^\vee$ is a point contained in $(X^\vee)_{\sing}$:
\begin{equation}
X_Z^\vee=V(x_{1},x_{2},x_{3}-2\,x_{4})\subset (X^\vee)_{\sing}\,.
\end{equation}
This example shows that $X_Z^\vee$ might be contained in $(X^\vee)_{\sing}$ even though $Z$ is not contained in $X_{\sing}$. Since $\codim_X(Z)=1$ and $\codim_{X^\vee}(X_Z^\vee)=2$, by Proposition \ref{prop: formulas relative codim X dual X Z dual}$(3)$ we conclude that $\mathrm{def}(X,Z)-\mathrm{def}(X)=1$, therefore $\mathrm{def}(X)=0$ and $\mathrm{def}(X,Z)=1$ because $X^\vee$ is a hypersurface.\hfill$\diamondsuit$
\end{example}

Now suppose that $Z$ is a subvariety of $X$ not contained in $X_{\sing}$. Then $X_Z^\vee$ is a nonempty subvariety of $X^\vee$, possibly equal to $X^\vee$. Hence, considering the dual variety of $X^\vee$ relative to $X_Z^\vee$ is natural. The next two lemmas fully describe this object.

\begin{lemma}\label{lem: Z contained in double relative dual}
Let $Z\subset X\subset\PP^N$ be irreducible varieties such that $X$ is dual regular relative to $Z$. Then $Z\subset(X^\vee)_{X_Z^\vee}^\vee\subset X$.
\end{lemma}
\begin{proof}
The hypotheses ensure that $X_{\sm}\cap Z$ and $X_Z^\vee\cap(X^\vee)_{\sm}$ are open dense subsets of $Z$ and $X_Z^\vee$, respectively.
Define
\begin{equation}\label{eq: open subset V X Z}
\sV_{X,Z}\coloneqq\{x\in X_{\sm}\mid\text{$x\supset T_H(X^\vee)$ for some $H\in X_Z^\vee\cap(X^\vee)_{\sm}$}\}\,.
\end{equation}
Let $x\in X_{\sm}\cap Z$, hence there exists a hyperplane $H$ that is tangent to $X$ at $x$, in particular $H\in X_Z^\vee\cap (X^\vee)_{\sm}$. By Theorem \ref{thm: reflexivity}, $x$ is tangent to $X^\vee$ at $H$. We have thus shown that $X_{\sm}\cap Z\subset\sV_{X,Z}$, and $\sV_{X,Z}\subset(X^\vee)_{X_Z^\vee}^\vee$. Hence, we conclude that $Z\subset(X^\vee)_{X_Z^\vee}^\vee$. The other inclusion $(X^\vee)_{X_Z^\vee}^\vee\subset X$ follows immediately by $(X^\vee)^\vee=X$.
\end{proof}

Using Theorem \ref{thm: reflexivity}, we can also write relative dual varieties by means of contact loci as
\begin{equation}
    X_Z^\vee=\overline{\{ H \in X^\vee_{\sm} \mid \mathrm{Cont}(H,X) \cap X_{\sm} \cap Z \neq \emptyset\}}\,.
\end{equation}
Next, we also describe the relative dual variety $(X^\vee)_{X_Z^\vee}^\vee$ in terms of the contact loci $\mathrm{Cont}(H,X)$.

\begin{lemma}\label{lem: identity double relative dual variety}
Let $Z\subset X\subset\PP^N$ be irreducible varieties such that $X$ is dual regular relative to $Z$. Then
\begin{equation}\label{eq: identity double relative dual variety}
(X^\vee)_{X_Z^\vee}^\vee = \overline{\bigcup_{\mathrm{Cont}(H,X)\cap Z\neq\emptyset}\mathrm{Cont}(H,X)}\,.
\end{equation}
\end{lemma}
\begin{proof}
We assume first that $Z\not\subset X_{\sing}$ and $X_Z^\vee\not\subset(X^\vee)_{\sing}$.
In particular, $X_{\sm}\cap Z$ and $X_Z^\vee\cap(X^\vee)_{\sm}$ are open dense subsets of $Z$ and $X_Z^\vee$, respectively.
Define
\begin{equation}\label{eq: open subset U X Z}
\sU_{X,Z}\coloneqq\{H\in(X^\vee)_{\sm}\mid\text{$H\supset T_xX$ for some $x\in X_{\sm}\cap Z$}\}\,.
\end{equation}
In particular $\overline{\sU_{X,Z}}=X_Z^\vee$. Note that here we use $X_Z^\vee\not\subset(X^\vee)_{\sing}$.
Consider also the set $\sV_{X,Z}$ introduced in \eqref{eq: open subset V X Z}. Since we are assuming $Z\not\subset X_{\sing}$, using the containment $Z\subset(X^\vee)_{X_Z^\vee}^\vee$ of Lemma \ref{lem: Z contained in double relative dual}, we conclude that $(X^\vee)_{X_Z^\vee}^\vee\not\subset X_{\sing}$, or equivalently that $(X^\vee)_{X_Z^\vee}^\vee\cap X_{\sm}$ is dense in $(X^\vee)_{X_Z^\vee}^\vee$, therefore $\overline{\sV_{X,Z}}=(X^\vee)_{X_Z^\vee}^\vee$.

Suppose that $x\in\sV_{X,Z}$, in particular $x\in X_{\sm}$ and $x$ is tangent to $X^\vee$ at some point $H\in X_Z^\vee\cap(X^\vee)_{\sm}$. Theorem \ref{thm: reflexivity} ensures also that $H$ is tangent to $X$ at $x$, namely $x\in\mathrm{Cont}(H,X)$.
Furthermore, since $H\in X_Z^\vee\cap(X^\vee)_{\sm}$, then $\emptyset\neq\mathrm{Cont}(H,X,Z)=\mathrm{Cont}(H,X)\cap Z$ by Lemma \ref{lem: properties relative duality}$(1)$.
We have thus shown that
\[
\sV_{X,Z} \subset \bigcup_{H\in X_Z^\vee\cap(X^\vee)_{\sm}}\mathrm{Cont}(H,X) \subset \bigcup_{\mathrm{Cont}(H,X)\cap Z\neq\emptyset}\mathrm{Cont}(H,X)\,.
\]
Therefore, by taking closures we obtain that
\begin{equation}\label{eq: first containment double relative dual variety}
(X^\vee)_{X_Z^\vee}^\vee \subset \overline{\bigcup_{\mathrm{Cont}(H,X)\cap Z\neq\emptyset}\mathrm{Cont}(H,X)}\,.
\end{equation}
Now suppose that $x\in\mathrm{Cont}(H,X)\cap X_{\sm}$ for some $H\in\sU_{X,Z}$, in particular $H\supset T_xX$. Furthermore, $H\in\sU_{X,Z}\subset  X_Z^\vee\cap(X^\vee)_{\sm}$ implies that $\mathrm{Cont}(H,X)\cap Z\neq\emptyset$. Again we can apply Theorem \ref{thm: reflexivity} and say that $x$ is tangent to $X^\vee$ at $H\in X_Z^\vee\cap(X^\vee)_{\sm}$, in particular $x\in\sV_{X,Z}$. We have shown that
\[
\left(\bigcup_{H\in\sU_{X,Z}}\mathrm{Cont}(H,X)\right)\cap X_{\sm} \subset \left(\bigcup_{\mathrm{Cont}(H,X)\cap Z\neq\emptyset}\mathrm{Cont}(H,X)\right)\cap X_{\sm} \subset \sV_{X,Z}\,.
\]
By taking closures, we get the reverse inclusion in \eqref{eq: first containment double relative dual variety}, hence the desired equality in \eqref{eq: identity double relative dual variety}.
Observe that in the proof we have also shown the identity
\begin{equation}\label{eq: identity double relative dual variety 2}
\overline{\bigcup_{\mathrm{Cont}(H,X)\cap Z\neq\emptyset}\mathrm{Cont}(H,X)}=\overline{\bigcup_{H\in X_Z^\vee\cap(X^\vee)_{\sm}}\mathrm{Cont}(H,X)}\,.
\end{equation}
Now if $Z\subset X_{\sing}$, then $X_Z^\vee=\emptyset$, so the right-hand side of \eqref{eq: identity double relative dual variety 2} is empty. But also $(X^\vee)_{X_Z^\vee}^\vee$ is empty.
Furthermore, if $X_Z^\vee\subset(X^\vee)_{\sing}$, then $(X^\vee)_{X_Z^\vee}^\vee$ is empty, but also $ X_Z^\vee\cap(X^\vee)_{\sm}=\emptyset$, hence the right-hand side of \eqref{eq: identity double relative dual variety 2} is empty. This means that \eqref{eq: identity double relative dual variety} is correct also when dropping the assumptions that $Z\not\subset X_{\sing}$ and $X_Z^\vee\not\subset(X^\vee)_{\sing}$.
\end{proof}

\begin{remark}\label{rmk: drop assumptions identity double relative dual variety}
The identity of Lemma \ref{lem: identity double relative dual variety} does not hold anymore if we drop the assumption of relative dual regularity.
For a counterexample, consider the surface $X\subset\PP^3$ defined by the equation
\[
c_{2}^{2}c_{3}^{2}-4\,c_{1}c_{3}^{3}-4\,c_{2}^{3}c_{4}+18\,c_{1}c_{2}c_{3}c_{4}-27\,c_{1}^{2}c_{4}^{2}=0\,,
\]
which is the discriminant locus of a cubic binary form $f(z,w)=c_1z^3+c_2z^2w+c_3zw^2+c_4w^3$.
The dual variety of $X$ is the rational normal cubic
\[
X^\vee=V(c_{3}^{2}-c_{2}c_{4},\,c_{2}c_{3}-c_{1}c_{4},\,c_{2}^{2}-c_{1}c_{3})\,.
\]
The previous three quadrics also define the radical ideal of $X^\vee$. Note also that $\mathrm{def}(X)=1$, so that $X$ is ruled by generic contact loci that are lines. If we let  $Z=X_{\sing}$, then at least one contact locus of $X$ meets $Z$. Hence, the right-hand side of \eqref{eq: identity double relative dual variety} is nonempty, but the left-hand side is.
\end{remark}

\begin{definition}\label{def: relative reflexive and reciprocal}
Let $Z\subset X\subset\PP^N$ be projective varieties.
Then $X$ is {\em reflexive relative to} $Z\subset X$ if $W_{X,Z}=W_{X^\vee,X_Z^\vee}$, while $X$ is {\em reciprocal relative to} $Z$ (or that {\em relative biduality} holds for the pair $Z\subset X$) when $Z=(X^\vee)_{X_Z^\vee}^\vee$.
\end{definition}

\begin{proof}[Proof of Theorem \ref{thm: conditions equivalent to relative biduality}]
$(1)\Rightarrow(2)$ Assume $\mathrm{def}(X,Z)=\mathrm{def}(X)$.
The hypotheses tell us that there exists an open dense subset $\sU\subset X_Z^\vee$ of points $H$ that are smooth for $X^\vee$ and $X_Z^\vee$, such that $\dim(\mathrm{Cont}(H,X,Z))=\dim(\mathrm{Cont}(H,X))$. Furthermore, similarly as in Lemma \ref{lem: properties relative duality}$(1)$ we also have that $\mathrm{Cont}(H,X,Z)=\mathrm{Cont}(H,X)\cap Z$ for all $H\in\sU$, then necessarily $\mathrm{Cont}(H,X,Z)=\mathrm{Cont}(H,X)$ for all $H\in\sU$.
Recall that $W_{X,Z}\subset\PP^N\times(\PP^N)^\vee$ and that $W_{X^\vee,X_Z^\vee}\subset(\PP^N)^\vee\times((\PP^N)^\vee)^\vee=(\PP^N)^\vee\times\PP^N$. Let $p$ and $q$ be the projections of $W_{X,Z}$ and $W_{X^\vee,X_Z^\vee}$ onto $(\PP^N)^\vee$, respectively. In particular $p(W_{X,Z})=q(W_{X^\vee,X_Z^\vee})=X_Z^\vee$.
On one hand, given $H\in\sU$, we have that
\[
p^{-1}(H)=\overline{\{x\in X_{\sm}\cap Z\mid H\supset T_xX\}}\times\{H\}=\mathrm{Cont}(H,X,Z)\times\{H\}\,.
\]
On the other hand, $q^{-1}(H)=\{H\}\times\overline{\sV_H}\subset\{H\}\times Z$, where
\[
\sV_H=\{x\in X_{\sm}\mid x\supset T_H(X^\vee)\}=\{x\in X_{\sm}\mid H\supset T_xX\}\,,
\]
and in the last identity, we have used Theorem \ref{thm: reflexivity}. But this implies that $\overline{\sV_H}=\mathrm{Cont}(H,X)$. Therefore $\mathrm{Cont}(H,X,Z)=\mathrm{Cont}(H,X)$ for all $H\in\sU$ implies that $p^{-1}(H)=q^{-1}(H)$ for all $H\in\sU$. In conclusion
\[
W_{X,Z}=\overline{\bigcup_{H\in\sU}p
^{-1}(H)}=\overline{\bigcup_{H\in\sU}q^
{-1}(H)}=W_{X^\vee,X_Z^\vee}\,.
\]
$(2)\Rightarrow(3)$ If $W_{X,Z}=W_{X^\vee,X_Z^\vee}$, they projections onto $\PP^N$ coincide as well, and they are equal respectively to $Z$ and $(X^\vee)_{X_Z^\vee}^\vee$.

$(3)\Rightarrow(1)$
By Lemma \ref{lem: Z contained in double relative dual}, we know that $Z\subset(X^\vee)_{X_Z^\vee}^\vee$. Using Lemma \ref{lem: identity double relative dual variety}, the identity $Z=(X^\vee)_{X_Z^\vee}^\vee$ implies that $\mathrm{Cont}(H,X)\subset Z$ for all $H\in X_Z^\vee\cap(X^\vee)_{\sm}$. Using  Lemma \ref{lem: properties relative duality}$(1)$, we conclude that $\mathrm{Cont}(H,X)=\mathrm{Cont}(H,X,Z)$ for all $H\in X_Z^\vee\cap(X^\vee)_{\sm}$. But this implies that $\mathrm{def}(X,Z)=\mathrm{def}(X)$.
\end{proof}

\begin{corollary}[Relative biduality]\label{corol: relative biduality}
Let $Z\subset X\subset\PP^N$ be irreducible varieties such that $X$ is dual regular relative to $Z$. Then
\begin{enumerate}
    \item If $\mathrm{def}(X)=0$, then $X$ is reciprocal relative to $Z$.
    \item If $X_Z^\vee=X^\vee$, then $(X^\vee)_{X_Z^\vee}^\vee=(X^\vee)^\vee=X$.
\end{enumerate}
\end{corollary}
\begin{proof}
$(1)$ Suppose that $\mathrm{def}(X)=0$. Lemma \ref{lem: properties relative duality}$(2)$ tells us that $0\le\mathrm{def}(X,Z)\le\mathrm{def}(X)$, hence $\mathrm{def}(X,Z)=\mathrm{def}(X)=0$. Then $Z=(X^\vee)_{X_Z^\vee}^\vee$ by Theorem \ref{thm: conditions equivalent to relative biduality}.
The implication $(2)$ is immediate.\qedhere
\end{proof}

\begin{remark}\label{rmk: implications cannot be reversed}
Both implications in Corollary \ref{corol: relative biduality} cannot be reversed:
\begin{enumerate}
    \item[$(i)$] A counterexample for the other implication in Corollary \ref{corol: relative biduality}(1) is the following. Consider $X$ to be the Segre product $\PP^1\times\PP^2\subset\PP^5$. Let $Z=\{P\}\times\PP^2$ for some point $P\in\PP^1$. In this case $X_Z^\vee$ is also of the form $\{P'\}\times\PP^2$ for some point $P'\in\PP^1$, and $(X^\vee)_{X_Z^\vee}^\vee=Z$, but $X^\vee=X$, hence $\mathrm{def}(X)=1$.
    \item[$(ii)$] About Corollary \ref{corol: relative biduality}(2), there exist varieties $Z\subset X$ such that $X$ is dual regular relative to $Z$, $(X^\vee)_{X_Z^\vee}^\vee=X$ but with $X_Z^\vee\subsetneq X^\vee$. For example, let $X$ be as in $(i)$, but this time choose $Z=\PP^1\times\{P\}$ for some $P\in\PP^2$. On one hand, we have $\codim_{X^\vee}(X_Z^\vee)=\codim_{X}(X_Z^\vee)=1$ and $\deg(X_Z^\vee)=2$, but $(X^\vee)_{X_Z^\vee}^\vee=X$.
\end{enumerate}
\end{remark}

Our next goal is to study more in detail the relation between $\mathrm{def}(X,Z)$ and $\mathrm{def}(X)$.
The most important thing to understand is whether $Z$ intersects the generic contact locus $\mathrm{Cont}(H,X)$, namely if $Z\cap \mathrm{Cont}(H,X)\neq\emptyset$ for all $H$ in an open dense subset of $X_Z^\vee$. If this is the case, then it is immediate to conclude that $X_Z^\vee=X^\vee$. We can be more precise when $Z=X\cap Y$ for some generic complete intersection variety $Y$.

\begin{proposition}\label{prop: relation codim dual relative dual generic complete intersection}
Let $X\subset\PP^N$ be an irreducible variety, and let $Y$ be a generic complete intersection variety of codimension $c\le n=\dim(X)$. If $Z=X\cap Y$, then
\[
\codim_{X^\vee}(X_Z^\vee)=\max\{0,c-\mathrm{def}(X)\}\quad\text{and}\quad\mathrm{def}(X,Z)=\max\{0,\mathrm{def}(X)-c\}\,.
\]
\end{proposition}
\begin{proof}
The genericity of $Y$ tells us that $\codim_{X}(Z)=c$, and ensures that $X$ is dual regular relative to $Z$. In particular, using Lemma \ref{lem: properties relative duality}(1), we conclude that for a generic $H\in X_Z^\vee$,
\begin{equation}\label{eq: contact locus intersect Y}
    \mathrm{Cont}(H,X,Z)=\mathrm{Cont}(H,X)\cap Z=\mathrm{Cont}(H,X)\cap Y\,.
\end{equation}
There are two scenarios:
\begin{enumerate}
    \item First, assume that $c\le\mathrm{def}(X)$. Consider the projections $\pi_2\colon W_X\to X^\vee$ and $\pi_{2,Z}\colon W_{X,Z}\to X_Z^\vee$.
    For a generic $H\in X_Z^\vee$, its contact locus $\mathrm{Cont}(H,X)$ is linear of dimension $\mathrm{def}(X)$. Then necessarily $\mathrm{Cont}(H,X)$ meets $Y$, implying that $X_Z^\vee=X^\vee$. Hence $\codim_{X^\vee}(X_Z^\vee)=0$ and $\mathrm{def}(X,Z)=\mathrm{def}(X)-c$ by Proposition \ref{prop: formulas relative codim X dual X Z dual}(3).
    \item Otherwise $c>\mathrm{def}(X)$.
    In this case, the generic contact locus $\mathrm{Cont}(H,X)$ does not meet $Y$. If instead $\mathrm{Cont}(H,X)\cap Y\neq\emptyset$, we claim that $\mathrm{Cont}(H,X)\cap Y$ consists of one point only. To prove the claim, let $L=\mathrm{Cont}(H,X)$ and assume by contradiction that $Y \cap L$ contains at least two points $p,q$.
    Consider a hyperplane $H$ not containing $p$ and the projection $\pi\colon\PP^N\dashrightarrow H$.
    Notice that $\pi(X)=\overline{X,p}\cap H$ for any subvariety $X$, where $\overline{X,p}\coloneqq\bigcup_{q\in X}\overline{qp}$ is the cone over $X$ with vertex $p$, see \cite[Example 3.10]{harris1992algebraic}.
    It follows that $\pi(q)\in \pi(L)\cap\pi(Y)$. On the other hand, $\dim(\pi(L))=\dim(L\cap H)=\dim(L)-1$ and $\dim(\pi(Y))=\dim(\overline{Y,p}\cap H)=\dim(Y)$. This implies that $\dim(\pi(L))+\dim(\pi(Y))=\dim(L)+\dim(Y)-1<\dim(H)$. Since $\pi(Y)$ is generic too, then the previous dimension count tells us that $\pi(L)\cap\pi(Y)=\emptyset$, thus giving a contradiction.
    Using \eqref{eq: contact locus intersect Y}, our claim implies that $\mathrm{def}(X,Z)=0$. Furthermore $\codim_{X^\vee}(X_Z^\vee)=c-\mathrm{def}(X)$ by Proposition \ref{prop: formulas relative codim X dual X Z dual}(3).\qedhere
\end{enumerate} 
\end{proof}

\begin{example}\label{ex: discriminant cubic binary form}
For a nonlinear example, consider again the discriminant of a cubic binary form introduced in Remark \ref{rmk: drop assumptions identity double relative dual variety}. Recall that $\mathrm{def}(X)=1$.
We consider several possible subvarieties $Z\subset X$ and the corresponding relative dual varieties $X_Z^\vee$:
\begin{enumerate}
    \item First, by Proposition \ref{prop: relation codim dual relative dual generic complete intersection}, if $Y$ is a generic hypersurface in $\PP^3$ and we consider $Z=X\cap Y$, then $X_Z^\vee=X^\vee$, because $\codim(Y)=1=\mathrm{def}(X)$. Instead, if $Y$ is a generic complete intersection curve of degree $d$, then $X_Z^\vee$ has codimension one within $X^\vee$, or rather $X_Z^\vee$ is zero-dimensional. We show in Proposition \ref{prop: class relative conormal generic Y} that in this case $\deg(X_Z^\vee)=\deg(Z)=\deg(Y)\deg(X)=4d$.
    \item Secondly, consider the line $Z=V(c_1,c_2)$ in $X$. In this case $\sI(X_Z^\vee)=\langle c_2,c_3,c_4\rangle$, which corresponds to the point $[1,0,0,0]$.
    \item Finally, consider the point $Z=V(c_1,c_2,\alpha c_3+\beta c_4)$ for arbitrary $\alpha,\beta$ and with $\beta\neq 0$. The last condition ensures that $Z$ is not contained in the singular locus of $X$ because $\langle c_1,c_2\rangle+\sI(X_{\sing})=\langle c_1,c_2,c_3^3\rangle$. Also in this case we get that $\sI(X_Z^\vee)=\langle c_2,c_3,c_4\rangle$.
\end{enumerate}
The last two cases show that the relations in Proposition \ref{prop: relation codim dual relative dual generic complete intersection} are no longer valid when $Z$ is not the intersection between $X$ and a generic complete intersection variety.
We refer to Example \ref{ex: discriminant binary form degree N} for a more general study on relative dual varieties of discriminants of binary forms.\hfill$\diamondsuit$
\end{example}

Recall the notations in \eqref{eq: notation hyperplane classes}.
As for the conormal variety $W_X$, we might consider the class of $W_{X,Z}$ in the Chow ring $A^*(\PP^N\times(\PP^N)^\vee)$, which can be written as
\begin{equation}\label{eq: def relative multidegrees}
[W_{X,Z}]=\delta_0(X,Z)h^N(h')^{n-d+1}+\cdots+\delta_{N-1-n+d}(X,Z)h^{n-d+1}(h')^N\,,
\end{equation}
for some nonnegative integer coefficients $\delta_i(X,Z)$, called the multidegrees of $W_{X,Z}$. Their interpretation is similar to the multidegrees $\delta_i(X)$, namely each coefficient $\delta_i(X,Z)$ is the intersection number
\begin{equation}\label{def: ranks delta_i(X,Z)}
\delta_i(X,Z) \coloneqq \int h^i(h')^{N-n+d-i-1}[W_{X,Z}]\,.
\end{equation}

The following result is inspired by \cite[Theorem 3.4]{holme1988geometric}.

\begin{proposition}\label{prop: relative version of Holme}
Let $Z$ be a $d$-dimensional subvariety of $X\subset\PP^N$. Let $\varepsilon_{X,Z}$ be the degree of the generic fiber of $\pi_{2,Z}\colon W_{X,Z}\to X_Z^\vee$. Consider the multidegrees $\delta_i(X,Z)$ introduced in \eqref{eq: def relative multidegrees}. Then
\begin{enumerate}
\item $X_Z^\vee$ has dimension $N-1-r$ if $\delta_i(X,Z)=0$ for all $0\le i\le r-n+d-1$ and $\delta_{r-n+d}(X,Z)\neq 0$. In this case $\delta_{r-n+d}(X,Z)=\varepsilon_{X,Z}\cdot\deg(X_Z^\vee)$.
\item The last nonzero multidegree of $W_{X,Z}$ is equal to $\deg(Z)$.
\end{enumerate}
\end{proposition}
\begin{proof}
Let $L$ be a generic subspace of $(\PP^N)^\vee$ of dimension $r$. Then the following facts are equivalent:
\begin{itemize}
    \item[$(i)$] $\dim(X_Z^\vee)\le N-1-r$
    \item[$(ii)$] $X_Z^\vee\cap L=\emptyset$
    \item[$(iii)$] $W_{X,Z}\cap \pr_2^{-1}(L)=\emptyset$
    \item[$(iv)$] $[W_{X,Z}]\cdot (h')^{N-r}=0$
    \item[$(v)$] $\sum_{j=0}^{r-1-n+d}\delta_j(X,Z)h^{N-j}(h')^{N+1+n-d-r+j}=0$
    \item[$(vi)$] $\delta_0(X,Z)=\cdots=\delta_{r-1-n+d}(X,Z)=0$ 
\end{itemize}
About the equivalence between $(ii)$ and $(iii)$, note that $X_Z^\vee\cap L=\emptyset$ if and only if $\pr_2^{-1}(X_Z^\vee\cap L)=\emptyset$, and $\pr_2^{-1}(X_Z^\vee\cap L)=\pr_2^{-1}(X_Z^\vee)\cap \pr_2^{-1}(L)=W_{X,Z}\cap \pr_2^{-1}(L)$. Then by genericity of $L$, we conclude that $[W_{X,Z}\cap \pr_2^{-1}(L)]=[W_{X,Z}]\cdot[\pr_2^{-1}(L)]=[W_{X,Z}]\cdot \pr_2^*([H']^{N-r})=[W_{X,Z}]\cdot (h')^{N-r}$.
The identity $\delta_{r-n+d}(X,Z)=\varepsilon_{X,Z}\cdot\deg(X_Z^\vee)$ and part $(2)$ are immediate.
\end{proof}

\begin{corollary}\label{cor: formula class relative conormal}
Let $Z$ be a subvariety of $X$ of dimension $d$. Then $0\le N-\dim(X_Z^\vee)-n+d-1\le d$ and
\[
[W_{X,Z}]=\sum_{i=N-\dim(X_Z^\vee)-n+d-1}^{d}\delta_i(X,Z)h^{N-i}(h')^{n-d+1+i}\,.
\]
\end{corollary}

Example \ref{ex: discriminant cubic binary form} tells us that understanding the multidegrees $\delta_i(X,Z)$ for arbitrary $X$ and $Z\subset X$ might be challenging. The following result is a consequence of Propositions \ref{prop: relation codim dual relative dual generic complete intersection} and \ref{prop: relative version of Holme}, and describes completely the relation between the multidegrees $\delta_i(X)$ and the relative multidegrees $\delta_i(X,Z)$, when $Z$ is the intersection between $X$ and a generic complete intersection variety.

\begin{proposition}\label{prop: class relative conormal generic Y}
Let $X\subset\PP^N$ be an irreducible variety, and let $Y$ be a generic complete intersection variety of codimension $c\le n=\dim(X)$. If $Z=X\cap Y$, then
\begin{equation}\label{eq: identities relative multidegrees}
\delta_{i-c}(X,Z)=\deg(Y)\delta_i(X)\quad\forall\,i\in\{\max\{\mathrm{def}(X),c\},\ldots,n\}\,.
\end{equation}
In particular, considering the degree $\varepsilon_{X,Z}$ introduced in Proposition \ref{prop: relative version of Holme}, we have
\begin{equation}
\varepsilon_{X,Z}=
\begin{cases}
\deg(Y) & \text{if $c\le\mathrm{def}(X)$}\\
1 & \text{if $c>\mathrm{def}(X)$}
\end{cases}
\quad\text{and}\quad
\deg(X_Z^\vee)=
\begin{cases}
\delta_{\mathrm{def}(X)}(X) & \text{if $c\le\mathrm{def}(X)$}\\
\deg(Y)\delta_{c}(X) & \text{if $c>\mathrm{def}(X)$.}
\end{cases}
\end{equation}
\end{proposition}
\begin{proof}
Thanks to \eqref{eq: def multidegrees conormal} and Proposition \ref{prop: multidegrees and polar classes Kleiman}, the class of $W_X$ in $A^*(\PP^N\times(\PP^N)^\vee)$ is equal to
\begin{equation}\label{eq: positive multidegrees conormal X}
[W_X] = \sum_{i=\mathrm{def}(X)}^n\delta_i(X)\,h^{N-i}(h')^{i+1}\,,
\end{equation}
where all the multidegrees $\delta_i(X)$ displayed are positive.
By genericity of $Y$, we have $W_{X,Z}=W_X\cap(Y\times(\PP^N)^\vee)$ and $[W_X\cap(Y\times(\PP^N)^\vee)] = [W_X]\cdot[Y\times(\PP^N)^\vee]$, where $[Y\times(\PP^N)^\vee]=\deg(Y)\,h^c$.
Then the identities \eqref{eq: identities relative multidegrees} follow immediately.
We have two scenarios:
\begin{enumerate}
    \item First, assume that $c\le\mathrm{def}(X)$. The first nonzero multidegree of $W_{X,Z}$ is $\delta_{\mathrm{def}(X)-c}(X,Z)=\deg(Y)\delta_{\mathrm{def}(X)}(X)$.
    By Proposition \ref{prop: relative version of Holme}(1), we also know that $\delta_{\mathrm{def}(X)-c}(X,Z)=\varepsilon_{X,Z}\deg(X_Z^\vee)$, and $X_Z^\vee=X^\vee$ by Proposition \ref{prop: relation codim dual relative dual generic complete intersection}.
    Since $\delta_{\mathrm{def}(X)}(X)=\deg(X^\vee)$, we necessarily have that $\varepsilon_{X,Z}=\deg(Y)$.
    \item Otherwise $c>\mathrm{def}(X)$. By Proposition \ref{prop: relation codim dual relative dual generic complete intersection}, we have that $X_Z^\vee$ is a proper subvariety of $X^\vee$ of codimension $c-\mathrm{def}(X)$. Furthermore, in part (2)  of the proof of Proposition \ref{prop: relation codim dual relative dual generic complete intersection}, we showed that the generic contact locus $\mathrm{Cont}(H,X,Z)$ consists of a single point. This implies that $\varepsilon_{X,Z}=1$. Since the first nonzero multidegree of $W_{X,Z}$ is $\delta_0(X,Z)=\deg(Y)\delta_{c}(X)$, we conclude again by Proposition \ref{prop: relative version of Holme}(1) that $\deg(X_Z^\vee)=\deg(Y)\delta_{c}(X)$.\qedhere
\end{enumerate} 
\end{proof}

\begin{example}\label{ex: discriminant binary form degree N}
Let $X\subset\PP^N$ be the hypersurface defined by the discriminant of a binary form of degree $N$. In particular $\deg(X)=2(N-1)$. The dual variety $X^\vee$ is a rational normal curve of degree $N$. Hence, we have that
\[
[W_X] = \delta_{N-2}(X)\,h^{2}(h')^{N-1}+\delta_{N-1}(X)\,h(h')^N=N\,h^{2}(h')^{N-1}+2(N-1)\,h(h')^N\,.
\]
Let $Y$ be a generic complete intersection variety of codimension $c$. Then
\[
[W_{X,Z}] =
\begin{cases}
    \deg(Y)[N\,h^{c+2}(h')^{N-1}+2(N-1)\,h^{c+1}(h')^N] & \text{if $c\le N-2$}\\
    2(N-1)\deg(Y)\,h^N(h')^N & \text{if $c=N-1$}    
\end{cases}
\]
We discarded the case when $c=n-1$, namely when $Y$ is a set of generic points that do not belong to $X$. Observe that, when $c=n-2$, there is only one positive multidegree of $W_{X,Z}$, which is actually its degree, since $W_{X,Z}$ is now a set of points in $\PP^N\times(\PP^N)^\vee$.\hfill$\diamondsuit$
\end{example}

\begin{example}\label{ex: determinantal 3x3 matrices}
Let $X\subset\PP(\C^3\otimes\C^3)\cong\PP^8$ be the determinantal variety of $3\times 3$ matrices with complex entries and with rank at most two. In particular $\deg(X)=3$. The dual variety $X^\vee$ is the determinantal variety of $3\times 3$ matrices with rank at most one, in particular $X^\vee\cong\PP^2\times\PP^2$. We have that
\[
[W_X] = 6\,h^{5}(h')^{4}+12\,h^{4}(h')^{5}+12\,h^{3}(h')^{6}+6\,h^{2}(h')^{7}+3\,h(h')^{8}\,.
\]
We consider several subvarieties of $X$:
\begin{enumerate}
    \item Let $Y$ be a generic complete intersection variety of codimension $c$, and consider the subvariety $Z=X\cap Y$ (hence $\codim_X(Z)=c$). Then $\max\{\mathrm{def}(X),c\}=\max\{3,c\}$. This means that
    \[
    [W_{X,Z}] = \deg(Y)[6\,h^{c+5}(h')^{4}+12\,h^{c+4}(h')^{5}+12\,h^{c+3}(h')^{6}+6\,h^{c+2}(h')^{7}+3\,h^{c+1}(h')^{8}]
    \]
    if $c\le 3$, while some of the first terms vanish if $c\ge 4$. Also for $c\le 3$ we have that $X_Z^\vee=X^\vee$, while $X_Z^\vee$ is a proper subvariety of $X^\vee$ for $c\ge 4$.
    \item Instead, let $Z$ be the linear subvariety of $X$ consisting of all $3\times 3$ matrices with the first row identically zero. We have $\codim_X(Z)=2$ and in this case $X_Z^\vee$ is the linear subvariety of $X$ consisting of all $3\times 3$ matrices with the second and third row identically zero. Hence it is a proper subvariety of $X^\vee$ and $\codim_{X^\vee}(X_Z^\vee)=2$.
    We expect from Proposition \ref{prop: relative version of Holme} and Corollary \ref{cor: formula class relative conormal} that
    \[
    [W_{X,Z}]=\delta_3(X,Z)h^{5}(h')^{6}+\delta_4(X,Z)h^{4}(h')^{7}+\delta_5(X,Z)h^{3}(h')^{8}\,.
    \]
    Indeed we have that $\delta_3(X,Z)=1$, $\delta_4(X,Z)=2$, and $\delta_5(X,Z)=1$. In this case, we also have $\varepsilon_{X,Z}=1$.
    \item Consider $\{x_{ij}\mid 1\le i,j\le 3\}$ as a set of homogeneous coordinates of $\PP^8$. Hence $X$ is cut out by the determinant of the matrix $(x_{ij})$. Let $Z=V(x_{11},x_{12},x_{21}x_{32}-x_{22}x_{31})$, hence $\codim_X(Z)=2$ and $\deg(Z)=2$.
    In this case, $X_Z^\vee$ is the subvariety consisting of all $3\times 3$ matrices with the third column identically zero and with the first two columns proportional. In particular, it is a proper subvariety of $X^\vee$ of degree three and dimension three. Therefore $\codim_{X^\vee}(X_Z^\vee)=1$, and by Proposition \ref{prop: formulas relative codim X dual X Z dual}$(3)$, we conclude that
    \[
    \mathrm{def}(X,Z)=\codim_{X^\vee}(X_Z^\vee)-\codim_X(Z)+\mathrm{def}(X)=1-2+3=2\,.
    \]
    Finally, using again Proposition \ref{prop: relative version of Holme} and Corollary \ref{cor: formula class relative conormal}, we conclude that the class of the relative conormal variety $W_{X,Z}$ is equal to
    \[
    [W_{X,Z}]=\delta_2(X,Z)h^{6}(h')^{5}+\delta_3(X,Z)h^{5}(h')^{6}+\delta_4(X,Z)h^{4}(h')^{7}+\delta_5(X,Z)h^{3}(h')^{8}\,.
    \]
    In particular we obtain that $\delta_2(X,Z)=\deg(X_Z^\vee)=3$, $\delta_3(X,Z)=5$, $\delta_4(X,Z)=4$, and $\delta_5(X,Z)=\deg(Z)=2$.
    In this case, we also have $\varepsilon_{X,Z}=1$.\hfill$\diamondsuit$
\end{enumerate}
\end{example}

Next, we introduce a relative analog of the polar degrees. We consider the restriction $\gamma_X|_Z\colon Z\dasharrow\G(n,N)$ of the Gauss map $\gamma_X$ to $Z$. We give the following definition.

\begin{definition}\label{def: relative polar variety}
The {\em polar variety} of an $m$-dimensional projective variety $X\subset\PP^N$ relative to a subvariety $Z\subset X$ and with respect to a subspace $L\subset\PP^N$ is
\[
\sP(X,L,Z)\coloneqq(\gamma_X|_Z)^{-1}(\Sigma_n(L))=\overline{\{x\in X_{\sm}\cap Z\mid\dim(T_xX\cap L)>\dim(L)-\codim(X)\}}\,.
\]
\end{definition}

If $X$ is a smooth variety, it follows immediately that $\sP(X,L,Z)=\sP(X,L)\cap Z$, otherwise a priori only the inclusion $\sP(X,L,Z)\subset\sP(X,L)\cap Z$ holds. Despite that, we can also conclude that $\dim(\sP(X,L,Z))=d-\codim(\sP(X,L))$ when $L$ is chosen generic. In particular, the class $[\sP(X,L,Z)]\in A^*(Z)$ does not depend on a generic choice of $L$.
More precisely, we have the following definition.

\begin{definition}\label{def: relative polar degrees}
For every $0\le i\le d$, let $L\subset\PP^N$ be a generic subspace of dimension $\dim(L)=N-n+i-2$. We define $p_i(X,Z)=[\sP(X,L,Z)]$ the {\em $i$th polar class of $X$ relative to $Z$}.
Furthermore, we define $\mu_i(X,Z)=\deg(p_i(X,Z))$ the {\em $i$th polar degree of $X$ relative to $Z$}.
\end{definition}

On one hand, for $i=0$ we have $\dim(L)=N-n-2$ and $\sP(X,L,Z)=Z$, hence $\mu_0(X,Z)=\deg(Z)$. On the other hand, assume $i>d$, or $\dim(L)>N-n+d-2$. Because $L$ is generic, we conclude that $\dim(\sP(X,L,Z))=d-\codim(\sP(X,L))=d-N+n-i<n-N<0$, namely $\sP(X,L,Z)=\emptyset$.
In general, the class $p_i(X,Z)$ represents a subvariety of $Z$ of codimension $i$. Therefore, the relative polar degrees $\mu_i(X,Z)$ can be positive only if $0\le i\le d$.

\begin{remark}\label{rmk: relations relative polar classes Chern classes jet bundle}
If $X$ is smooth and $Z$ is an irreducible subvariety of $X$, it follows by the identities \eqref{eq: relations polar classes Chern classes first jet bundle X} and Definition \ref{def: relative polar degrees} that
\begin{equation}\label{eq: relations relative polar classes Chern classes first jet bundle X restricted Z}
    p_i(X,Z)=c_i(\sP^1(\sO_X(1))|_Z)\quad\forall\,i\in\{0,\ldots,d\}\,,
\end{equation}
where $\sP^1(\sO_X(1))|_Z$ is the restriction to $Z$ of the first jet bundle of $\sO_X(1)$.    
\end{remark}

\begin{lemma}\label{lem: identity relative polar locus}
Let $Z\subset X\subset\PP^N$ be irreducible varieties.
Consider the projections $\pi_{1,Z}\colon W_{X,Z}\to Z$ and $\pi_{2,Z}\colon W_{X,Z}\to X_Z^\vee$. Then $\sP(X,L,Z) = \pi_{1,Z}(\pi_{2,Z}^{-1}(L^\vee))$ for every subspace $L\subset\PP^N$.
\end{lemma}
\begin{proof}
Recall that $L^\vee=\{H\in(\PP^N)^\vee\mid\text{$H\supset L$ as hyperplane in $\PP^N$}\}$. Then
\begin{align*}
\pi_{2,Z}^{-1}(L^\vee) &= \overline{\{(x,H)\mid\text{$x\in X_{\sm}\cap Z$, $H\in L^\vee$, and $H\supset T_xX$}\}}\\
&= \overline{\{(x,H)\mid\text{$x\in X_{\sm}\cap Z$, $H\supset L$, and $H\supset T_xX$}\}}\\
&= \overline{\{(x,H)\mid\text{$x\in X_{\sm}\cap Z$ and $H\supset\langle L,T_xX\rangle$}\}}\,.
\end{align*}
Therefore
\begin{align*}
\pi_{1,Z}(\pi_{2,Z}^{-1}(L^\vee)) &= \overline{\{x\in X_{\sm}\cap Z\mid\text{$H\supset\langle L,T_xX\rangle$ for some $H\in(\PP^N)^\vee$}\}}\\
&= \overline{\{x\in X_{\sm}\cap Z\mid\dim(\langle L,T_xX\rangle)<N\}}\\
&= \overline{\{x\in X_{\sm}\cap Z\mid\dim(L\cap T_xX)>\dim(L)-\codim(X)\}}\\
&= \sP(X,L,Z)\,.\qedhere
\end{align*}
\end{proof}

When $Z=X$, the previous result tells us that $\sP(X,L) = \pi_1(\pi_2^{-1}(L^\vee))$, where $\pi_1$ and $\pi_2$ are now the projections of $W_X$ onto $X$ and $X^\vee$, respectively.
The next result extends Proposition \ref{prop: multidegrees and polar classes Kleiman} to the relative setting.

\begin{proof}[Proof of Proposition \ref{prop: properties relative polar ranks}]
Since $X$ and $Z$ are smooth varieties, the variety $W_{X,Z}$ has the structure of a projective bundle over $Z$, hence is a smooth variety as well.

Let $L$ be a generic projective subspace of $\PP^N$ of dimension $N-n+i-2$. Then by definition $\mu_i(X,Z)=\deg(p_i(X,Z))=\deg(\sP(X,L,Z))$. The dimension of $L^\vee$ is $N-1-\dim(L)=n-i+1$. Let $H$ and $H'$ be hyperplanes in $\PP^N$ and $(\PP^N)^\vee$, respectively. In particular $[L^\vee]=[H']^{N-1-n+i}$. Recalling that $j_Z\colon W_{X,Z}\hookrightarrow\PP^N\times(\PP^N)^\vee$ is the inclusion, the notations in \eqref{eq: notation hyperplane classes}, and knowing that $\dim(\sP(X,L,Z))=d-i$, we have
\begin{align*}
\mu_i(X,Z) &= \deg(\sP(X,L,Z)) = \int [H]^{d-i}\cdot[\sP(X,L,Z)]\\
&\stackrel{\mathclap{(*)}}{=} \int [H]^{d-i}\cdot[\pi_{1,Z}(\pi_{2,Z}^{-1}(L^\vee))] = \int [H]^{d-i}\cdot \pi_{1,Z*}(\pi_{2,Z}^*([L^\vee]))\\
&= \int [H]^{d-i}\cdot \pi_{1,Z*}(\pi_{2,Z}^*([H']^{N-1-n+i}))\stackrel{\mathclap{(**)}}{=} \int \pi_{1,Z*}(\pi_{1,Z}^*([H]^{d-i})\cdot\pi_{2,Z}^*([H']^{N-1-n+i}))\\
&= \int \pi_{1,Z*}(\pi_{1,Z}^*([H])^{d-i}\cdot\pi_{2,Z}^*([H'])^{N-1-n+i}) = \int \pi_{1,Z*}(j_Z^*(h^{d-i})\cdot j_Z^*((h')^{N-1-n+i}))\\
&= \int \pi_{1,Z*}(j_Z^*(h^{d-i}\cdot (h')^{N-1-n+i})) = \int \pi_{1,Z*}(h^{d-i}\cdot (h')^{N-1-n+i}\cdot[W_{X,Z}])\\
&\stackrel{\mathclap{(***)}}{=} \int h^{d-i}\cdot (h')^{N-1-n+i}\cdot[W_{X,Z}] = \delta_{d-i}(X,Z)\,.
\end{align*}
In $(*)$ we applied Lemma \ref{lem: identity relative polar locus}. In $(**)$ we applied the general projection formula
\begin{equation}\label{eq: projection formula}
    f_*(\alpha\cdot f^*(\beta)) = f_*(\alpha)\cdot\beta\,,
\end{equation}
where $f\colon X\to Y$ is a proper morphism of smooth projective varieties, while $\alpha$ and $\beta$ are elements of $A^*(X)$ and $A^*(Y)$, respectively. In $(***)$, we used the fact that $h^{d-i}(h')^{N-n+i-1}[W_{X,Z}]$ is a zero-dimensional cycle, so the degree is preserved under projection.
\end{proof}

An immediate consequence of Propositions \ref{prop: properties relative polar ranks} and \ref{prop: class relative conormal generic Y} is the following result, that generalizes the formula in the case $c=1$, observed in \cite[p. 272]{piene1978polar}.

\begin{corollary}\label{cor: new identity relative polar degrees}
Let $X\subset\PP^N$ be a smooth variety and let $Y$ be a generic complete intersection variety of codimension $c\le n=\dim(X)$. If $Z=X\cap Y$, then
\begin{equation}
    \mu_i(X,Z)=\deg(Y)\mu_i(X)\quad\forall\,i\in\{0,\ldots,n-\max\{\mathrm{def}(X),c\}\}\,.
\end{equation}
\end{corollary}

\section{Relative tangency in determinantal varieties}\label{sec: relative tang determinantal}

Consider two vector spaces $V$ and $W$ of dimensions $N+1$ and $M+1$ respectively, with $N\ge M$. The tensor product $V^*\otimes W$ is isomorphic to the vector space $\mathrm{Hom}(V,W)$ of linear maps $f\colon V\to W$. After fixing bases of $W$ and $V$, the elements of $V^*\otimes W\cong\C^{M+1}\otimes\C^{N+1}\cong\C^{(M+1)\times (N+1)}$ are simply $(M+1)\times(N+1)$ matrices with complex entries.
Given an integer $1\le r\le M+1$, we consider the subset
\begin{equation}
X_r\coloneqq\{A\in V^*\otimes W\mid\rank(A)\le r\}\,.
\end{equation}
This is a projective variety in $\PP(V^*\otimes W)$. In particular $(X_r)_{\sm}=\{A\in V^*\otimes W\mid\rank(A)=r\}$, or equivalently $(X_r)_{\sing}=X_{r-1}$.
The codimension and the degree of $X_r$ are well-known:
\begin{equation}\label{eq: dim deg X_r}
    \codim(X_r)=(M+1-r)(N+1-r)\,,\quad\deg(X_r)=\prod_{i=0}^{M+1-r}\frac{\binom{N+1+i}{r}}{\binom{r+i}{r}}\,.
\end{equation}
The previous formulas can be found for example in \cite[\S II.5]{arbarello1985geometry}, and are a consequence of Porteous' formula that computes the fundamental class of a degeneracy locus of a morphism of vector bundles in terms of Chern classes.
The conormal variety of $X_r$ is (see \cite[Prop. I.4.11 and Lemma I.4.12]{gelfand1994discriminants})
\begin{align}\label{eq: conormal determinantal}
\begin{split}
W_{X_r} &= \{(A,B)\mid\text{$A\in X_r$, $B\in X_{M+1-r}$, $\image(A)\subset\ker(B^\mT)$ and $\image(A^\mT)\subset\ker(B)$}\}\\
&= \{(A,B)\mid\text{$A\in X_r$, $B\in X_{M+1-r}$, $B^\mT A=0$ and $BA^\mT=0$}\}\,.
\end{split}
\end{align}
As a consequence, the dual variety $X_r^\vee$ coincides with $X_{M+1-r}$. In particular we have that
\begin{equation}\label{eq: defect X_r}
\mathrm{def}(X_r)=\codim(X_r^\vee)-1=\codim(X_{M+1-r})-1=r(N-M+r)-1\,.
\end{equation}

\subsection{Duality of low-rank matrices relative to special column and row spaces}\label{sec: relative duality row column space}

Pick two linear subspaces $L_1 \subset V$ and $L_2 \subset W$ of dimensions $\ell_1+1$ and $\ell_2+1$, and consider the tensor product $L_1^*\otimes L_2\subset V^*\otimes W$. In coordinates, this subspace corresponds to the set of all $(M+1)\times(N+1)$ matrices $A$ such that the column span of $A$ (corresponding to the range of $A$) is contained in $L_2$ and the row span of $A$ (corresponding to the range of $A^\mT$) is contained in $L_1$.
It is natural to consider the intersection $Z=X_r\cap\PP(L_1^*\otimes L_2)$ and to study the relative dual variety $(X_r)_Z^\vee$. In the next result, we compute its dimension, particularly whether $(X_r)_Z^\vee$ coincides with the classical dual variety $X_r^\vee$.

\begin{proposition}\label{prop: dim relative dual Z X_r}
Let $X_r\subset\PP(V^*\otimes W)$ denote the projective variety of linear maps $f\colon V\to W$ of rank at most $r$, where $\dim(\PP(V))=N\ge M=\dim(\PP(W))$.
Let $\PP(L_1)\subset\PP(V)$, $\PP(L_2)\subset\PP(W)$ be subspaces of dimensions $\ell_1$ and $\ell_2$, and consider $Z=X_r\cap\PP(L_1^*\otimes L_2)$.
\begin{itemize}
    \item[$(i)$] If $\min\{\ell_1,\ell_2\}<r-1$, then $(X_r)_Z^\vee=\emptyset$.
    \item[$(ii)$] If instead $\min\{\ell_1,\ell_2\}\ge r-1$, then
    \begin{equation}\label{eq: defect X_r Z}
    \mathrm{def}(X_r,Z)=r\,\max\{\ell_1-M,0\}+r^2-1\,.
    \end{equation}
    As a consequence, we have that
    \begin{equation}\label{eq: relative codim X_r Z dual}
    \codim_{X_r^\vee}((X_r)_Z^\vee)=r(2M-\ell_1-\ell_2)+r\,\max\{\ell_1-M,0\}\,.
    \end{equation}
    In particular $(X_r)_Z^\vee=X_r^\vee$ if and only if $\min\{\ell_1,\ell_2\}\ge M$.
\end{itemize}
\end{proposition}

\begin{proof}
$(i)$ First, consider the case $\min\{\ell_1,\ell_2\}<r-1$: then every element in $L_1^*\otimes L_2$ has rank smaller than $r$, hence $Z$ is contained in the singular locus of $X_r$ and both $W_{X,Z}$ and $(X_r)_Z^\vee$ are empty.

$(ii)$ Assume $\min\{\ell_1,\ell_2\}\ge r-1$. To prove the statement, we need to study the defect $\mathrm{def}(X_r,Z)$ depending on the chosen parameters $M$, $N$, $\ell_1$, $\ell_2$, and $r$.
First, we have that
\begin{align*}
W_{X_r,Z} &= \overline{\{(A,B)\mid\text{$\rank(A)=r$, $\rank(B)=M+1-r$, $\image(A)\subset\ker(B^\mT)\cap L_2$, $\image(A^\mT)\subset\ker(B)\cap L_1$}\}}\,.
\end{align*}
For a generic matrix $B\in X_r^\vee=X_{M+1-r}$, we have $\rank(B)=M+1-r$ and
\[
\mathrm{Cont}(B,X_r,Z)=\overline{\{A\mid\text{$\rank(A)=r$, $\image(A)\subset\ker(B^\mT)\cap L_2$ and $\image(A^\mT)\subset\ker(B)\cap L_1$}\}}\,.
\]
We know that $\dim(\ker(B^\mT))=r$, therefore a generic $B$ is such that the $r$-dimensional subspace $\image(A)$ is contained in $\ker(B^\mT)\cap L_2$ only if $L_2=W$, namely $\ell_2=M$. Similarly, for a generic $B$, the $r$-dimensional subspace $\image(A^\mT)$ is contained in $\ker(B)\cap L_1$ only if
\begin{align*}
\dim(\PP(\ker(B)))+\dim(\PP(L_1)) &\ge \dim(\PP(V))+\dim(\PP(\image(A^\mT)))\\
N-M+r-1+\ell_1 &\ge N+r-1\,,
\end{align*}
namely $\ell_1\ge M$. The previous considerations lead to the fact that $(X_r)_Z^\vee=X_r^\vee$ if and only if $\ell_2=M$ and $\ell_1\ge M$.

Now what happens if $\ell_2<M$, namely $L_2\subsetneq W$? In this case, the generic matrix $B\in X_r^\vee$ is such that $\ker(B^\mT)\not\subset L_2$, therefore the $r$-dimensional subspace $\image(A)$ cannot be contained in the intersection $\ker(B^\mT)\cap L_2$ by dimensional reasons. This implies that $\mathrm{Cont}(B,X_r,Z)=\emptyset$ for a generic $B\in X_r^\vee$, or rather that $(X_r)_Z^\vee\subsetneq X_r^\vee$. Instead, consider a generic $B\in (X_r)_Z^\vee$. Our goal is to compute $\dim(\mathrm{Cont}(B,X_r,Z))$. Since $\mathrm{Cont}(B,X_r,Z)\neq\emptyset$ by assumption, there exists a matrix $A$ of rank $r$ such that $\image(A)\subset\ker(B^\mT)\cap L_2$. Again, a generic $B\in(X_r)_Z^\vee$ has rank $M+1-r$, so $\dim(\ker(B^\mT))=r-1$ and by genericity $\ker(B^\mT)\subset L_2$, therefore the previous condition is equivalent to $\image(A)=\ker(B^\mT)$.
Regarding the second condition $\image(A^\mT)\subset\ker(B)\cap L_1$, there are two options:
\begin{enumerate}
    \item $\ell_1\ge M$: in this case, we know that for the generic $B$, the condition $\image(A^\mT)\subset\ker(B)\cap L_1$ has to be satisfied by dimensional reasons, in particular $\dim(\PP(\ker(B)\cap L_1))=\ell_1-M+r-1$. A similar computation tells us that
    \begin{align*}
    \mathrm{def}(X_r,Z) &= \dim(\mathrm{Cont}(B,X_r,Z))\\
    &=\dim(\mathrm{GL}(r))-1+\dim(\G(r-1,\ell_1-M+r-1))\\
    &= r^2-1+r(\ell_1-M)\\
    &= r(\ell_1-M+r)-1\,.
    \end{align*}
    \item $\ell_2<M$: keeping into account that $\mathrm{Cont}(B,X_r,Z)\neq\emptyset$ we know that for the generic $B$, the intersection $\PP(\ker(B)\cap L_2)$ has dimension exactly $r-1$, therefore the inclusion $\image(A^\mT)\subset\ker(B)\cap L_2$ is an equality. Hence, we only need to consider the action of $\mathrm{GL}(r)$, namely 
    \[
    \mathrm{def}(X_r,Z) = \dim(\mathrm{Cont}(B,X_r,Z)) = \dim(\mathrm{GL}(r))-1=r^2-1\,.
    \]
\end{enumerate}
It remains to consider the case $\ell_2=M$ and $\ell_1<M$, but this case is similar to $\ell_2<M$ and $\ell_1<M$, hence also in this case $\mathrm{def}(X_r,Z)=r^2-1$.
All the previous subcases lead to the relative defect formula in \eqref{eq: defect X_r Z}. Finally, knowing that
\begin{align*}
\codim_{X_r}(Z) &= \dim(X_r)-\dim(Z)\\
&= [r(M+N+2)-r^2-1]-[r(\ell_1+\ell_2+2)-r^2-1]\\
&= r(N-\ell_1+M-\ell_2)\,.
\end{align*}
and applying Proposition \ref{prop: formulas relative codim X dual X Z dual}$(3)$, we obtain equation \eqref{eq: relative codim X_r Z dual}:
\begin{align*}
    \codim_{X_r^\vee}((X_r)_Z^\vee) &= \codim_{X_r}(Z)+\mathrm{def}(X_r,Z)-\mathrm{def}(X_r)\\
    &= r(N-\ell_1+M-\ell_2)+r\,\max\{\ell_1-M,0\}+r^2-1-r(N-M+r)+1\\
    &= r(2M-\ell_1-\ell_2)+r\,\max\{\ell_1-M,0\}\,.
\end{align*}
This completes the proof.
\end{proof}

\begin{remark}\label{rmk: relative duality tensor subspaces}
A similar study may be carried out for higher-order tensors.
Let $k\ge 2$ be an integer. For all $i\in[k]$, let $V_i$ be a complex vector space of dimension $n_i+1$. The tensor product $V=V_1^*\otimes\cdots\otimes V_{k-1}^*\otimes V_k$ is isomorphic to the vector space of multilinear maps $f\colon V_1\times\cdots\times V_{k-1}\to V_k$. After fixing bases for the spaces $V_i$, these multilinear maps are represented by $k$-dimensional arrays filled with complex numbers. We call them {\em tensors of format $(n_1+1,\ldots,n_k+1)$}.
For all $i\in[k]$, we consider $\PP^{n_i}=\PP(V_i)$ and we call $X$ the Segre embedding of $\PP^{n_1}\times\cdots\times\PP^{n_k}$ into $\PP(V)$. Its elements are (classes of) {\em decomposable} or {\em rank-one} tensors.
For a given integer $r\ge 1$, we call $X_r$ the $r$th secant variety of $X$. When $k=2$, we recover the variety of matrices of rank at most $r$. For $k\ge 3$, the notion of rank is subtler, and we refer to \cite{landsberg2012tensors} for more details.
 
For all $i\in[k]$, we fix linear subspaces $L_i\subset V_i$, and we consider the intersection $Z=X_r\cap\PP(L_1\otimes\cdots\otimes L_k)$. It would be interesting to study the relative dual varieties $X_Z^\vee$.
This is a hard problem already in the non relative setting.  The dual-defectivity of the varieties $X_r$ is well understood only in the case $r=1$ where $X^\vee$ is a hypersurface in $\PP(V^*)$ if and only if $n_i\le\sum_{j\neq i}n_j$ for all $i\in[k]$, see \cite[Chapter 14, Theorem 1.3]{gelfand1994discriminants}.
Furthermore, this case is related to the study of generalized Kalman varieties of matrices and tensors, which we briefly recall in Remark \ref{rmk: Kalman varieties data loci}.
\end{remark}

\subsection{Duality of low-rank matrices relative to symmetric and skew-symmetric low-rank matrices}\label{sec: relative dual determinantal symmetric}

In this section, we assume that $V=W$, and we consider the subspaces $S=S^2V$ and $A=\bigwedge^2V$ of $V\otimes V$. In coordinates, these correspond to the subspaces of symmetric and skew-symmetric $(N+1)\times(N+1)$ matrices, respectively. Define $S_r\coloneqq X_r\cap S$ and $A_r\coloneqq X_r\cap A$, namely the varieties of symmetric and skew-symmetric $(N+1)\times(N+1)$ matrices of rank at most $r$, respectively. We also denote by $S_r$ and $A_r$ the corresponding projective varieties in $\PP(V\otimes V)\cong\PP^{N(N+2)}$. We have that
\begin{equation}\label{eq: dim deg S_r}
\dim(S_r)=\binom{N+2}{2}-1-\binom{N-r+2}{2}\,,\quad\deg(S_r)=\prod_{i=0}^{N-r}\frac{\binom{N+1+i}{N+1-r-i}}{\binom{2i+1}{i}}
\end{equation}
whereas
\begin{equation}\label{eq: dim deg A_r}
\dim(A_r)=\binom{N+1}{2}-1-\binom{N+1-r}{2}\,,\quad\deg(A_r)=\frac{1}{2^{N-r}}\prod_{i=0}^{N-r-1}\frac{\binom{N+1+i}{N-r-i}}{\binom{2i+1}{i}}
\end{equation}
for all even $r$, while $A_r=A_{r-1}$ for every odd integer $r\ge 1$ because the rank of a skew-symmetric matrix is always even.
The degree formulas are due to Corrado Segre in the symmetric case and Giovanni Giambelli in the skew-symmetric case.
Similarly to the non-symmetric case explained in \eqref{eq: conormal determinantal}, we have that (see also \cite[Example 5.15]{rostalski2013dualities})
\begin{equation}\label{eq: conormal determinantal symmetric}
W_{S_r} = \{(A,B)\mid\text{$A\in S_r$, $B\in S_{n-r}$ and $AB=0$}\}\,,
\end{equation}
whereas
\begin{equation}\label{eq: conormal determinantal antisymmetric}
W_{A_r} =
\begin{cases}
    \{(A,B)\mid\text{$A\in A_r$, $B\in A_{n-r}$ and $AB=0$}\} & \text{if $r$ is even}\\
    \emptyset & \text{if $r$ is odd.}
\end{cases}
\end{equation}
Therefore, in the symmetric case, the projection of $W_{S_r}$ on the second factor is $S_r^\vee=S_{n-r}$.
In the skew-symmetric case, the projection of $W_{A_r}$ on the second factor is $A_r^\vee=A_{n-r}$ for $r$ even, while $A_r^\vee=\emptyset$ for $r$ odd. 

In this section, we consider $S_r$ and $A_r$ as subvarieties of $X_r$, and describe the relative dual varieties $(X_r)_{S_r}^\vee$ and $(X_r)_{A_r}^\vee$. First, we recall the following definition.

\begin{definition}\label{def: compound matrix}
Let $A$ be an $(M+1)\times(N+1)$ matrix for some integers $1\le M\le N$. Given subsets $I$ and $J$ of $[M+1]$ and $[N+1]$ with $|I|=|J|$, we denote by $A[I;J]$ the submatrix of $A$ obtained by selecting the rows and the columns of $A$ indexed by $I$ and $J$, respectively. Consider an integer $r\ge 1$. If $r>M+1$, we define $C_r(A)\coloneqq 0$. Otherwise $C_r(A)$ is a matrix of size $\binom{M+1}{r}\times\binom{N+1}{r}$. The rows and the columns of $C_r(A)$ are indexed by all subsets of $[M+1]$ and $[N+1]$ with $r$ elements, ordered with the lexicographic order. Finally, we set
\[
(C_r(A))_{I,J}=\det(A[I;J])\quad\forall (I,J)\in 2^{[M+1]}\times 2^{[N+1]}\,,|I|=|J|=r\,.
\]
The matrix $C_r(A)$ is called the {\em $r$th compound matrix} of $A$. It is also denoted by $\bigwedge^rA$.
\end{definition}

\begin{proposition}\label{prop: relative dual determinantal wrt subspace symmetric}
Consider the determinantal variety $X_r\subset\PP(V\otimes V)\cong\PP^{N(N+2)}$ and its subvarieties $S_r$ and $A_r$. Then
\begin{equation}
(X_r)_{S_r}^\vee=\overline{\{B\mid\text{$\rank(B)=N+1-r$ and $C_{N+1-r}(B)$ is symmetric}\}}\,.
\end{equation}
Furthermore $(X_r)_{A_r}^\vee=(X_r)_{S_r}^\vee$ for $r$ even and $(X_r)_{A_r}^\vee=\emptyset$ for $r$ odd.
\end{proposition}
\begin{proof}
Using \eqref{eq: conormal determinantal}, the relative conormal variety $W_{X_r,S_r}$ is the Zariski closure of the set
\begin{align}\label{eq: relative conormal determinantal symmetric}
\begin{split}
\sU_r &= \{(A,B)\mid\text{$A\in S_r\setminus S_{r-1}$, $B\in X_{N+1-r}\setminus X_{N-r}$ and $AB=AB^\mT=0$}\}\\
&= \{(A,B)\mid\text{$A\in S_r\setminus S_{r-1}$, $B\in X_{N+1-r} \setminus X_{N-r}$ and $\image(A)=\ker(B)=\ker(B^\mT)$}\}\,.
\end{split}
\end{align}
To compute $(X_r)_{S_r}^\vee$, it is enough to consider the projection of $\sU_r$ which is dense in $W_{X_r,S_r}$. From the last identity in \eqref{eq: relative conormal determinantal symmetric}, we see that
\begin{align}\label{eq: projection U_r}
\begin{split}
    \pr_2(\sU_r) &= \{B\mid\text{$\rank(B)=N+1-r$ and $\ker(B) = \ker(B^\mT)$}\}\\
    &= \{B\mid\text{$\rank(B)=N+1-r$ and $\image(B) = \image(B^\mT)$}\}\,.
\end{split}
\end{align}
We claim that the last property is equivalent to imposing that the compound matrix $C_{N+1-r}(B)$ is symmetric. Let $R$ be a matrix of rank $p$, and suppose that $v$ and $w$ are the vectors in $\bigwedge^p\C^{N+1}$ corresponding to the subspaces $\image(A)$ and $\image(A^\mT)$, seen as elements of $\G(p-1,N)$. We have that
\[
C_p(R) = vw^\mT\,,
\]
in particular $C_p(R)$ has rank one, hence $C_p(R)$ is symmetric if and only if $v=w$. Continuing the chain of equalities in \eqref{eq: projection U_r}, we have that $B\in \pr_2(\sU_r)$ if and only if $C_{N+1-r}(B)$ is symmetric.

Regarding the second part of the proof, if $r$ is even, then $A_{r-1}=A_{r-2}$ and $W_{X_r,A_r}=\overline{\sV_r}$, where
\begin{equation}\label{eq: relative conormal determinantal antisymmetric}
\sV_r = \{(P,Q)\mid\text{$P\in A_r\setminus A_{r-2}$, $Q\in X_{N+1-r}\setminus X_{N-r}$ and $PQ=PQ^\mT=0$}\}\,.
\end{equation}
Since the second projection of $\sV_r$ is equal to the second projection of $\sU_r$, we conclude that $(X_r)_{A_r}^\vee=(X_r)_{S_r}^\vee$. If $r$ is odd, then $A_r=A_{r-1}$ and for this reason $W_{X_r,S_r}=\emptyset$, hence $(X_r)_{A_r}^\vee=\emptyset$.
\end{proof}

\begin{corollary}\label{cor: relative dual determinantal symmetric corank 1}
Assume $r=N$. After identifying the spaces $\PP(V\otimes V)$ and $\PP(V\otimes V)^\vee$, we have
\[
(X_N)_{S_N}^\vee=X_1\cap S=S_1\,.
\]
Furthermore $(X_N)_{A_N}^\vee=(X_N)_{S_N}^\vee=S_1$ if $N$ is even, otherwise $(X_N)_{A_N}^\vee=\emptyset$ if $N$ is odd.
\end{corollary}
\begin{proof}
By the previous result, we have that
\[
(X_N)_{S_N}^\vee=\overline{\{B \in X_1\setminus\{0\}\mid\text{$C_1(B)=B$ is symmetric}\}}=X_1\cap S=S_1\,.
\]
The last part is immediate.
\end{proof}

Corollary \ref{cor: relative dual determinantal symmetric corank 1} does not generalize to arbitrary $r$, namely it is not true that $(X_r)_{S_r}^\vee=X_{N+1-r}\cap S=S_{N+1-r}$. The first counterexample appears for $N=2$ and $r=1$. On one hand $S_{N+1-r}=S_{2}$ has codimension $5$ in $\PP^8$ and degree $3$. On the other hand, the relative dual variety $(X_1)_{S_1}^\vee$ has codimension 3 and degree 7. In particular, it is properly contained in $X_1^\vee=X_2$, which is the hypersurface defined by the determinant of a $3\times 3$ matrix.

A special property of determinantal varieties, not valid in general, is that the inclusions $X_r\subset X_{r+1}$ are reversed when taking duals, namely $(X_{r+1})^\vee\subset(X_r)^\vee$. A similar phenomenon happens when considering their relative dual varieties with respect to their intersections with the subspace $S$.

\begin{corollary}
For any $0\le r\le N$, we have that $(X_{r+1})_{S_{r+1}}^\vee\subset(X_r)_{S_r}^\vee$.
\end{corollary}

\begin{proof}
The proof is an immediate consequence of Proposition \ref{prop: relative dual determinantal wrt subspace symmetric}, because if $A\in X_j$, then all minors of $A$ of size larger than $j$ vanish, in particular $C_i(A)=0$ for all $j>i$, and the zero matrix is a special symmetric matrix.
\end{proof}

Proposition \ref{prop: relative dual determinantal wrt subspace symmetric} suggests to investigate in more detail the varieties
\begin{equation}\label{eq: def varieties SC_r}
    \mathcal{SC}_r\coloneqq\{B\in\C^{(N+1)\times(N+1)}\mid\text{$C_r(B)$ is symmetric}\}\,,
\end{equation}
where $r\in[N+1]$. With a little abuse of notation, we use the same notation for the corresponding projective variety. It is immediate to verify that $\mathcal{SC}_1=S$ and that $S\subset\mathcal{SC}_r$ for all $r\in[N+1]$.

\begin{example}
For $N=2$, we verified symbolically that $\mathcal{SC}_2$ has two irreducible components, namely $\mathcal{SC}_1=S=S^2\C^3$ and $(X_1)_{S_1}^\vee$. For $N=3$, we verified that
\begin{itemize}
    \item $\mathcal{SC}_2$ has three irreducible components, respectively $S=S^2\C^4$, $A=\bigwedge^2\C^4$, and $(X_2)_{S_2}^\vee$. This last component has codimension 8 and degree 38.
    \item $\mathcal{SC}_3$ has three irreducible components, respectively $S$, $(X_1)_{S_1}^\vee$ of codimension 4 and degree 24, and another component of codimension 4 and degree 20.\hfill$\diamondsuit$
\end{itemize}
\end{example}

The previous example tells us that in general $\mathcal{SC}_i$ is not contained in $\mathcal{SC}_j$ for all $i\le j\in[N+1]$.
We suggest the following conjecture based on more experimental evidence not discussed here.

\begin{conjecture}
For all $r\in[N+1]$, consider the variety $\mathcal{SC}_r$ introduced in \eqref{eq: def varieties SC_r}. Then
\begin{itemize}
    \item[$(i)$] $\mathcal{SC}_1=S=S^2\C^{N+1}$ is an irreducible component of $\mathcal{SC}_r$ for every $r\in[N+1]$.
    \item[$(ii)$] $(X_{n-r})_{S_{n-r}}^\vee$ is an irreducible component of $\mathcal{SC}_r$ for every $2\le r\le N+1$.
    \item[$(iii)$] $A=\bigwedge^2\C^{N+1}$ is an irreducible component of $\mathcal{SC}_r$ for all $N\ge 3$ and $r\in[N]$ even. 
\end{itemize}
\end{conjecture}

Part $(iii)$ of the previous conjecture is motivated by the fact that, if $B\in A$ and $r$ is even, then $C_r(B)$ is symmetric. Otherwise $r$ is odd and $C_r(B)$ is skew-symmetric (see \cite[Chapter 8]{cullis1913matrices}, \cite[Chapter 5]{aitken1939determinants}, and \cite{prells2002compound}). This implies that $A\subset\mathcal{SC}_r$ for every even $r\in[N+1]$. For $N=2$, we have that $A$ is contained in the irreducible component $(X_1)_{S_1}^\vee$ of $\mathcal{SC}_2$.

Motivated by Remark \ref{rmk: relative duality tensor subspaces}, we suggest the following multi-dimensional generalization of the problem studied in this section. Assuming $V_1=\cdots=V_k=V'$, we work in the tensor space $V=(V')^{\otimes k}$, and we consider the $r$th secant variety $X_r$ of the variety $X$ of decomposable tensors in $V$. The analog problem is then studying the relative dual variety $(X_r)_Z^\vee$, where in this case either $Z=X_r\cap\PP(S^kV')$ or $Z=X_r\cap\PP(\bigwedge^kV')$. Again, the problem is challenging already for $r=1$, and is related to the generalized symmetric Kalman varieties introduced in \cite[\S6]{shahidi2021degrees}.

\section{Euclidean Distance data loci of affine varieties}\label{sec: affine ED data loci}

In Sections \ref{sec: preliminaries} and \ref{sec: relative polar classes of a projective variety}, we have introduced the relative notions of tangency and duality in projective algebraic geometry, in particular the relative analog of the polar varieties of a given algebraic variety. Our motivation comes from the need to study the so-called ED data loci of algebraic varieties. To define them rigorously, we briefly recall the notion of Euclidean Distance degree of \cite{DHOST}. In this and the following sections the ground field is either $\R$ or $\C.$

Our ambient space is the Euclidean space $(V^\mR,\langle\,,\rangle^\mR)$, where $V^\mR$ is a real vector space of dimension $N+1$, and $\langle\,,\rangle^\mR\colon V^\mR\times V^\mR\to\R$ is a real inner product on $V^\mR$, which induces a positive definite quadratic form $q^\mR\colon V^\mR\to\R$ defined by $q^\mR(v)=\langle v,v\rangle^\mR$. 

Let $u\in V^\mR$, and $X^\mR\subset V^\mR$ be a real algebraic variety.  A well-known problem is computing the critical points of the squared Euclidean distance function $x\mapsto q^\mR(x-u)$ from the point $u$, when restricted to the points $x\in X^\mR$. This problem relaxes the more difficult problem of finding the closest point to $u$ on $X^\mR$.
Algebraic geometers love to consider complex algebraic varieties to study algebraically the complexity of the given problem. Here, by complexity, we intend the number of such critical points when it is finite. Therefore, we rephrase the previous problem in a complex setting. First, we consider the complex vector space $V\coloneqq V^\mR\otimes\C$. Secondly, we define the bilinear form $\langle\,,\rangle\colon V\times V\to\C$ and the complex-valued polynomial function $q\colon V\to\C$ taking the same values of $\langle\,,\rangle^\mR$ and $q^\mR$ over $V^\mR$, respectively.
After fixing a coordinate system $\{x_0,\ldots,x_N\}$ of $V$, we identify $V$ with $\C^{N+1}$ and assume that $\langle\,,\rangle^\mR$ is the standard Euclidean inner product on $V^\mR$, or equivalently that $q(x)=x_0^2+\cdots+x_N^2$.
Finally, we let $X\subset V$ be the Zariski closure of $X^\mR$.
All this leads to studying the critical points of the polynomial objective function
\begin{equation}\label{eq: def objective function}
d_{X,u}\colon X\to\C\,,\quad d_{X,u}(x)\coloneqq q(u-x)=\sum_{i=0}^N(u_i-x_i)^2\,.
\end{equation}
Since both $X$ and $q$ are defined by polynomials, the subset of critical points of $d_{X,u}$ is the vanishing locus of a certain polynomial ideal.
Furthermore, we assume that $X$ is a reduced and irreducible complex algebraic variety, and we denote by $I_X$ the radical ideal whose vanishing locus is $X$. In particular $I_X=\langle f_1,\ldots,f_s\rangle\subset\C[x]\coloneqq\C[x_0,\ldots,x_N]$ for some vector of polynomials $f\coloneqq(f_1,\ldots,f_s)$. If $J(f)$ is the Jacobian matrix of the vector $f$ and $c=\codim(X)$, then $\rank(J(f))=c$ when evaluated at a smooth point of $X$. Hence the singular locus $X_{\sing}=X\setminus X_{\sm}$ of $X$ is the vanishing locus of the ideal
\begin{equation}\label{eq: ideal sing X}
I_{X_{\sing}} = I_X+\langle\text{$c\times c$ minors of $J(f)$}\rangle\,.
\end{equation}
Now consider the ideal
\begin{equation}\label{eq: ideal ED correspondence}
I_{X,u}\coloneqq\left(I_X+\left\langle\text{$(c+1)\times(c+1)$ minors of $\begin{pmatrix}u-x\\J(f)\end{pmatrix}$}\right\rangle\right)\colon I_{X_{\sing}}^\infty\,,
\end{equation}
If $u\in V$ is fixed, then $I_{X,u}$ is an ideal in $\C[x]$ and its vanishing locus is the variety of complex critical points of $d_{X,u}$ restricted to $X$. Indeed, a smooth point $x$ of $X$ is {\em critical} for $d_{X,u}$ when the gradient of $d_{X,u}$ at $x$, which is proportional to $u-x$, lies in the normal space of $X$ at $x$
\[
N_xX\coloneqq\{y\in V\mid\text{$\langle y,v\rangle=0$ for all $v\in T_xX$}\}\,,
\]
and $N_xX$ is spanned by the rows of $J(f)$ evaluated at $x$.
Otherwise $I_{X,u}$ lives in $\C[x,u]$ and therefore defines a variety in the cartesian product $V_x\times V_u$. The subscripts in $V_x$ and $V_u$ are needed to clarify which coordinates are used in the specific copy of $V$.

\begin{definition}\label{def: ED correspondence}
The {\em ED correspondence} of $X$ is the subvariety $\sE_X\subset V_x\times V_u$ defined by the ideal $I_{X,u}$ introduced in \eqref{eq: ideal ED correspondence}, seen as an ideal in $\C[x,u]$. In particular
\begin{equation}\label{eq: ED correspondence}
\sE_X=\overline{\{(x,u)\in V_x\times V_u\mid\text{$x\in X_{\sm}$ and $x$ is critical for $d_{X,u}$}\}}\,.
\end{equation}
\end{definition}

It is worth noting that the previous complex relaxation has forgotten the original real setting of distance minimization. Therefore, one might start directly by fixing a bilinear form $\langle\,,\rangle\colon V\times V\to\C$, inducing the quadratic form $q(x)=\langle x,x\rangle$, and by considering a reduced and irreducible complex algebraic variety $X\subset V$. Then Definition \ref{def: ED correspondence} remains unchanged.

The variety $\sE_X$ is equipped with the two projections

\begin{equation}\label{eq: diagram ED correspondence}
\begin{tikzcd}
& \sE_X \arrow{dl}[swap]{\pr_1} \arrow{dr}{\pr_2} & \\
V_x & & V_u\,.
\end{tikzcd}
\end{equation}

\begin{theorem}{\cite[Theorem 4.1]{DHOST}}\label{thm: ED correspondence irreducible}
Let $X\subset V\cong\C^{N+1}$ be an irreducible variety of codimension $c$.
The ED correspondence $\sE_X$ is an irreducible variety of dimension $N+1$ inside $V_x\times V_u$.
The first projection $\pr_1\colon\sE_X\to V_x$ induces an affine vector bundle of rank $c$ over $X_{\sm}$.
Furthermore, if $T_xX\cap N_xX=\{0\}$ at some point $x\in X_{\sm}$, then the second projection $\pr_2\colon\sE_X\to V_u$ is a dominant map, and its generic fibers are positive of constant cardinality.
\end{theorem}

\begin{definition}\label{def: ED degree}
Consider the diagram \eqref{eq: diagram ED correspondence} and let $\mathrm{Crit}(d_{X,u})\coloneqq\pr_1(\pr_2^{-1}(u))\subset X$ be the {\em critical locus} of $d_{X,u}$ for all $u\in V_u$. By Theorem \ref{thm: ED correspondence irreducible}, there exists an open dense subset $\sU=\sU(d_{X,u})\subset V_u$ such that $\mathrm{Crit}(d_{X,u})$ is zero dimensional and consists of simple points. The cardinality $|\mathrm{Crit}(d_{X,u})|$ given $u\in\sU$ is called {\em Euclidean Distance degree} of $X$ and is denoted by $\mathrm{EDD}(X)$.
\end{definition}

The main goal of this section is rephrasing the notions of ED correspondence and ED degree in a ``conditional'' setting, as we motivated in the introduction. Let $Z\subset X$ be a reduced and irreducible subvariety, and consider the ideal
\begin{equation}\label{eq: ideal ED correspondence Z}
I_{X,u|Z}\coloneqq\left(I_Z+\left\langle\text{$(c+1)\times(c+1)$ minors of $\begin{pmatrix}u-x\\J(f)\end{pmatrix}$}\right\rangle\right)\colon I_{X_{\sing}}^\infty\,.
\end{equation}

\begin{definition}\label{def: ED correspondence Z}
The {\em (conditional) ED correspondence of $X$ given $Z$} is the subvariety $\sE_{X|Z}\subset V_x\times V_u$ defined by the ideal $I_{X,u|Z}$ introduced in \eqref{eq: ideal ED correspondence Z}, seen as an ideal in $\C[x,u]$. In particular
\begin{equation}\label{eq: ED correspondence Z}
\sE_{X|Z}=\overline{\{(x,u)\in V_x\times V_u\mid\text{$x\in X_{\sm}\cap Z$ and $x$ is critical for $d_{X,u}$}\}}\,.
\end{equation}
\end{definition}

\begin{proposition}\label{prop: ED correspondence Z irreducible}
Let $X\subset V\cong\C^{N+1}$ be an irreducible affine variety of codimension $c$.
Let $Z\subset X$ be an irreducible affine subvariety of codimension $\delta$ in $X$, such that $Z\not\subset X_{\sing}$.
The variety $\sE_{X|Z}$ is irreducible of dimension $N+1-\delta$ inside $V_x\times V_u$. Its projection onto the second factor $V_u$ is irreducible as well.
\end{proposition}
\begin{proof}
For a fixed $z\in X_{\sm}\cap Z$, the Jacobian matrix $J(f)$ has rank $c=\codim(X)$, so the data points $u$ where the $(c+1)\times(c+1)$-minors of the augmented Jacobian matrix in \eqref{eq: ideal ED correspondence Z} vanish form an affine-linear subspace in $V$ of dimension $c$. This means that the first projection $\pr_1$ in the diagram
\begin{equation}\label{eq: diagram ED correspondence Z}
\begin{tikzcd}
& \sE_{X|Z} \arrow{dl}[swap]{\pr_1} \arrow{dr}{\pr_2} & \\
V_x & & V_u\,.
\end{tikzcd}
\end{equation}
induces an affine vector bundle of rank $c$ over $X_{\sm}\cap Z$. Hence, the variety $\sE_{X|Z}$ is irreducible of dimension $\dim(Z)+c=N+1-\delta$. In turn, the image $\pr_2(\sE_{X|Z})$ is irreducible.
\end{proof}

We are ready to introduce ED data loci in the context of ED optimization. The problem is natural: identify the locus of data points $u\in V$ having at least one smooth critical point of the function $d_{X,u}$ on a given subvariety $Z$.

\begin{definition}\label{def: ED data locus}
Let $Z$ be an affine subvariety of $X$. The {\em ED data locus of $X$ given $Z$} is the variety
\begin{equation}\label{eq: ED data locus}
\mathrm{DL}_{X|Z}\coloneqq\overline{\pr_2(\sE_{X|Z})}=\overline{\{u\in V\mid\text{$\exists\,x\in X_{\sm}\cap Z$ such that $x$ is critical for $d_{X,u}$}\}}\,.
\end{equation}
\end{definition}

Given a vector $a\in V$ and a subset $S\subset V$, we write $a+S\coloneqq\{a+s\mid s\in S\}$. This allows us to rewrite $\mathrm{DL}_{X|Z}$ as
\begin{equation}\label{eq: alternative description of ED data locus}
\mathrm{DL}_{X|Z} = \overline{\bigcup_{x\in X_{\sm}\cap Z}x+N_xX}\,,
\end{equation}
where the bar indicates Zariski closure in $V$.
The easiest ED data loci to consider are the ones with respect to single smooth points of $X$.

\begin{example}
Let $x$ be a smooth point of $X$, and define $Z=\{x\}$. In this particular case, we have that $\mathrm{DL}_{X|Z}=x+N_xX$.\hfill$\diamondsuit$
\end{example}

A crucial property of $\mathrm{DL}_{X|Z}$, which descends immediately from Proposition \ref{prop: ED correspondence Z irreducible}, is that $\mathrm{DL}_{X|Z}$ is an irreducible variety of $V$. We also expect $\mathrm{DL}_{X|Z}$ to be a proper subvariety of $V$ because $\mathrm{Crit}(d_{X,u})\cap Z=\emptyset$ for a generic data point $u$.

\begin{definition}\label{def: conditional ED regularity}
Let $Z$ be a subvariety of $X\subset V$. We say that $X$ is {\em ED regular given $Z$} if $Z\not\subset X_{\sing}$ and if $\mathrm{DL}_{X|Z}\not\subset V_u\setminus\sU$, where $\sU$ is introduced in Definition \ref{def: ED degree}.
\end{definition}

The following Proposition is a generalization of \cite[Proposition 3.4]{horobet2023does}.
\begin{proposition}\label{prop: data locus is irreducible}
Let $X\subset V\cong\C^{N+1}$ be an irreducible affine variety.
Let $Z\subset X$ be an irreducible affine subvariety of codimension $\delta$ in $X$. If $X$ is ED regular given $Z$, then $\mathrm{DL}_{X|Z}$ is irreducible of codimension $\delta$ in $V$.
\end{proposition}
\begin{proof}
The irreducibility of $\mathrm{DL}_{X|Z}$ follows immediately by Proposition \ref{prop: ED correspondence Z irreducible}.
For the second part, the assumptions given imply that the morphism $\pr_2|_{\sE_{X|Z}}\colon\sE_{X|Z}\to \pr_2(\sE_{X|Z})$ is also generically finite, and its degree is bounded above by $\deg(\pr_2)=\mathrm{EDD}(X)$.
This implies that $\dim(\mathrm{DL}_{X|Z})=\dim(\sE_{X|Z})=N+1-\delta$.
\end{proof}

\begin{definition}\label{def: ED degree Z}
Let $Z\subset X\subset V$ be irreducible affine varieties such that $X$ is ED regular given $Z$.
Consider the diagram \eqref{eq: diagram ED correspondence Z}.
We define {\em (conditional) Euclidean Distance degree of $X$ given $Z$} the cardinality of the generic fiber of the restriction $\pr_2|_{\sE_{X|Z}}\colon\sE_{X|Z}\to \pr_2(\sE_{X|Z})$. We denote this invariant by $\mathrm{EDD}(X|Z)$.    
\end{definition}

In Proposition \ref{prop: double data strictly contained}, we give a sufficient condition on $Z$ under which $\mathrm{EDD}(X|Z)=1$, no matter what is the value of $\mathrm{EDD}(X)$. This would mean that a generic data point on $\mathrm{DL}_{X|Z}$ admits exactly one critical point of the objective function $d_{X,u}$ on $Z$.
The following is an easy example of an affine variety $X$ and a subvariety $Z$ such that $\mathrm{EDD}(X|Z)>1$.

\begin{example}\label{ex: projection not birational}
Let $S$ be a sphere in $\R^n$ with center $c$. For every $u\in\R^n$ distinct from $c$, there are exactly two critical points of $d_{X,u}|_S$.
They can be computed by intersecting $S$ with the line joining $u$ and the center of $S$. This tells us that $\mathrm{EDD}(S)=2$.
Let $H$ be a hyperplane containing $c$, and call $Z=S\cap H$. It is easy to see that $\mathrm{DL}_{S|Z}=H$, and for every $u\in H$ distinct from $c$, the two critical points of $d_{X,u}$ on $S$ belong to $H$. This means that the generic fiber of the restriction $\pr_2|_{\sE_{S|Z}}\colon\sE_{S|Z}\to \pr_2(\sE_{S|Z})$ has cardinality two, hence $\mathrm{EDD}(S|Z)=2$. This case is illustrated on the left of Figure \ref{fig: ED data locus sphere}.
Otherwise, suppose that $c\notin H$ and that $H$ meets $S$ in a smooth quadric $Z$.
Then $\mathrm{DL}_{S|Z}$ is the green quadric surface on the right of Figure \ref{fig: ED data locus sphere}, and $\mathrm{EDD}(S|Z)=1$.\hfill$\diamondsuit$
\end{example}
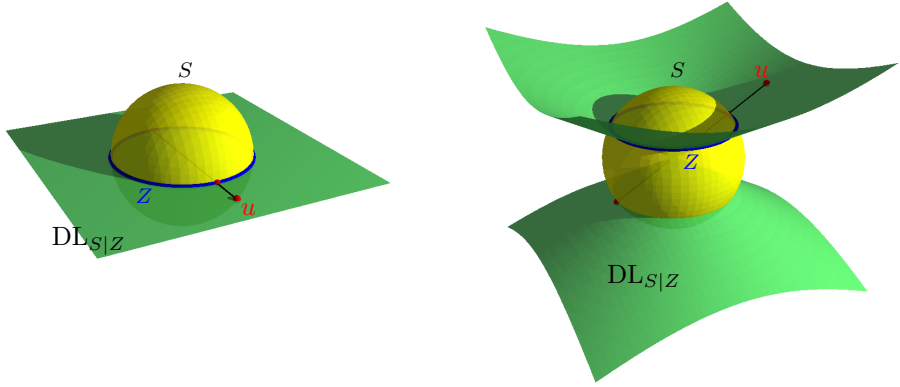
\begin{figure}[ht]
\begin{overpic}[width=2.2in]{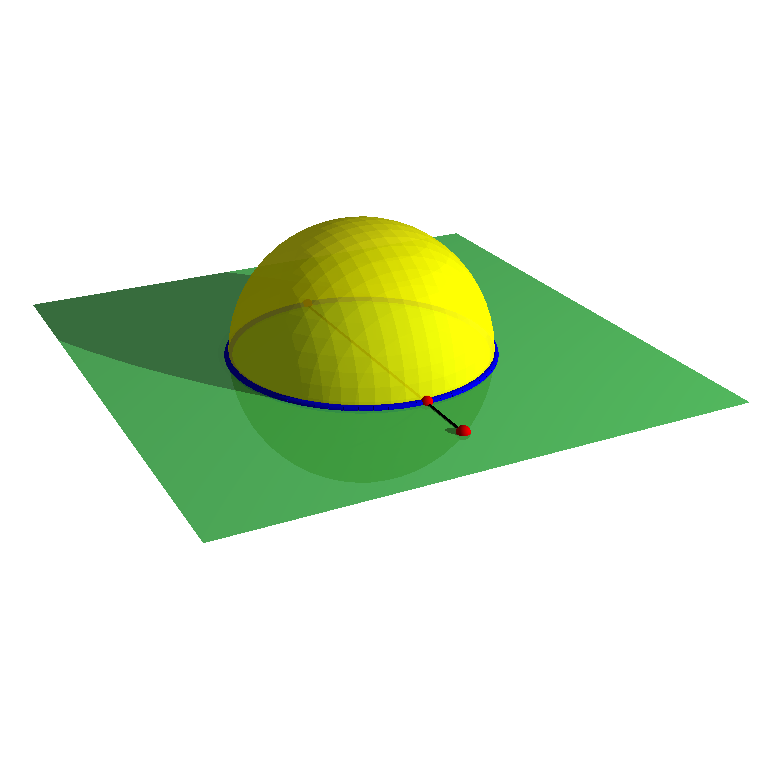}
\put (45,74) {{\scriptsize $S$}}
\put (35,44) {{\scriptsize\blue{$Z$}}}
\put (15,34) {{\small $\mathrm{DL}_{S|Z}$}}
\put (60,41) {{\small\red{$u$}}}
\end{overpic}
\hspace{20pt}
\begin{overpic}[width=2.2in]{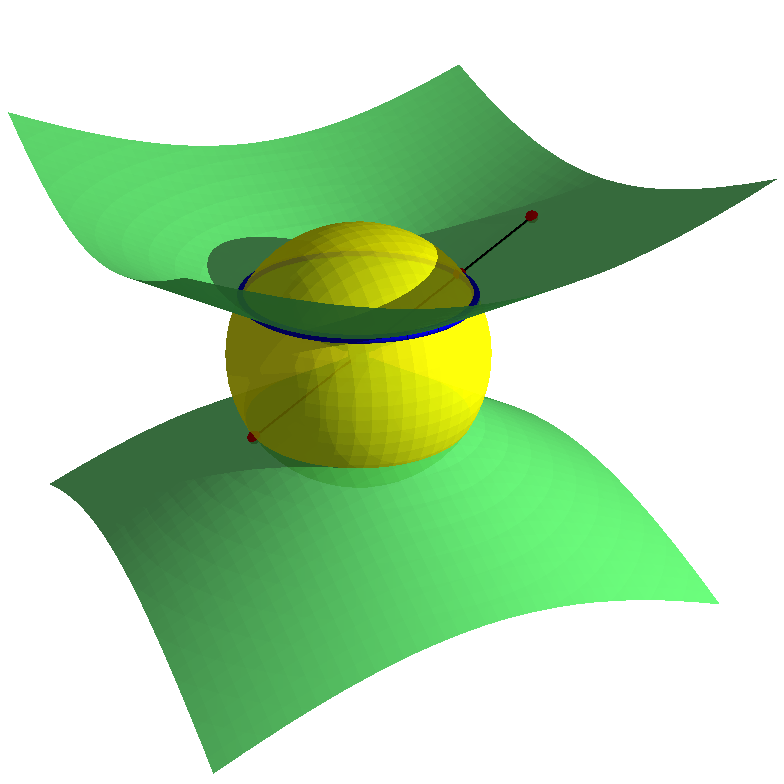}
\put (45,74) {{\scriptsize $S$}}
\put (48,52) {{\scriptsize\blue{$Z$}}}
\put (30,25) {{\small $\mathrm{DL}_{S|Z}$}}
\put (65,74) {{\small\red{$u$}}}
\end{overpic}
\caption{On the left, the green plane $x_3=0$ is the ED data locus of the unit sphere $S$ through the origin given the blue unit circle $Z=V(x_1^2+x_2^2-1,x_3)$. On the right, the green quadric cone $V(x_1^2+x_2^2-3x_3^2)$ is the ED data locus of $S$ given the blue circle $Z=V(x_1^2+x_2^2-1/4,x_3-1/2)$.}\label{fig: ED data locus sphere}
\end{figure}

The varieties $\sE_X$ and $\sE_{X|Z}$ are closely related to each other. More precisely, we have that
\begin{equation}\label{eq: containment ED correspondence Z}
\sE_{X|Z} \subset \sE_X\cap(Z\times V_u)\,.
\end{equation}
It is immediate to verify that the previous inclusion is an equality when $Z$ does not intersect the singular locus of $X$, in particular when $X$ is a smooth variety. Furthermore, the previous inclusion yields another inclusion when projecting both sides to $V_u$:
\begin{equation}\label{eq: containment data loci}
\mathrm{DL}_{X|Z} \subset \overline{\pr_2(\sE_X\cap(Z\times V_u))}\,.
\end{equation}
The right-hand side of \eqref{eq: containment data loci} coincides with the data locus of $Z$ introduced in \cite[\S3]{horobet2022data} and denoted there by $\mathrm{DL}_Z$. On one hand, our notion of ED data locus is clearly inspired by the work just cited. On the other hand, the ED data locus at the left-hand side of \eqref{eq: containment data loci} is always irreducible when $Z$ is. The next example shows a case when the containment in \eqref{eq: containment data loci} is strict.

\begin{example}\label{ex: Cayley}
Consider again Cayley's nodal cubic surface of Example \ref{ex: Cayley intro}.
The real locus of $X$ is shown in Figure \ref{fig: ED data locus Cayley circle}. Its singular locus is
\[
X_{\sing} = \left\{\left(-\frac{1}{2},\frac{1}{2},\frac{1}{2}\right),\left(\frac{1}{2},-\frac{1}{2},\frac{1}{2}\right),\left(\frac{1}{2},\frac{1}{2},-\frac{1}{2}\right),\left(-\frac{1}{2},-\frac{1}{2},-\frac{1}{2}\right)\right\}\,.
\]
Let $Z$ be the circle obtained by intersecting $X$ with the plane $x_1=0$. In particular $Z$ is parametrized by the map $t\mapsto(0,\cos(t)/2,\sin(t)/2)$ for $t\in[0,2\pi]$. The ED data locus $\mathrm{DL}_{X|Z}$ is the sextic surface of equation
\begin{equation}\label{eq: ED data locus Cayley circle}
4\,u_{1}^{2}u_{2}^{4}+8\,u_{1}^{2}u_{2}^{2}u_{3}^{2}-4\,u_{2}^{4}u_{3}^{2}+4\,u_{1}^{2}u_{3}^{4}-4\,u_{2}^{2}u_{3}^{4}+4\,u_{1}u_{2}^{3}u_{3}+4\,u_{1}u_{2}u_{3}^{3}+u_{2}^{2}u_{3}^{2}=0\,.
\end{equation}
Since $J(f)=8(2\,x_{2}x_{3}+x_{1},\,2\,x_{1}x_{3}+x_{2},\,2\,x_{1}x_{2}+x_{3})$, at the points $x\in Z$ the normal line $N_xZ$ is spanned by the vector $(\cos(t)\sin(t),\cos(t),\sin(t))$.
Hence we can write $\mathrm{DL}_{X|Z}$ as the image of the parametrization
\begin{equation}
    (t,s)\mapsto \left(s\cos(t)\sin(t),\frac{\cos(t)}{2}+s\cos(t),\frac{\sin(t)}{2}+s\sin(t)\right)\,,
\end{equation}
where $(t,s)\in[0,2\pi]\times\R$. In this case $\mathrm{DL}_{X|Z}$ coincides with $\overline{\pr_2(\sE_X\cap(Z\times V_u))}$.

Secondly, we let $Z$ be the line contained in $X$ joining the singular points $P_1=(1/2,1/2,-1/2)$ and $P_2=(-1/2,-1/2,-1/2)$. On one hand, the ED data locus $\mathrm{DL}_{X|Z}$ is just the plane $u_1-u_2=0$, whereas $\overline{\pr_2(\sE_X\cap(Z\times V_u))}$ consists of three irreducible components: the variety $\mathrm{DL}_{X|Z}$, and the two quadric surfaces whose equations are
\begin{align}\label{eq: components ED data locus Cayley}
\begin{split}
4\,u_{1}u_{2}-4\,u_{1}u_{3}-4\,u_{2}u_{3}-4\,u_{1}-4\,u_{2}+4\,u_{3}+3 &= 0\\
4\,u_{1}u_{2}+4\,u_{1}u_{3}+4\,u_{2}u_{3}+4\,u_{1}+4\,u_{2}+4\,u_{3}+3 &= 0\,.
\end{split}
\end{align}
They correspond to the normal cones of $X$ at the two singularities $P_1$ and $P_2$, see Figure \ref{fig: ED data locus Cayley line}.\hfill$\diamondsuit$
\end{example}

We recall the construction of \cite[\S2.2]{horobet2022data}, useful for us to compute the degrees of ED data loci in Section \ref{sec: degrees ED data loci}.

\begin{definition}\label{def: normal variety}
Given an affine variety $X\subset V$, the {\em affine conormal variety} of $X$ is
\begin{equation}\label{eq: affine conormal variety}
\sN_X\coloneqq\overline{\{(x,y)\in V\times V\mid\text{$x\in X_{\sm}$ and $y\in N_xX$}\}}\,.
\end{equation}
Furthermore, if $Z$ is a subvariety of $X$, the {\em affine conormal variety} of $X$ {\em relative to} $Z$ is
\begin{equation}\label{eq: affine conormal variety Z}
\sN_{X|Z}\coloneqq\overline{\{(x,y)\in V\times V\mid\text{$x\in X_{\sm}\cap Z$ and $y\in N_xX$}\}}\,.
\end{equation}
\end{definition}
The expressions in \eqref{eq: affine conormal variety} and \eqref{eq: affine conormal variety Z} depend on the choice of $\langle\,,\rangle$, which we are assuming equal to the standard Euclidean bilinear form.
The notation $\sN_X$ is used in \cite[\S5]{DHOST}. It is almost immediate that $\sN_X$ and $\sN_{X|Z}$ are cut out respectively by the ideals
\begin{equation}\label{eq: ideal normal space}
N_X \coloneqq \left(I_X+\left\langle\text{$(c+1)\times(c+1)$ minors of $\begin{pmatrix}y\\J(f)\end{pmatrix}$}\right\rangle\right)\colon\left(I_{X_{\sing}}\right)^\infty
\end{equation}
and $N_{X|Z}$, which is obtained by replacing $I_X$ with $I_Z$ in \eqref{eq: ideal normal space}.
Similarly as in \eqref{eq: containment ED correspondence Z}, we have the containment
\begin{equation}\label{eq: containment normal space Z}
\sN_{X|Z} \subset \sN_X\cap(Z\times V_y)\,.
\end{equation}
Consider the map $\Gamma\colon\sN_X\to V_u$ defined by $\Gamma(x,y)\coloneqq x+y$. Its graph is
\begin{equation}\label{eq: graph Gamma}
\sG(\Gamma)=\overline{\{(x,y,u)\mid\text{$(x,y)\in\sN_X$ and $u=x+y$}\}}\,.
\end{equation}
This variety in $V_x\times V_y\times V_u$ is called joint ED correspondence in \cite[\S5]{DHOST}. We also consider the restriction $\Gamma_Z\coloneqq\Gamma|_{\sN_{X|Z}}$, and its graph $\sG(\Gamma_Z)$. Thus we obtain the two diagrams illustrated below:
\begin{equation}\label{eq: diagram graph Gamma}
\begin{tikzpicture}[node distance={17mm}, thick] 
\node (1) {$\sG(\Gamma)$}; 
\node (2) [below left of=1] {$\sN_X$}; 
\node (3) [below right of=1] {$\sE_X$}; 
\node (4) [below left of=2] {$X\subset V_x$}; 
\node (5) [below right of=2] {$V_y$}; 
\node (6) [below right of=3] {$V_u$}; 
\draw[->] (1) to (2); 
\draw[->] (1) to (3); 
\draw[->] (2) to (4); 
\draw[->] (2) to (5);
\draw[->, draw=blue] (2) -- node[midway, above right, sloped, pos=0.55] {{\scriptsize\blue{$\Gamma$}}} (6);
\draw[->] (3) to (4); 
\draw[->] (3) to (6);
\draw[->] (1) to [out=345, in=100, looseness=1] (6);
\end{tikzpicture}
\quad
\begin{tikzpicture}[node distance={17mm}, thick] 
\node (1) {$\sG(\Gamma_Z)$}; 
\node (2) [below left of=1] {$\sN_{X|Z}$}; 
\node (3) [below right of=1] {$\sE_{X|Z}$}; 
\node (4) [below left of=2] {$Z\subset V_x$}; 
\node (5) [below right of=2] {$V_y$}; 
\node (6) [below right of=3] {$V_u$}; 
\draw[->] (1) to (2); 
\draw[->] (1) to (3); 
\draw[->] (2) to (4); 
\draw[->] (2) to (5); 
\draw[->, draw=blue] (2) -- node[midway, above right, sloped, pos=0.55] {{\scriptsize\blue{$\Gamma_Z$}}} (6);
\draw[->] (3) to (4); 
\draw[->] (3) to (6);
\draw[->] (1) to [out=345, in=100, looseness=1] (6);
\end{tikzpicture}
\end{equation}
All arrows in \eqref{eq: diagram graph Gamma} denote projection maps, except for the blue arrows corresponding to the maps $\Gamma$ and $\Gamma_Z$. The varieties $\sG(\Gamma)$ and $\sG(\Gamma_Z)$ are irreducible because the varieties $\sE_X$ and $\sE_{X|Z}$ are, and the relation between $\sG(\Gamma)$ and $\sG(\Gamma_Z)$ is the following:
\begin{equation}\label{eq: containment graph gamma Z}
\sG(\Gamma_Z) \subset \sG(\Gamma)\cap (Z\times V_y\times V_u)\,.
\end{equation}
The projections $\sG(\Gamma)\to\sE_X$, $\sG(\Gamma_Z)\to\sE_{X|Z}$ are birational with inverse maps given by $(x,u)\mapsto(x,u-x,u)$. Furthermore, the ED data locus $\mathrm{DL}_{X|Z}$ of Definition \ref{def: ED data locus} is also equal to the projection of $\sG(\Gamma_Z)$ onto $V_u$. Instead, the authors of \cite{horobet2022data} define the ED data locus as the closure of the projection onto $V_u$ of the right-hand side in \eqref{eq: containment graph gamma Z}.
We have already discussed that the latter subset may not be irreducible, in particular, it may contain strictly the irreducible variety $\mathrm{DL}_{X|Z}$.

The following result is useful to determine a rational parametrization of $\mathrm{DL}_{X|Z}$, provided that $X$ and $Z$ are rational. Since its proof is very similar to Proposition 5.14 in \cite{DHOST} we only give a brief outline.

\begin{proposition}\label{prop: rationality ED data loci}
Let $X$ be a rational variety. Let $Z\subset X$ be a rational subvariety. Then the conditional ED correspondence of $X$ given $Z$ is rational.
\end{proposition}
\begin{proof}
Let, $m=\dim(X)$, $c=\codim(X)=N+1-m$, $d=\dim(Z)$ and consider the map $\varphi\colon\C^m\to V$ that parametrizes $X$, so $X=\overline{\image(\varphi)}$. The components of $\varphi$ are rational functions in $t=(t_1,\ldots,t_m)$.
We denote by $J\varphi$ the jacobian matrix of $\varphi.$ 
The normal space $N_{\varphi(t)}X$ is parametrized by a basis $\{\beta_1(t),\ldots,\beta_c(t)\}$ of the left kernel of $J\varphi$, where each vector $\beta_i$ is a vector of rational functions in $t$.
Now consider $W=\varphi^{-1}(Z)$, which is also rational.
So there exists a rational parametrization $\psi\colon\C^d\to\C^m$ of $W$, namely $W=\overline{\image(\psi)}$. The components of $\psi$ are rational functions in $v=(v_1,\ldots,v_d)$.
Finally, we introduce the rational function
\[
\theta\colon\C^d\times\C^c\to V_x\times V_u\,,\theta(v,w)=\left(\varphi(\psi(v)),\sum_{i=1}^cw_i\beta_i(\varphi(\psi(v)))\right)\,.
\]
The previous is a rational parametrization of $\sE_{X|Z}$.
\end{proof}

\begin{example}\label{ex: smooth quadric 2}
Consider the smooth affine surface $X=V(x_1-x_2x_3)\subset\C^3$, and let $Z$ be the line $V(x_1,x_2)$ contained in $X$.
In particular $X$ and $Z$ are parametrized by the maps $\varphi\colon\C^2\to\C^3$, $(t_1,t_2)\mapsto(t_1t_2,t_1,t_2)$ and $\psi=\varphi(0,t)\colon\C^1\to\C^3$, $t\mapsto(0,0,t)$. Then
\[
J\varphi(t_1,t_2)=
\begin{pmatrix}
t_2 & t_1 \\
1 & 0 \\
0 & 1
\end{pmatrix}\,
\]
and the left-kernel of $J\varphi$ is spanned by the vector $\beta(t_1,t_2)=(1,-t_2,-t_1)$.
Therefore, the variety $\sE_{X|Z}$ is parametrized by the map $(t,s)\mapsto(\varphi(0,t),\varphi(0,t)+s\beta(0,t))=((0,0,t),(s,-st,t))$.
Hence $\sE_{X|Z}$ is the surface defined by the ideal $\langle x_2,x_1,u_1u_3+u_2,x_3-u_3\rangle\subset\C[x_1,x_2,x_3,u_1,u_2,u_3]$.
Eliminating the variables in $\{x_1,x_2,x_3\}$, one gets the ideal $\langle u_1u_3+u_2\rangle$ which defines the ED data locus $\mathrm{DL}_{X|Z}$. We obtain the smooth quadric hypersurface in Figure \ref{fig: ED data locus quadric}, that is ruled by the family of normal lines $N_xX$ at the points $x\in Z$.\hfill$\diamondsuit$
\end{example}

\begin{figure}[htbp]
\begin{overpic}[width=2.5in]{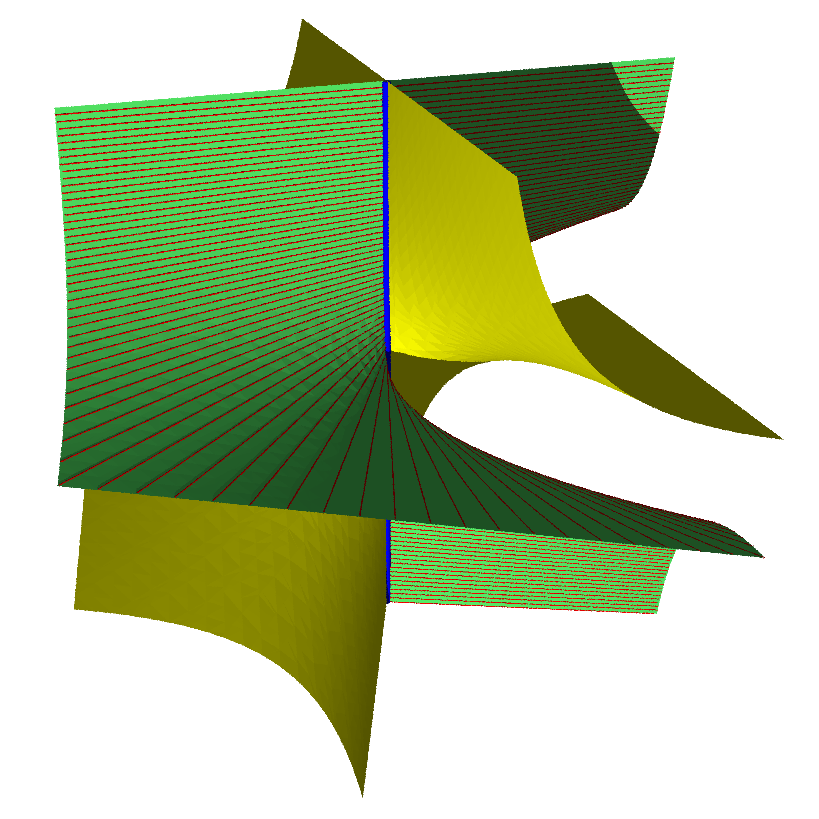}
\put (12,28) {{\scriptsize $X$}}
\put (48,70) {{\scriptsize\blue{$Z$}}}
\put (15,75) {{\small $\mathrm{DL}_{X|Z}$}}
\end{overpic}
\caption{The yellow quadric surface $X$ and the green ED data locus $\mathrm{DL}_{X|Z}$ given the line $Z\subset X$, ruled by the red lines normal to $X$.}\label{fig: ED data locus quadric}
\end{figure}

We conclude this section with a formula providing an upper bound for the product $\deg(\mathrm{DL}_{X|Z})\cdot\mathrm{EDD}(X|Z)$, in terms of the degrees of the generators of the ideal of $X$.
We do not provide proof of the next statement. Indeed, the argument needed in the proof follow from Theorem \ref{thm: degree data loci complete intersection} and arguments as in \cite[\S6]{DHOST}.

\begin{proposition}\label{prop: upper bound degree data loci affine varieties}
Let $X\subset V$ be an affine variety of codimension $c$ that is cut out by the polynomials $f_1,\ldots,f_c,\ldots,f_s$ of degrees $d_1\ge\cdots\ge d_c\ge\cdots\ge d_s$. Let $Z\subset X$ be a subvariety of dimension $d$, and assume that $X$ is ED regular given $Z$.
Then
\begin{equation}
\deg(\mathrm{DL}_{X|Z})\cdot\mathrm{EDD}(X|Z)\le \deg(Z)\sum_{i_1+\cdots+i_c\le d}(d_1-1)^{i_1}\cdots(d_c-1)^{i_c}\,.
\end{equation}
Equality holds when $X$ is a generic complete intersection of codimension $c$.
\end{proposition}

\section{Euclidean Distance data loci of projective varieties}\label{sec: projective}

In the previous section, we developed a theory of ED data loci for affine varieties and subvarieties. In this section we restrict to affine varieties $X$ and $Z$ that are cones in $V\cong\C^{N+1}$ through the origin, namely $x\in X$ if and only if $\lambda\,x\in X$ for all $\lambda\in\C$. Equivalently, we work with projective varieties in $\PP(V)=\PP^N$.
A crucial role in this section is played by the isotropic quadric in $V$, namely by the projective quadric $Q\subset\PP^N$ defined by the equation $q(x)=0$. We recall that $q$ is the quadratic form associated to the real inner product $\langle\,,\rangle$ fixed at the beginning of Section \ref{sec: affine ED data loci}. Up to a linear change of coordinates in $V$, we are assuming that $q(x)=x_0^2+\cdots+x_N^2$.

Consider the ideal $I_{X,u|Z}$ in \eqref{eq: ideal ED correspondence Z}, which defines the conditional ED correspondence of $X$ given $Z$.
To take advantage of the homogeneous equations of $X$ and $Z$, and of the symmetry of the equation of $Q$, we replace the ideal \eqref{eq: ideal ED correspondence Z} with the following homogeneous ideal in $\R[x_0,\ldots,x_N]$:

\begin{equation}\label{eq: ideal projective ED correspondence Z}
\left(I_Z+\left\langle\text{$(c+2)\times(c+2)$ minors of $\begin{pmatrix}u\\x\\J(f)\end{pmatrix}$}\right\rangle\right)\colon\left(I_{X_{\sing}}\cdot\langle x_0^2+\cdots+x_N^2\rangle\right)^\infty\,.
\end{equation}

In \cite[\S4]{DHOST}, the authors introduced the projective ED correspondence $\sPE_X$ of $X$, as the closure of the image of $\sE_X\cap[(V_x\setminus\{0\})\times V_u]$ under the map
\begin{equation}\label{eq: map affine-projective}
(V_x\setminus\{0\})\times V_u\to\PP_x^N\times V_u\,,\quad (x,u)\mapsto([x],u)\,.
\end{equation}
Similarly we introduce the following.  

\begin{definition}\label{def: proj ED correspondence Z}
The {\em projective conditional ED correspondence of $X$ given $Z$}, denoted by $\sPE_{X|Z}$, is the closure of the image of $\sE_{X|Z}\cap[(V_x\setminus\{0\})\times V_u]$ under the map \eqref{eq: map affine-projective}.
\end{definition}

The incidence variety $\sPE_{X|Z}$ comes equipped with the two projections
\begin{equation}\label{eq: diagram projective ED correspondence Z}
\begin{tikzcd}
& \sPE_{X|Z} \arrow{dl}[swap]{\pr_1} \arrow{dr}{\pr_2} & \\
\PP_x^N=\PP(V_x) & & V_u\,.
\end{tikzcd}
\end{equation}
and satisfies the following properties, which follow similarly as in \cite[Theorem 4.4]{DHOST}.

\begin{proposition}\label{prop: projective ED correspondence Z irreducible}
Let $Z\subset X$ be irreducible affine cones in $V_x\cong\C^{N+1}$, where $\dim(X)=n+1$ and $\dim(Z)=d+1$. Assume that $Z\not\subset Q\cup X_{\sing}$.
Then $\sPE_{X|Z}\subset\PP_x^N\times V_u$ is irreducible of dimension $N+1-n+d$.
It is the zero set of the ideal \eqref{eq: ideal projective ED correspondence Z}.
Its projection onto the first factor $\PP^N$ is a vector bundle over $Z\setminus(X_{\sing}\cup Q)$ of rank $N-n+1$.
The projection of $\sPE_{X|Z}$ onto $V_u$ is irreducible and coincides with $\mathrm{DL}_{X|Z}$.
\end{proposition}

From the previous result, we can also derive that
\begin{equation}\label{eq: def projective ED data loci}
\mathrm{DL}_{X,Z}=\overline{\pr_2(\sPE_{X|Z})}=\overline{\{u\in V\mid\text{$\exists\,[x]\in Z\setminus(X_{\sing}\cup Q)$ such that $x$ is critical for $d_{X,u}$}\}}\,,
\end{equation}
where in the previous equation $X$, $Z$ and $Q$ are regarded as projective varieties.
It is immediate that $\mathrm{DL}_{X,Z}$ is an affine cone in $V_u$, provided that $X$ and $Z$ are affine cones in $V_x$.

We slightly adapt the notion of ED regularity of Definition \ref{def: conditional ED regularity} in the case of projective varieties.
\begin{definition}\label{def: projective conditional ED regularity}
Let $Z\subset X$ be projective varieties in $\PP(V)$, seen as affine cones in $V$. We say that $X$ is {\em ED regular given $Z$} if $Z\not\subset X_{\sing}\cup Q$ and if $\mathrm{DL}_{X|Z}$ is not contained in the proper subvariety of data points $u$ such that $\dim(\pr_2(\pr_1^{-1}(u)))>0$, where $\pr_1$ and $\pr_2$ are the projections of $\sPE_X$ onto the factors $\PP_x^N$ and $V_u$.
\end{definition}

The following is an immediate consequence of Proposition \ref{prop: projective ED correspondence Z irreducible}.

\begin{proposition}\label{prop: proj data locus is irreducible}
Let $Z\subset X$ be irreducible affine cones in $V\cong\C^{N+1}$, where $\dim(X)=n+1$ and $\dim(Z)=d+1$. Assume that $X$ is ED regular given $Z$. Then $\mathrm{DL}_{X|Z}$ is irreducible of dimension $N+1-n+d$.
\end{proposition}

The nondegenerate bilinear form $\langle\,,\rangle$ induces an isomorphism $\varphi\colon V\to V^*$ defined by $\varphi(x)\coloneqq\langle x,\cdot\rangle\in V^*$. At the level of projective spaces, the previous isomorphism yields a notion of orthogonality, or polarity, between points and hyperplanes in $\PP^N$. More precisely, if $[p]=[p_0\colon\cdots\colon p_N]$ is a point in $\PP^N$, the polar hyperplane to $P$ is $P^\perp\coloneqq \langle p,x\rangle=0$.

Let $X$ be an affine cone in $V\cong\C^{N+1}$, or equivalently a projective variety in $\PP(V)=\PP^N$. We have already introduced in Definition \ref{def: normal variety} the variety $\sN_X$. Now consider the map
\begin{equation}\label{eq: map affine-projective 2}
\psi\colon(V_x\setminus\{0\})\times(V_y\setminus\{0\})\longrightarrow\PP_x^N\times\PP_y^N\,,\quad \psi(x,y)=([x],[y])\,,
\end{equation}
and define
\begin{equation}\label{eq: def PN X}
\sPN_X\coloneqq\overline{\psi(\sN_X\cap[(V_x\setminus\{0\})\times(V_y\setminus\{0\})])}\,.  
\end{equation}
It is easy to see that the map $\varphi$ induces an isomorphism between the variety $\sPN_X\subset\PP_x^N\times\PP_y^N$ and the conormal variety $W_X\subset\PP^N\times(\PP^N)^\vee$.
This isomorphism lets us see the variety $X$ and its dual variety $X^\vee$ in the same projective space $\PP^N$.
More precisely, we define $X^\perp$ as the image of $\sPN_X$ under the projection onto $\PP_y^N$. Then $X^\perp\cong X^\vee$, and the isomorphism depends on the map $\varphi$.
In the same vein, we define
\begin{equation}\label{eq: def PN X Z}
\sPN_{X|Z}\coloneqq\overline{\psi(\sN_{X,Z}\cap[(V_x\setminus\{0\})\times(V_y\setminus\{0\})])}\,.  
\end{equation}
Then $\sPN_{X,Z}$ is isomorphic to the relative conormal variety $W_{X,Z}$. We also define $X_Z^\perp$ as the image of $\sPN_{X,Z}$ under the projection onto $\PP_y^N$, and we have the isomorphism $X_Z^\perp\cong X_Z^\vee$. As an affine cone in $V$, we can rewrite $X_Z^\perp$ as
\begin{equation}\label{eq: alternative form relative perp}
X_Z^\perp = \overline{\bigcup_{x\in X_{\sm}\cap Z}N_xX}\,,
\end{equation}
Finally, we also consider the map
\begin{equation}\label{eq: map affine-projective 3}
(V_x\setminus\{0\})\times(V_y\setminus\{0\})\times V_u\to\PP_x^N\times\PP_y^N\times V_u\,,\quad (x,y,u)\mapsto([x],[y],u)\,,
\end{equation}
and we define the varieties $\sPG(\Gamma)$ and $\sPG(\Gamma_Z)$ by taking the closures of the images under the previous map of the intersections of $\sG(\Gamma)$ and $\sG(\Gamma_Z)$, respectively, with $(V_x\setminus\{0\})\times(V_y\setminus\{0\})\times V_u$.
The variety $\sPG(\Gamma)$ coincides with the {\em projective joint ED correspondence of $X$} defined in \cite[\S5]{DHOST}, therefore we may address $\sPG(\Gamma_Z)$ as the {\em projective joint conditional ED correspondence of $X$ given $Z$}.
We may also consider the following projective analog of the diagrams in \eqref{eq: diagram graph Gamma}:
\begin{equation}\label{eq: diagram graph projective Gamma}
\begin{tikzpicture}[node distance={17mm}, thick] 
\node (1) {$\sPG(\Gamma)$}; 
\node (2) [below left of=1] {$\sPN_X$}; 
\node (3) [below right of=1] {$\sPE_X$}; 
\node (4) [below left of=2] {$X\subset\PP_x^N$}; 
\node (5) [below right of=2] {$\PP_y^N$}; 
\node (6) [below right of=3] {$V_u$}; 
\draw[->] (1) to (2); 
\draw[->] (1) to (3); 
\draw[->] (2) to (4); 
\draw[->] (2) to (5);
\draw[->] (3) to (4); 
\draw[->] (3) to (6);
\draw[->] (1) to [out=345, in=100, looseness=1] (6);
\end{tikzpicture}
\quad
\begin{tikzpicture}[node distance={17mm}, thick] 
\node (1) {$\sPG(\Gamma_Z)$}; 
\node (2) [below left of=1] {$\sPN_{X,Z}$}; 
\node (3) [below right of=1] {$\sPE_{X|Z}$}; 
\node (4) [below left of=2] {$Z\subset\PP_x^N$}; 
\node (5) [below right of=2] {$\PP_y^N$}; 
\node (6) [below right of=3] {$V_u$}; 
\draw[->] (1) to (2); 
\draw[->] (1) to (3); 
\draw[->] (2) to (4); 
\draw[->] (2) to (5);
\draw[->] (3) to (4); 
\draw[->] (3) to (6);
\draw[->] (1) to [out=345, in=100, looseness=1] (6);
\end{tikzpicture}
\end{equation}

The next lemma is an immediate consequence of the previous considerations and Theorem \ref{thm: reflexivity}.

\begin{lemma}\label{lem: relative bi-duality same space}
Let $Z\subset X$ be irreducible projective varieties such that $X$ is dual regular relative to $Z$.
There exists an open set of points $(x,y)\in\sPN_{X,Z}$ (seen as an affine cone in $V_x\times V_y$) such that $x$ is a smooth point of $X$ in $Z$, $y$ is a smooth point of $X^\perp$ in $X_Z^\perp$ and the following properties hold:
\[
y\perp T_xX\quad\text{and}\quad x\perp T_yX^\perp\,.
\]
\end{lemma}

\begin{proposition}\label{prop: ideal projective joint relative}
Let $Z\subset X\subset V$ be irreducible affine cones, and consider the variety $X_Z^\perp$. Assume that neither $Z$ nor $X_Z^\perp$ are contained in the isotropic quadric $Q$, where we assume that $q(x)=x_0^2+\cdots+x_N^2$.
Then $\sPG(\Gamma_Z)$ is an irreducible variety in $\PP_x^N\times\PP_y^N\times V_u$ of dimension $N+1-\codim_X(Z)$.
It is the zero set of the tri-homogeneous ideal
\begin{equation}\label{eq: ideal projective joint relative}
\left(N_{X,Z}+\left\langle\text{$3\times 3$ minors of $\begin{pmatrix}u\\x\\y\end{pmatrix}$}\right\rangle\right)\colon\langle q(x)\cdot q(y)\rangle^\infty\,.
\end{equation}
\end{proposition}
The proof of Proposition \ref{prop: ideal projective joint relative} is essentially the same as the one of \cite[Proposition 5.3]{DHOST}. In our case, we apply Lemma \ref{lem: relative bi-duality same space}.

We conclude this section by investigating more on the similarity between identities \eqref{eq: alternative description of ED data locus} and \eqref{eq: alternative form relative perp}. Indeed, we have the following containment relation. Notice that \cite[Theorem 1]{horobet2017data} proves a similar containment for the different notion of data singular locus.
\begin{lemma}\label{lem: relative perp contained in data locus}
Let $X$ and $Z$ be affine cones in $V$. Then $X_Z^\perp\subset\mathrm{DL}_{X|Z}$.
\end{lemma}
\begin{proof}
If $Z\subset X_{\sing}$, the containment is trivial, so we assume that there exists $x\in X_{\sm}\cap Z$. Clearly, the normal space $N_xX$ is contained in the union of the affine linear spaces $\lambda\,x+N_{\lambda\,x}X$ for all $\lambda\in\C$, and $N_{\lambda\,x}X=N_xX$ for all $\lambda\neq 0$, because $X$ is a cone. This yields the containment
\begin{align}\label{eq: chain containments relative perp data locus}
\begin{split}
\bigcup_{x\in X_{\sm}\cap Z}N_xX &\subset \bigcup_{x\in X_{\sm}\cap Z}\bigcup_{\lambda\in\C}(\lambda\,x+N_xX) = \bigcup_{x\in X_{\sm}\cap Z}\overline{\bigcup_{\lambda\neq 0}(\lambda\,x+N_xX)}\\
&\subset \overline{\bigcup_{x\in X_{\sm}\cap Z}\bigcup_{\lambda\neq 0}(\lambda\,x+N_xX)} = \overline{\bigcup_{y\in X_{\sm}\cap Z}(y+N_yX)} = \mathrm{DL}_{X|Z}\,.
\end{split}
\end{align}
In the second last equality, we used the fact that $X$ and $Z$ are affine cones. By taking closure on the left-hand side of \eqref{eq: chain containments relative perp data locus} and using \eqref{eq: alternative form relative perp}, we get the desired containment $X_Z^\perp\subset\mathrm{DL}_{X|Z}$.
\end{proof}

We can also say more about the relationship between the varieties $X_Z^\perp$ and $\mathrm{DL}_{X|Z}$ when $Z\subset X$ are affine cones in $V$. For simplicity, we consider the case when $\mathrm{def}(X)=\mathrm{def}(X,Z)$, seeing $X$ and $Z$ as varieties in $\PP(V)$. Applying Propositions \ref{prop: formulas relative codim X dual X Z dual}(3)-(4) and \ref{prop: data locus is irreducible}, it is almost immediate to see that
\begin{equation}
    \dim(X_Z^\perp)=\dim(\mathrm{DL}_{X|Z})-\mathrm{def}(X)-1\,.
\end{equation}
In order to interpret the previous identity, observe the role played by contact loci in $\PP(V)$ is replaced in $V$ by the linear spaces
\begin{equation}
    L_y\coloneqq\overline{\{x\in X_{\sm}\mid y\in N_xX\}}\,,
\end{equation}
where $y\in N_xX$ for some fixed $x\in X_{\sm}$. One may verify that $L_y$ is isomorphic to the affine cone over the contact locus $\mathrm{Cont}(H,X)$ of $X\subset\PP(V)$ with respect to the hyperplane $H$ induced by $y$. Furthermore, since $X$ is a cone, at least the line spanned by $x$ is contained in $L_y$, hence $\dim(L_y)\ge 1$. But the dimension of $L_y$ increases exactly as the value of $\mathrm{def}(X)$, more explicitly $\dim(L_y)=\mathrm{def}(X)+1$.

\section{Degrees of projective ED data loci}\label{sec: degrees ED data loci}

Our next goal is to provide a formula for the degree of $\mathrm{DL}_{X|Z}$ that can be applied when neither $X$ nor $Z$ are generic complete intersection varieties or when $Z=X\cap Y$ for some generic complete intersection variety $Y$. We require that the variety $\sPN_{X,Z}$ introduced in \eqref{eq: def PN X Z} is disjoint from the diagonal $\Delta(\PP^N)$ of $\PP^N\times\PP^N$.

Due to the isomorphism $\sPN_{X,Z}\cong W_{X,Z}$, we have that $\dim(\sPN_{X,Z})=N-1-\codim_X(Z)$. Furthermore, its class in the Chow group $A^*(\PP^N\times\PP^N)$ can be written similarly as in \eqref{eq: def relative multidegrees}: 
\begin{equation}\label{eq: class normal space Z}
[\sPN_{X,Z}]=\delta_0(X,Z)h^N(h')^{n-d+1}+\cdots+\delta_{N-1-n+d}(X,Z)h^{n-d+1}(h')^N\,,
\end{equation}
where we recall that $d=\dim(Z)$ and with the only clarification that, in this case, $h'$ denotes the class $[\PP^N\times H']$ for some hyperplane $H'$ of $\PP^N$, differently from \eqref{eq: notation hyperplane classes}.
We already know from Proposition \ref{prop: properties relative polar ranks} the relations $\delta_{d-i}(X,Z)=\mu_i(X,Z)$, where $\mu_i(X,Z)$ is the $i$th polar degree of $X$ relative to $Z$ given in Definition \ref{def: relative polar degrees}.
We are ready to prove Theorem \ref{thm: degree data locus no genericity X, Z}, which can be interpreted as a ``conditional'' version of \cite[Theorem 5.4]{DHOST}.

\begin{proof}[Proof of Theorem \ref{thm: degree data locus no genericity X, Z}]
The argument of the proof is identical to \cite[Theorem 5.4]{DHOST}. We shortly outline the main ingredients for the readers' convenience.

First, one needs to consider the variety $\sPG(\Gamma_Z)\subset\PP_x^N\times\PP_y^N\times V_u$ introduced after \eqref{eq: map affine-projective 3}. By Proposition \ref{prop: ideal projective joint relative}, it is irreducible of dimension $N+1-n+d$, where $n=\dim(X)$ and is cut out by the ideal in \eqref{eq: ideal projective joint relative}. Secondly, we define $\Lambda\coloneqq\{([x],[y],u)\in\PP_x^N\times\PP_y^N\times V_u\mid\dim(\langle x,y,u\rangle)\le 2\}$.
Using that $\sPN_{X,Z}\cap\Delta(\PP^N)=\emptyset$, one verifies that $\sPG(\Gamma_Z)=(\sPN_{X,Z} \times V_u)\cap \Lambda$.

It follows that $\mathrm{EDD}(X|Z)\cdot\deg(\mathrm{DL}_{X|Z})$ equals the length of a generic fiber of the projection $\pi\colon(\sPN_{X,Z} \times V_u)\cap \Lambda\to V_u$.
The generic fiber $\pi^{-1}(u)$ consists of simple points because of generic smoothness, and this fiber coincides scheme-theoretically with $\sPN_{X,Z}\cap \Lambda_u$, where $\Lambda_u$ is the fiber in $\Lambda$ over $u$.
The cardinality of this intersection is the coefficient of $h^N(h')^N$ in the class $[\sPN_{X,Z}\cap \Lambda_u]=[\sPN_{X,Z}]\cdot[\Lambda_u]\in A^*(\PP^N\times\PP^N)$. The determinantal variety $\Lambda_u$ has codimension $N-1$ and
\begin{equation}\label{eq: class lambda_u}
    [\Lambda_u] = \sum_{i=0}^{N-1}h^{N-1-i}(h')^i\,.
\end{equation}
By computing modulo $\langle h^{N+1},(h')^{N+1}\rangle$, we find
\[
[\sPN_{X,Z}]\cdot[\Lambda_u] = \sum_{i=0}^{N-1-n+d}\delta_i(X,Z)\,h^{N-i}(h')^{n-d+1+i}\cdot\sum_{i=0}^{N-1}h^{N-1-i}(h')^i = \left(\sum_{i=0}^d\delta_i(X,Z)\right)h^N(h')^N\,.
\]
This establishes the desired identity.\qedhere
\end{proof}

The following is an immediate consequence of Proposition \ref{prop: properties relative polar ranks} and Theorem \ref{thm: degree data locus no genericity X, Z}.

\begin{corollary}\label{corol: degree data locus no genericity X, Z with relative polar degrees}
Consider the assumptions of Theorem \ref{thm: degree data locus no genericity X, Z}. If $X$ and $Z$ are smooth, then
\begin{equation}\label{eq: degree data locus no genericity X, Z with relative polar degrees}
\mathrm{EDD}(X|Z)\deg(\mathrm{DL}_{X|Z})=\sum_{i=0}^d\mu_i(X,Z)\,.
\end{equation}
\end{corollary}

\begin{remark}\label{rmk: Kalman varieties data loci}
The hypotheses of Theorem \ref{thm: degree data locus no genericity X, Z} are satisfied whenever $X$ intersects the isotropic quadric $Q$ transversally along the points of $Z\subset X$. Changing our perspective, having fixed $X$ and $Z$, there exists an open dense subset of the space $\PP(S^2V)$ of quadratic forms $q$ on $V$ whose associated isotropic quadric $Q$ is transversal to $X$. Instead, for a specific choice of a quadratic form $q$, the right-hand side of \eqref{eq: degree data locus no genericity X, Z} might be strictly larger than $\mathrm{EDD}(X|Z)\cdot\deg(\mathrm{DL}_{X|Z})$. This happens when considering the so-called generalized Kalman varieties of matrices and tensors. We briefly recall their definition, and we refer to \cite{ottaviani2013matrices,OSh,shahidi2021degrees} for further details. We concentrate on nonsymmetric tensors for simplicity.

Consider the complex tensor space $V=V_1\otimes\cdots\otimes V_k$ and the Segre product $X\subset\PP(V)$ defined in Remark \ref{rmk: relative duality tensor subspaces}. Suppose the real tensor space $V^\mR$ comes equipped with an inner product $\langle\,,\rangle^\mR$ with associated quadratic form $q^\mR\in S^2V^\mR$. Given a real tensor $u\in V^\mR$, a {\em best rank-one approximation} of $u$ is a global minimizer $x\in X^\mR$ of the distance function $d_{X,u}(x)=\sqrt{q^\mR(u-x)}$. Hence, we are clearly in the setting of ED optimization. However, in this context, it is natural to fix a ``product'' quadratic form $q^\mR$ on $V^\mR$, that is, an element $q^\mR=q_1^\mR\otimes\cdots\otimes q_k^\mR$, where each factor $q_i^\mR\in S^2V_i^\mR$ defines an inner product on $V_i^\mR$. Such forms $q^\mR$ are also called {\em Frobenius} or {\em Bombieri-Weyl} quadratic forms. In such a Euclidean space, it turns out that the ED degree of $X$ is considerably smaller than the sum of its polar degrees. This is because the intersection between $X$ and the isotropic quadric $Q$ is rather special. The ED degree of $X$ is computed in \cite[Theorem 1]{friedland2014number}, while the work \cite{kozhasov2023minimal} studies the minimality of the ED degree of $X$ with respect to the Frobenius inner product.
For every $i\in[k]$ we let $Z_i\subset\PP(V_i)$ be an irreducible subvariety, and we call $Z$ the Segre embedding of $Z_1\times\cdots\times Z_k$ in $\PP(V)$. The generalized Kalman variety introduced in \cite[Definition 15]{shahidi2021degrees} coincides with the data locus $\mathrm{DL}_{X|Z}$, and the following formula provides its degree.    
\end{remark}

\begin{theorem}{\cite[Theorem 1]{shahidi2021degrees}}\label{thm: Kalman}
For every $i\in[k]$, let $Z_i\subset\PP(V_i)=\PP^{n_i}$ be an irreducible projective variety of codimension $\delta_i$. Let $X$ and $Z\subset X$ be the products $\prod_{i=1}^k\PP^{n_i}$ and $\prod_{i=1}^k Z_i$ embedded in $\PP(V)=\PP(V_1\otimes\cdots\otimes V_k)$.
We assume each $Z_i$ is not contained in the isotropic quadric $Q_i\subset\PP(V_i)$.
Let $\delta=(\delta_1,\ldots,\delta_k)$. The data locus $\mathrm{DL}_{X|Z}$ is connected of codimension $\delta\coloneqq\sum_{i=1}^k\delta_i$ in $\PP(V)$. Its degree is $d(n,\delta)\prod_{i=1}^k\deg(Z_i)$, where $d(n,\delta)$ is the coefficient of the monomial $h^{\delta}\prod_{i=1}^kt_{i}^{n_i-\delta_i}$ in the polynomial
\[
\prod_{i=1}^{k}\frac{(\widehat{t_i}+h)^{n_i+1}-t_i^{n_i+1}}{(\widehat{t_i}+h)-t_i}\,,\quad\widehat{t_i}\coloneqq\sum_{j\neq i}t_j\,.
\]
\end{theorem}

\begin{example}\label{ex: degree Kalman}
We compare the formulas in Theorems \ref{thm: degree data locus no genericity X, Z} and \ref{thm: Kalman} in a specific example. Let $k=2$, $(n_1,n_2)=(1,2)$, so that $X=\PP^1\times\PP^2\subset\PP(\C^2\otimes\C^3)\cong\PP^5$. Fix a generic line $L\subset\PP^2$ and consider $Z=\PP^1\times L$ as a subvariety of $X$. In particular $(\delta_1,\delta_2)=(0,1)$ and $\delta=\delta_1+\delta_2=1$.
By Proposition \ref{prop: double data Segre strictly contained}, we have $\mathrm{EDD}(X|Z)=1$.
On one hand, applying Theorem \ref{thm: Kalman} we obtain that the degree of the data locus $\mathrm{DL}_{X|Z}$ with respect to a Frobenius quadratic form on $\R^2\otimes\R^3$ equals the coefficient of $ht_1t_2$ in the product
\[
(t_1+t_2+h)[(t_1+h)^2+(t_1+h)t_2+t_2^2]=\cdots+4ht_1t_2+\cdots
\]
Therefore $\deg(\mathrm{DL}_{X|Z})=4$. On the other hand, we verified symbolically on \verb|Macaulay2| that
\[
[W_{X,Z}]=3h^5(h')^2+3h^4(h')^3+2h^3(h')^4\,.
\] 
Hence $\deg(\mathrm{DL}_{X|Z})=4<8=3+3+2$, and $8$ is the degree of $\mathrm{DL}_{X|Z}$ expected by Theorem \ref{thm: degree data locus no genericity X, Z}. This is because, if $Q\subset\PP(V)$ is the Frobenius isotropic quadric induced by the isotropic quadrics $Q_1\subset\PP^1$, $Q_2\subset\PP^2$ on the two factors of $X$, then $(X\cap Q)_{\sing}$ coincides with the Segre embedding of $Q_1\times Q_2$, and $(\PP^1\times L)\cap(Q_1\times Q_2)=Q_1\times(Q_2\cap L)\neq\emptyset$. In conclusion, in this example, the hypotheses of Theorem \ref{thm: degree data locus no genericity X, Z} are not satisfied.\hfill$\diamondsuit$
\end{example}

The formula \eqref{eq: degree data locus no genericity X, Z} can be simplified when $Z$ is the intersection between $X$ and a generic complete intersection variety $Y$, as we see in the following result. 

\begin{corollary}\label{corol: degrees data loci Z intersection X Y}
Consider the assumptions of Theorem \ref{thm: degree data locus no genericity X, Z}. Assume also that $Z=X\cap Y$ for some generic complete intersection variety $Y$ of codimension $c\le n=\dim(X)$. Then $\mathrm{EDD}(X|Z)=1$ and
\begin{equation}\label{eq: degrees data loci}
\deg(\mathrm{DL}_{X|Z})=\deg(Y)\sum_{i=\max\{\mathrm{def}(X),c\}}^n\delta_i(X)\,.
\end{equation}
In particular, if $c\le\mathrm{def}(X)$, then $\deg(\mathrm{DL}_{X|Z})=\deg(Y)\mathrm{EDD}(X)$.
If additionally $Y$ is linear, then $\deg(\mathrm{DL}_{X|Z})=\mathrm{EDD}(Z)$.
\end{corollary}
\begin{proof}
The identity $\mathrm{EDD}(X|Z)=1$ follows by Proposition \ref{prop: double data strictly contained}.
The projective version of the containment in \eqref{eq: containment normal space Z} is $\sPN_{X,Z} \subset \sPN_X\cap (Z\times \PP_y^N)$. Our genericity assumptions imply that the previous containment is, in fact, an equality. More precisely, we can say that $\sPN_{X,Z} = \sPN_X\cap (Y\times \PP_y^N)$.
Furthermore, the genericity of $Y$ allows us to say that
\[
\begin{gathered}
[\sPN_X\cap (Y\times \PP_y^N)] = [\sPN_X]\cdot[Y\times \PP_y^N] = \left(\sum_{i=\mathrm{def}(X)}^n\delta_i(X)\,h^{N-i}(h')^{i+1}\right)\cdot\deg(Y)\,h^c\\
= \deg(Y)\sum_{i=\mathrm{def}(X)}^n\delta_i(X)\,h^{N-i+c}(h')^{i+1} = \deg(Y)\sum_{i=\max\{\mathrm{def}(X),c\}}^n\delta_i(X)\,h^{N-i+c}(h')^{i+1}\,,
\end{gathered}
\]
where in the last equality we used that $h^{N-i+c}=0$ for all $i<c$. The rest of the proof of \eqref{eq: degrees data loci} follows by Theorem \ref{thm: degree data locus no genericity X, Z}.
In particular, if $c\le\mathrm{def}(X)$, using that $\delta_i(X)=0$ for all $i<\mathrm{def}(X)$ we conclude that $\sum_{i=c}^n\delta_i(X)=\mathrm{EDD}(X)$.

Finally, we assume that $Y$ is linear. A result of Piene \cite{piene1978polar} says that $\delta_i(Z)=\delta_{i+c}(X)$ for all $i$. Using these identities and applying \eqref{eq: degrees data loci}, we have that
\[
\deg(\mathrm{DL}_{X|Z})=\sum_{i=c}^n\delta_i(X)=\sum_{j=0}^{n-c}\delta_{j+c}(Z)=\sum_{i=0}^d\delta_i(Z)\,,
\]
and the last quantity is equal to $\mathrm{EDD}(Z)$, where we are using the fact that $\sPN_Z\cap\Delta(\PP^N)=\emptyset$.
\end{proof}

Let $X$ be a smooth projective variety of dimension $n$.
We denote by $c_i(X)$ the $i$-th Chern class of the tangent bundle $\sT_X$. From \cite[p. 150]{holme1988geometric}, we have the following linear relation between Chern classes of $X$ and multidegrees of $W_X$:
\begin{equation}\label{eq: polar vs Chern}
\delta_i(X)=\sum_{j=i}^n(-1)^{n-j}\binom{j+1}{i+1}\deg(c_{n-j}(X))\quad\forall\,i\in\{0,\ldots,n\}\,.
\end{equation}

\begin{proposition}\label{prop: degrees data loci Chern classes}
Consider the assumptions of Theorem \ref{thm: degree data locus no genericity X, Z}. Assume additionally that $X$ is smooth and that $Z=X\cap Y$ for some generic complete intersection variety $Y$ of codimension $e$.
Then
\begin{equation}\label{eq: degrees data loci Chern classes}
\deg(\mathrm{DL}_{X|Z})=\deg(Y)\sum_{k=0}^{n-e}(-1)^k\alpha_k(n,e)\deg(c_k(X))\,,
\end{equation}
where
\begin{equation}\label{eq: def alpha k}
\alpha_k(n,e) \coloneqq
\begin{cases}
\sum_{i=e}^{n-k}\binom{n-k+1}{i+1} & \text{if $0\le k\le n-e=\dim(Z)$}\\
0 & \text{otherwise.}
\end{cases}
\end{equation}
\end{proposition}
\begin{proof}
Using \eqref{eq: polar vs Chern}, we get that
\begin{align}
\begin{split}
\sum_{i=e}^n\delta_i(X) &= \sum_{i=c}^n\sum_{j=i}^n(-1)^{n-j}\binom{j+1}{i+1}\deg(c_{n-j}(X)) = \sum_{i=e}^n\sum_{j=0}^n(-1)^{n-j}\binom{j+1}{i+1}\deg(c_{n-j}(X))\\
&= \sum_{i=e}^n\sum_{k=0}^n(-1)^k\binom{n-k+1}{i+1}\deg(c_k(X)) = \sum_{k=0}^n(-1)^k\deg(c_k(X))\sum_{i=e}^n\binom{n-k+1}{i+1}\,.
\end{split}
\end{align}
Note that $\binom{n-k+1}{i+1}=0$ if $i>n-k$, so $\sum_{i=e}^n\binom{n-k+1}{i+1}=\sum_{i=e}^{n-k}\binom{n-k+1}{i+1}$. The statement follows by Corollary \ref{corol: degrees data loci Z intersection X Y}.
\end{proof}

\begin{remark}
Observe that $\alpha_k(n,0)=2^{n-k+1}-1$ for all $k\in\{0,\ldots,n\}$. Hence, we can interpret Proposition \ref{prop: degrees data loci Chern classes} as a first ``conditional'' version of \cite[Theorem 5.8]{DHOST}, although the assumptions on $Z$ are quite restrictive. In order to drop the additional assumption on $Z$ in Proposition \ref{prop: degrees data loci Chern classes}, it would be interesting also to introduce a notion of Chern classes of vector bundles on a smooth variety $X$ ``relative'' to a subvariety $Z\subset X$. We leave this problem to future research.
\end{remark}

\begin{theorem}\label{thm: degree data loci Catanese}
Let $Z\subset X\subset\PP^N$ be irreducible smooth varieties of dimensions $d$ and $n$, respectively. Suppose that $X$ is ED regular given $Z$, and that $\sPN_{X,Z}$ is disjoint from the diagonal $\Delta(\PP^N)\subset\PP^N\times\PP^N$. Define $\sQ_{X,Z}\coloneqq\sO_Z^{\oplus(N+1)}/\sPE_{X|Z}$. Then
\begin{equation}\label{eq: degree data loci Catanese}
\mathrm{EDD}(X|Z)\deg(\mathrm{DL}_{X|Z})=\deg(c_d(\sQ_{X,Z}))= \int_Z\frac{1}{c(\sN_{X/\PP^N}^\vee|_Z\otimes\sO_Z(1))\cdot c(\sO_Z(-1))}\,,
\end{equation}
where $\int_Z$ indicates the coefficient of the class $c_1(\sO_Z(1))^d$.
\end{theorem}
\begin{proof}
The assumption that $\sPN_{X,Z}$ is disjoint from the diagonal $\Delta(\PP^N)\subset\PP^N\times\PP^N$ and Proposition \ref{prop: projective ED correspondence Z irreducible} ensure that the incidence variety $\sPE_{X|Z}\subset Z\times\C^{N+1}$ is the total space of a projective bundle over $Z$ of rank $N-n+1$. Recalling that $H^0(Z,\sO_Z(1))\cong Z\times\C^{N+1}\cong\sO_Z^{\oplus(N+1)}$, we get a surjection $q\colon \sO_Z^{\oplus(N+1)}\to\sQ_{X,Z}$. In particular $\sQ_{X,Z}$ is a globally generated vector bundle on $Z$ of rank $N+1-(N-n+1)=n$. We claim that
\begin{equation}\label{eq: degree data loci inverse Chern class PE XZ}
\mathrm{EDD}(X|Z)\deg(\mathrm{DL}_{X|Z})=\deg(c_d(\sQ_{X,Z}))\,,
\end{equation}
where $c_d(\sQ_{X,Z})$ denotes the $d$th Chern class of $\sQ_{X,Z}$.
Notice that every $u\in\C^{N+1}$ defines a global section $\sigma_u$ of $\sQ_{X,Z}$ mapping $z\mapsto q(z,u)$. 
By construction we have $\sigma_u(z)=0$ if and only if $(z,u)\in\sPE_{X|Z}$, or equivalently $z$ is a critical point of $d_{X,u}$ lying on $Z$. 
The variety $\mathrm{DL}_{X|Z}$ is the image of the projection $\sPE_{X|Z}\to\C^{N+1}$. Proposition \ref{prop: proj data locus is irreducible} tells us that $\codim(\mathrm{DL}_{X|Z})=n-d.$ Considering $n-d+1$ generic points $u_1,\ldots,u_{n-d+1}$ and letting $L=\langle u_1,\ldots,u_{n-d+1}\rangle$, then $\PP(L)$ meets $\mathrm{DL}_{X|Z}\subset\PP^N$ in finitely many points, whose number equals $\deg(\mathrm{DL}_{X|Z})$.
It follows that if $u\in \mathrm{DL}_{X|Z}\cap L$, then $D(\sigma_u)\neq\emptyset$ by construction, and $D(\sigma_u)$ is zero-dimensional of cardinality $\mathrm{EDD}(X|Z)$ because $X$ is ED regular given $Z$ by assumption.
Because  the section $\sigma_u$ is a linear combination of the sections $\sigma_{u_i}$ it follows that $\sigma_u(z)=0$ if and only if  $\sigma_{u_1}(z)\wedge\cdots\wedge\sigma_{n-d+1}(z)=0$. Summing up, given a generic vector space $L=\langle u_1,\ldots,u_{n-d+1}\rangle$, the degeneracy locus $D(\sigma_{u_1},\ldots,\sigma_{u_{n-d+1}})=\{z\in Z\mid \sigma_{u_1}(z)\wedge\cdots\wedge\sigma_{n-d+1}(z)=0\}$ is zero-dimensional and
\[
\deg(D(\sigma_{u_1},\ldots,\sigma_{u_{n-d+1}}))=\mathrm{EDD}(X|Z)\deg(\mathrm{DL}_{X|Z})\,.
\]
Since $\sQ_{X,Z}$ is a globally generated vector bundle of rank $n$, the cycle $D(\sigma_{u_1},\ldots,\sigma_{u_{n-d+1}})$ is rationally equivalent to $c_d(\sQ_{X,Z})$, see \cite[\S III.3]{griffiths1978principles}. This completes the proof of identity \eqref{eq: degree data loci inverse Chern class PE XZ}.

Consider again the projective bundle $\sPE_{X|Z}\subset Z\times\C^{N+1}$.
We have that $\sPE_{X|Z}=\PP_{{\rm quot}}(\sE_Z)$, namely the projective bundle of hyperplanes in $\sE_Z$, where 
\begin{equation}
\sE_Z\coloneqq(\sN_{X/\PP^N}(-1)\oplus \sO_X(1))|_Z=\sN_{X/\PP^N}(-1)|_Z\oplus \sO_Z(1)\,.
\end{equation}
Using Fulton's notation \cite[Appendix B.5.5]{fulton1998intersection}, we write $\PP_{{\rm quot}}(\sE_Z)=\PP(\sE_Z^\vee)$. Since
\begin{equation}\label{eq: Segre class E_Z dual}
c(\sQ_{X,Z})=\frac{1}{c(\sE_Z^\vee)}=\frac{1}{c(\sN_{X/\PP^N}^\vee(1)|_Z\oplus\sO_Z(-1))}=\frac{1}{c(\sN_{X/\PP^N}^\vee(1)|_Z)\cdot c(\sO_Z(-1))}\,,
\end{equation}
the second identity in \eqref{eq: degree data loci Catanese} follows by \eqref{eq: degree data loci inverse Chern class PE XZ}.
\end{proof}

One may remove the assumption that $\sPN_{X,Z}$ is disjoint from the diagonal $\Delta(\PP^N)\subset\PP^N\times\PP^N$ in Theorem \ref{thm: degree data loci Catanese}. In this case $\sPE_{X|Z}$ may not be a vector bundle over $Z$, but it remains at least a projective cone over $Z$. Therefore, it is still possible to compute the product $\mathrm{EDD}(X|Z)\deg(\mathrm{DL}_{X|Z})$ in terms of the Segre classes of $\sPE_{X|Z}$, or equivalently of $\sE_Z^\vee$, via the identity
\begin{equation}\label{eq: identity segre class}
\mathrm{EDD}(X|Z)\deg(\mathrm{DL}_{X|Z})=\deg(s_d(\sE_Z^\vee))\,.
\end{equation}
We refer to \cite[Chapter 4]{fulton1998intersection} for more details on Segre classes of projective cones.

\begin{remark}
The previous result gives also an alternative proof of Corollary \ref{corol: degree data locus no genericity X, Z with relative polar degrees}.
Restricting the first jet exact sequence \eqref{eq: first jet sequence} to $Z$ yields another exact sequence
\begin{equation}
0\to\sN_{X/\PP^N}^\vee(1)|_Z \to \sO_Z^{\oplus(N+1)}\to\sP^1(\sO_X(1))|_Z\to 0\,.
\end{equation}
Hence, using Whitney's identity we get $c(\sN_{X/\PP^N}^\vee(1)|_Z)=1/c(\sP^1(\sO_X(1))|_Z)$. This allows us to rewrite \eqref{eq: Segre class E_Z dual} as 
\begin{equation}\label{eq: Segre class E_Z dual rewritten}
s(\sE_Z^\vee)=\frac{1}{c(\sE_Z^\vee)}=\frac{c(\sP^1(\sO_X(1))|_Z)}{c(\sO_Z(-1))}\,.
\end{equation}
Calling $\xi_Z=c_1(\sO_Z(1))$, we have $1/c(\sO_Z(-1))=1/(1-\xi_Z)=\sum_{i\ge 0}\xi_Z^i$ and 
\begin{equation}
s_d(\sE_Z^\vee)=\sum_{i=0}^dc_i(\sP^1(\sO_X(1))|_Z)\,\xi_Z^{d-i}\,,
\end{equation}
see also \cite[\S5]{piene2015polar}.
By \eqref{eq: relations relative polar classes Chern classes first jet bundle X restricted Z} we have that $c_i(\sP^1(\sO_X(1))|_Z)\,\xi_Z^{d-i}=\deg(p_i(X,Z))=\mu_i(X,Z)$ for all $i\in\{0,\ldots,d\}$. This yields the formula of Theorem \ref{thm: degree data locus no genericity X, Z}.
\end{remark}

\begin{proof}[Proof of Theorem \ref{thm: degree data loci complete intersection}]
Let $h$ denote the hyperplane divisor in $\PP^N$, so that the cohomology ring $H^*(\PP^N)$ is isomorphic to the quotient ring $\Z[h]/\langle h^n\rangle$. In particular, the total Chern class of the line bundle $\sO_{\PP^N}(1)$ is $1+h$. For every $i\in[s]$, the line bundle $\sO_{X_i}(1)$ is the pullback of $\sO_{\PP^N}(1)$. Since each $X_i$ is a hypersurface, each vector bundle $\sN_{X_i/\PP^N}$ is a line bundle. By \cite[Example II.8.20.3]{hartshorne1977algebraic}, we have $\sN_{X_i/\PP^N}=(\sO_{X_i}(1))^{\otimes d_i}$, so $\sN_{X_i/\PP^N}^\vee\otimes\sO_{X_i}(1)=\sO_{X_i}(-1)^{\otimes(d_i-1)}$.
In $H^*(\PP^N)$ we have
\begin{equation}\label{eq: Catanese complete intersection}
\frac{1}{c(\sO_{\PP^N}(-1))\cdot \prod_{i=1}^s c(\sO_{\PP^N}(-1)^{\otimes(d_i-1)})}=\frac{1}{(1-h)\prod_{i=1}^s(1-(d_i-1)h)}\,.
\end{equation}
By Theorem \ref{thm: degree data loci Catanese}, in order to compute the product $\mathrm{EDD}(X|Z)\deg(\mathrm{DL}_{X|Z})$, we need first to compute the coefficient of $h^d$ in the expression
\[
\frac{1}{c((\sN_{X/\PP^N}^\vee\otimes\sO_Z)\otimes\sO_Z(1))\cdot c(\sO_Z(-1))}\,,
\]
that corresponds to the coefficient of $h^d$ in \eqref{eq: Catanese complete intersection}. The latter equals
\[
\sum_{i_1+\cdots+i_s\le d}(d_1-1)^{i_1}\cdots(d_s-1)^{i_s}\,.
\]
Since the image of $h^d$ in $H^*(Z)$ under pullback is equal to $\deg(Z)$ times the class of a point, we obtain the degree formula \eqref{eq: degree data loci complete intersection}.
\end{proof}

\begin{example}\label{ex: applications formula with X complete intersection}
Consider some applications of Theorem \ref{thm: degree data loci complete intersection}:
\begin{enumerate}
    \item If $Z=\{z\}$ is a point of $X\subset\PP^N$ not in $Q$, then $\mathrm{DL}_{X|Z}$ is a linear space in $\C^{N+1}$, more precisely coincides with the linear span $\langle v, N_vX\rangle$, where $z=[v]$. In particular, it has degree $1$, and also $\mathrm{EDD}(X|Z)=1$.
    The right-hand side of \eqref{eq: degree data loci complete intersection} also equals $1$.
    \item Let $X$ be a smooth quadric surface in $\PP^3$. In particular, it is doubly ruled by two families of pairwise disjoint lines. Let $Z$ be one of such lines.
    Setting $s=1$ and $d_1=2$, the right-hand side of \eqref{eq: degree data loci complete intersection} is equal to $1\cdot(1+1)=2$, therefore $\mathrm{EDD}(X|Z)\deg(\mathrm{DL}_{X|Z})=2$. Notice that $X\cong\PP^1\times\PP^1\hookrightarrow\PP^3$ via the Segre embedding, and $Z\cong P\times\PP^1$ for some point $P\in\PP^1$. We will see in Proposition \ref{prop: double data Segre strictly contained} that $\mathrm{EDD}(X|Z)=1$. Hence we conclude that $\deg(\mathrm{DL}_{X|Z})=2$.
    \item More generally, if $X\subset\PP^N$ is a smooth hypersurface of degree $d$, and if $Z$ is a line contained in $X$, then $\mathrm{EDD}(X|Z)\deg(\mathrm{DL}_{X|Z})=d$. Furthermore, we have $\mathrm{EDD}(X|Z)=1$ by Proposition \ref{prop: double data hypersurface wrt subspace strictly contained}, hence $\deg(\mathrm{DL}_{X|Z})=d$.
    \item Instead, if $X$ is a smooth quadric hypersurface in $\PP^N$, and if $Z$ is a subspace contained in $X$, then $\mathrm{EDD}(X|Z)\deg(\mathrm{DL}_{X|Z})=\dim(Z)+1$. Similarly as before we have $\mathrm{EDD}(X|Z)=1$ and $\deg(\mathrm{DL}_{X|Z})=\dim(Z)+1$.\hfill$\diamondsuit$
\end{enumerate}
\end{example}

\section{Multiple ED data loci and singularities of ED data loci}\label{sec: multiple}

In Definitions \ref{def: conditional ED regularity} and \ref{def: projective conditional ED regularity}, we have introduced the notion of  ED regularity for affine or projective varieties $Z\subset X$. The motivation is to ensure that the conditional ED degree $\mathrm{EDD}(X|Z)$ is well-defined, in particular, that there exists an open dense subset of $\mathrm{DL}_{X|Z}$ (or of $\mathrm{DL}_{X|Z}$ for projective varieties) of data points admitting finitely many critical points of $d_{X,u}$ within $Z$. Now the natural question is: what is the actual value of $\mathrm{EDD}(X|Z)$? Typically, one expects that $\mathrm{EDD}(X|Z)=1$, or that a generic point on $\mathrm{DL}_{X|Z}$ has exactly one critical point of $d_{X,u}$ within $Z$. For this reason, we define the {\em conditional multiple ED data loci of $X$ given $Z$} as
\begin{equation}\label{eq: def ED data locus > s}
\mathrm{DL}_{X|Z}^{>s}\coloneqq\{u\in V\mid|\pr_2^{-1}(u)|\ge s+1\}\quad\forall\,s\in\Z_{\ge 1}\,,
\end{equation}
where $\pr_2$ is the second projection in the diagram \eqref{eq: diagram ED correspondence Z}.
We also define
\begin{equation}\label{eq: def ED data locus infty}
\mathrm{DL}_{X|Z}^{\infty}\coloneqq\{u\in V\mid\dim(\pr_2^{-1}(u))>0\}\,.
\end{equation}
Note that all varieties $\mathrm{DL}_{X|Z}^{>s}$ and $\mathrm{DL}_{X|Z}^{\infty}$ are Zariski closed in $V$ and
\begin{equation}\label{eq: chain containments multiple data loci}
\mathrm{DL}_{X|Z}\supset\mathrm{DL}_{X|Z}^{>1}\supset\mathrm{DL}_{X|Z}^{>2}\supset\cdots\supset\mathrm{DL}_{X|Z}^{\infty}\,.
\end{equation}
In the next result, we list a few basic properties of multiple data loci. 

\begin{lemma}\label{lem: properties multiple data loci}
Let $Z \subset X$ be varieties in $V$ such that $Z\not\subset X_{\sing}$. Then
\begin{enumerate}
    \item There exists an integer $s$ such that $\mathrm{DL}_{X|Z}^{\infty}=\mathrm{DL}_{X|Z}^{>i}$ for all $i\ge s$.
    \item $X$ is ED regular given $Z$ if and only if $\mathrm{DL}_{X|Z}^{\infty}\subsetneq\mathrm{DL}_{X|Z}$.
    \item For all $s\in\Z_{\ge 1}$, $\mathrm{EDD}(X|Z)=s$ if and only if  $\mathrm{DL}_{X|Z}=\mathrm{DL}_{X|Z}^{>1}=\cdots=\mathrm{DL}_{X|Z}^{>s-1}\supsetneq\mathrm{DL}_{X|Z}^{>s}$, where $\mathrm{DL}_{X|Z}^{>0}\coloneqq\mathrm{DL}_{X|Z}$.
    \item If for all pairs $(x_1,x_2)$ of distinct points of $X_{\sm}\cap Z$ we have $(x_1+N_{x_1}X)\cap(x_2+N_{x_2}X)=\emptyset$, then $\mathrm{DL}_{X|Z}^{>1}=\emptyset$.
\end{enumerate}
\end{lemma}
\begin{proof}
The first three properties follow by equations \eqref{eq: def ED data locus > s}, \eqref{eq: def ED data locus infty} and \eqref{eq: chain containments multiple data loci}. Part (4) is an immediate consequence of identity \eqref{eq: alternative description of ED data locus}.
\end{proof}

The next proposition furnishes a criterion for having $\mathrm{EDD}(X|Z)\le s$, using the multiple ED data loci $\mathrm{DL}_{X|Z}^{>s}$, $\mathrm{DL}_{X|Z}^{\infty}$, and the singular locus of $\mathrm{DL}_{X|Z}$.

\begin{proposition}\label{prop: relation EDD X,Z 1 and singular locus data locus}
Let $Z \subset X$ be irreducible varieties in $V$ such that $X$ is ED regular given $Z$. Consider an integer $s\ge 1$ and assume additionally that $\mathrm{DL}_{X|Z}^{>s}$ is irreducible. Then
\[
\mathrm{EDD}(X|Z)\le s \Longleftrightarrow (\mathrm{DL}_{X|Z}^{>s}=\mathrm{DL}_{X|Z}^{\infty})\vee(\mathrm{DL}_{X|Z}^{>s}\subset(\mathrm{DL}_{X|Z})_{\sing})\,.
\]
\end{proposition}
\begin{proof}
The implication ``$\Leftarrow$'' is immediate: if $\mathrm{DL}_{X|Z}^{>s}=\mathrm{DL}_{X|Z}^{\infty}$, then Lemma \ref{lem: properties multiple data loci}(2)-(3) implies that $\mathrm{EDD}(X|Z)\le s$. Otherwise $\mathrm{DL}_{X|Z}^{>s}\subset(\mathrm{DL}_{X|Z})_{\sing}$, and $(\mathrm{DL}_{X|Z})_{\sing}\subsetneq\mathrm{DL}_{X|Z}$ by generic smoothness of $\mathrm{DL}_{X|Z}$, yielding $\mathrm{DL}_{X|Z}^{>s}\subsetneq\mathrm{DL}_{X|Z}$ and hence $\mathrm{EDD}(X|Z)\le s$.

We now show ``$\Rightarrow$''. Without loss of generality, we assume that $\mathrm{EDD}(X|Z)=s$: indeed, proving this case also implies the case $\mathrm{EDD}(X|Z)\le s$. For all $i\in[s+1]$, let $V_x^{(i)}$ be a copy of $V_x$, and define $\Sigma_{s+1}(V)\coloneqq (V_x^{(1)}\times\cdots\times V_x^{(s+1)})/\Sigma_{s+1}$, where $\Sigma_{s+1}$ is the symmetric group on $s+1$ elements which acts naturally on $V_x^{(1)}\times\cdots\times V_x^{(s+1)}$ by permuting its elements $x=(x_1,\ldots,x_{s+1})$.
Define the incidence variety
\begin{equation}
\sE_{X|Z}^{>s}\coloneqq\overline{\{(x,u)\in\Sigma_{s+1}(V)\times V_u\mid\text{$x_i\in X_{\sm}\cap Z$, $u-x_i\in N_{x_i}X$ and $x_i\neq x_j$ $\forall\,i\neq j$}\}}\,.
\end{equation}
It follows almost by definition that $\mathrm{DL}_{X|Z}^{>s}=\overline{\pr_2^{>s}(\sE_{X|Z}^{>s})}$, where $\pr_2^{>s}\colon\sE_{X|Z}^{>s}\to V_u$ is the projection.
Furthermore, the assumption $\mathrm{EDD}(X|Z)=s$ implies that $\pr_2^{>(s-1)}$ is birational onto its image $\mathrm{DL}_{X|Z}^{>(s-1)}=\mathrm{DL}_{X|Z}$.
Suppose that there exists a point $u\in\mathrm{DL}_{X|Z}^{>s}\cap(\mathrm{DL}_{X|Z})_{\sm}$. Then necessarily $|(\pr_2^{>(s-1)})^{-1}(u)|\ge 2$, and $u$ is a normal point of the image of $\pr_2^{>(s-1)}$. By Zariski's Main Theorem \cite[Corollary III.11.4]{hartshorne1977algebraic}, every fiber of a birational morphism between noetherian integral schemes at a normal point is connected. Since $|(\pr_2^{>(s-1)})^{-1}(u)|\ge 2$, the only possibility is that $\dim((\pr_2^{>(s-1)})^{-1}(u))>0$. Therefore $u\in\mathrm{DL}_{X|Z}^{\infty}$. Hence $\mathrm{DL}_{X|Z}^{>s}\cap(\mathrm{DL}_{X|Z})_{\sm}$ is nonempty and is also dense in $\mathrm{DL}_{X|Z}^{>s}$ because $\mathrm{DL}_{X|Z}^{>s}$ is irreducible by assumption. Summing up, we have shown that $\mathrm{DL}_{X|Z}^{>s}\subset\mathrm{DL}_{X|Z}^{\infty}$. The other inclusion follows by definition of $\mathrm{DL}_{X|Z}^{\infty}$.
The other possibility is that $\mathrm{DL}_{X|Z}^{>s}\cap(\mathrm{DL}_{X|Z})_{\sm}=\emptyset$, but then 
$\mathrm{DL}_{X|Z}^{>s}\subset(\mathrm{DL}_{X|Z})_{\sing}$.
\end{proof}

We leave to future research the study of the irreducibility of the varieties $\mathrm{DL}_{X|Z}^{>s}$ and $\mathrm{DL}_{X|Z}^{\infty}$.

\begin{remark}
Under the assumptions of Proposition \ref{prop: relation EDD X,Z 1 and singular locus data locus}, the condition $\mathrm{EDD}(X|Z)\le s$ is not sufficient to have $\mathrm{DL}_{X|Z}^{>s}=\mathrm{DL}_{X|Z}^{\infty}$.
We consider two counterexamples for $s=1$:
\begin{enumerate}
    \item Let $X$ be any linear or affine subspace in $V$, and $Z$ a subvariety of $X$ that is singular. On one hand, it is easy to see that the hypotheses of Lemma \ref{lem: properties multiple data loci} are satisfied, hence $\mathrm{DL}_{X|Z}^{>1}=\emptyset$ and therefore $\mathrm{EDD}(X|Z)=1$. On the other hand, for every $x\in Z_{\sing}$, every point in the linear space $x+N_xX$ is singular for $\mathrm{DL}_{X|Z}$, hence in this case trivially $\emptyset=\mathrm{DL}_{X|Z}^{>1}\subsetneq(\mathrm{DL}_{X|Z})_{\sing}$.
    \item In \cite{ottaviani2013matrices}, the authors consider the projective space $\PP^{n^2-1}$ of complex $n\times n$ matrices, and study the locus of matrices having at least one eigenvector in a fixed subspace $L\subset\PP^{n-1}$. Also, this variety may be interpreted as a data locus, which they call {\em Kalman variety} $\sK_{d,n}(L)$, where $d=\dim(L)$. In particular, the data points are represented by matrices $A$, while the eigenvectors of $A$ play the role of the critical points. In this case, $(\sK_{d,n}(L))_{\sing}$ equals the locus of matrices having a two-dimensional linear subspace $L'\subset L$ which is $A$-invariant, and the generic element of $(\sK_{d,n}(L))_{\sing}$ has precisely two eigenvectors on $L$ (see \cite[Lemma 4.1]{ottaviani2013matrices} for more details). In our setting, this means that $\mathrm{EDD}(X|Z)=1$, but also that
    \[\mathrm{DL}_{X|Z}^{\infty}\subsetneq\mathrm{DL}_{X|Z}^{>1}=(\mathrm{DL}_{X|Z})_{\sing}\,.
    \]
    Kalman varieties are also studied in the more general context of tensors \cite{OSh,shahidi2021degrees}, and we have briefly discussed them before Theorem \ref{thm: Kalman}. In this case, it might happen that the inclusion $\mathrm{DL}_{X|Z}^{>1}\subset(\mathrm{DL}_{X|Z})_{\sing}$ is also strict.
\end{enumerate}
\end{remark}

\begin{remark}
Under the assumptions of Proposition \ref{prop: relation EDD X,Z 1 and singular locus data locus}, the condition $\mathrm{EDD}(X|Z)\le s$ is not sufficient to have $\mathrm{DL}_{X|Z}^{>s}\subset(\mathrm{DL}_{X|Z})_{\sing}$. For example, in the first part of Example \ref{ex: projection not birational} we consider a sphere $S$ and a maximal circle $Z\subset S$. On one hand $\mathrm{EDD}(S|Z)=2$, but $(\mathrm{DL}_{X|Z})_{\sing}=\emptyset$ and $\mathrm{DL}_{X|Z}^{>2}=\mathrm{DL}_{X|Z}^{\infty}$ is precisely the center of $S$.
\end{remark}

In the last part of this section, we concentrate on sufficient conditions to have $\mathrm{EDD}(X|Z)=1$.

\begin{proposition}\label{prop: double data strictly contained}
Let $X\subset V$ be an irreducible affine variety, and let $Z\subset X$ be a nonempty subvariety obtained intersecting $X$ with a generic complete intersection variety $Y$ of codimension $c$. Then $\mathrm{EDD}(X|Z)=1$. A similar result holds for projective varieties.
\end{proposition}
\begin{proof}
The genericity of $Y$ implies that $X$ is ED regular given $Z$, particularly that $Z$ is not contained in the ramification locus of the projection $\pr_2\colon\sE_X\to V$.
Therefore, there exists a point $x\in Z\cap X_{sm}$ such that $\pr_2^{-1}(x)$ is reduced and consists of $\mathrm{EDD}(X)$ simple points $x_1,\ldots,x_{\mathrm{EDD}(X)}$. Up to relabeling, we may assume that $x=x_1$.
Let $f_1,\ldots,f_c$ be the generators of the vanishing ideal of $Y$.
Since all polynomials $f_i$ are generic, there exists $i\in[c]$ such that $x_j\notin V(f_i)$ for all $j\in\{2,\ldots,\mathrm{EDD}(X)\}$. For this reason, we conclude that $x\in\mathrm{DL}_{X|Z}\setminus\mathrm{DL}_{X|Z}^{>1}$. Since $\mathrm{DL}_{X|Z}^{>1}$ is Zariski closed within $\mathrm{DL}_{X|Z}$, it is necessarily a proper subvariety, and by Lemma \ref{lem: properties multiple data loci} this is equivalent to the condition $\mathrm{EDD}(X|Z)=1$. The proof in the projective case is the same.
\end{proof}

We extend the previous result to Segre products of projective varieties. This new result is more technical and is needed in Example \ref{ex: degree Kalman}. An even more general statement can be formulated for Segre-Veronese products of projective varieties.

\begin{proposition}\label{prop: double data Segre strictly contained}
Let $k\ge 1$ be an integer. For all $i\in[k]$, let $X_i\subset\PP(V_i)$ be an irreducible projective variety of dimension $n_i$. Define $X$ as the image of $X_1\times\cdots\times X_k$ under the Segre embedding $\sigma\colon\PP(V_1)\times\cdots\times\PP(V_k)\hookrightarrow\PP(V)$, where $V\coloneqq V_1\otimes\cdots\otimes V_k$. For all $i\in[k]$, let $Z_i\subset X_i$ be a nonempty subvariety obtained intersecting $X_i$ with a generic complete intersection variety $Y_i$ of codimension $c_i$, and define $Z\subset X$ as the image of $Z_1\times\cdots\times Z_k$ under $\sigma$. Then $\mathrm{EDD}(X|Z)=1$.
\end{proposition}
\begin{proof}
The argument follows the one in the proof of Proposition \ref{prop: double data strictly contained}. Also in this case, there exists a point $x=\sigma(\xi_1,\ldots,\xi_k)\in Z\cap X_{sm}$ such that $\pr_2^{-1}(x)$ is reduced and consists of $\mathrm{EDD}(X)$ simple points $x_1,\ldots,x_{\mathrm{EDD}(X)}$. For each $i\in[\mathrm{EDD}(X)]$, we can write $x_i=\sigma(\xi_1^{(i)},\ldots,\xi_k^{(i)})$, where $\xi_j^{(i)}\in X_j$ for all $j\in[k]$. Up to relabeling, we may assume that $x=x_1$.
For all $j\in[k]$, let $f_{j,1},\ldots,f_{j,c_j}$ be the generators of the vanishing ideal of $Y_i$.
Since all polynomials $f_{j,\ell_j}$ are generic, there exists $\ell_j\in[c_j]$ such that $\xi_j^{(i)}\notin V(f_{j,\ell_j})$ for all $i\in\{2,\ldots,\mathrm{EDD}(X)\}$. For this reason, we conclude that $x\in\mathrm{DL}_{X|Z}\setminus\mathrm{DL}_{X|Z}^{>1}$ and, therefore, $\mathrm{EDD}(X|Z)=1$.
\end{proof}

\begin{proposition}\label{prop: double data hypersurface wrt subspace strictly contained}
Let $X\subset\PP^N$ be a smooth hypersurface, and consider a linear subspace $Z\subset X$ not contained in the isotropic quadric $Q$. Then $\mathrm{EDD}(X|Z)=1$.
\end{proposition}
\begin{proof}
Let $d=\dim(Z)$.
Without loss of generality, we assume that $Z=V(x_0,\ldots,x_{N-d-1})$. Since $Z\subset X$, we can write $X=V(f)$, where $f=x_0f_0+\cdots +x_{N-d-1}f_{N-d-1}$ for some homogeneous polynomials $f_0,\ldots,f_{N-d-1}$ of degree $\deg(X)-1$. For all $x\in X$, the normal space $N_xX$ has affine dimension one and is spanned by $\nabla f(x)$. In particular
\begin{equation}
    \frac{\partial f}{\partial x_i}(x) =
    \begin{cases}
    f_i(x)+\sum_{j=0}^{N-d-1}x_j\frac{\partial f_j}{\partial x_j}(x) & \text{if $0\le i\le N-d-1$}\\
    \sum_{j=0}^{N-d-1}x_j\frac{\partial f_j}{\partial x_j}(x) & \text{if $N-d\le i\le N$.}
    \end{cases}
\end{equation}
This implies that, for all $x\in Z$, we have $\nabla f(x)=(f_0(x),\ldots,f_{N-d-1}(x),0,\ldots,0)$. Now suppose that there exist two vectors $z_1\neq z_2$ in $Z$ such that $(z_1+N_{z_1}X)\cap(z_2+N_{z_2}X)\neq\emptyset$. More precisely, we assume that $z_1+\lambda_1\nabla f(z_1)=z_2+\lambda_2\nabla f(z_2)$ for some coefficients $\lambda_1,\lambda_2$. On one hand, the first $N-d$ components of the vectors $z_i$ are zero. On the other hand, the last $d+1$ components of the gradients $\nabla f(z_i)$ are zero. Then necessarily $z_1=z_2$, a contradiction. Therefore $\mathrm{EDD}(X|Z)=1$ by Lemma \ref{lem: properties multiple data loci}(4).
\end{proof}

\bibliographystyle{alpha}
\bibliography{biblio}

\newcommand{\etalchar}[1]{$^{#1}$}
\begin{thebibliography}{DREW20}

\bibitem[ACGH85]{arbarello1985geometry}
E.~Arbarello, M.~Cornalba, P.~A. Griffiths, and J.~Harris.
\newblock {\em Geometry of algebraic curves. {V}ol. {I}}, volume 267 of {\em
  Grundlehren der mathematischen Wissenschaften}.
\newblock Springer-Verlag, New York, 1985.

\bibitem[AGK{\etalchar{+}}21]{amendola2021maximum}
C.~Am\'{e}ndola, L.~Gustafsson, K.~Kohn, O.~Marigliano, and A.~Seigal.
\newblock The maximum likelihood degree of linear spaces of symmetric matrices.
\newblock {\em Matematiche (Catania)}, 76(2):535--557, 2021.

\bibitem[Ait39]{aitken1939determinants}
A.~C. Aitken.
\newblock {\em {D}eterminants and {M}atrices}.
\newblock Oliver \& Boyd, 1939.

\bibitem[BBE{\etalchar{+}}21]{blair2021phenotipic}
P.~W. Blair, J.~Brandsma, N.~J. Epsi, S.~A. Richard, D.~Striegel, J.~Chenoweth,
  R.~Mehta, E.~Clemens, A.~Malloy, C.~Lanteri, J.~S. Dumler, D.~Tribble,
  T.~Burgess, S.~Pollett, B.~Agan, and D.~Clark.
\newblock Phenotypic differences between distinct immune biomarker clusters
  during the 'hyperinflammatory' middle-phase of {COVID-19}.
\newblock {\em Open Forum Infectious Diseases}, 8(Supplement 1):S320--S321,
  November 2021.

\bibitem[BPT12]{blekherman2012semidefinite}
G.~Blekherman, P.~A. Parrilo, and R.~R. Thomas.
\newblock {\em Semidefinite optimization and convex algebraic geometry}.
\newblock SIAM, 2012.

\bibitem[CHKS06]{catanese2006maximum}
F.~Catanese, S.~Ho\c{s}ten, A.~Khetan, and B.~Sturmfels.
\newblock The maximum likelihood degree.
\newblock {\em Amer. J. Math.}, 128(3):671--697, 2006.

\bibitem[cJM{\etalchar{+}}21]{celik2021wasserstein}
T.~\"{O}. \c{C}elik, A.~Jamneshan, G.~Mont\'{u}far, B.~Sturmfels, and
  L.~Venturello.
\newblock Wasserstein distance to independence models.
\newblock {\em J. Symbolic Comput.}, 104:855--873, 2021.

\bibitem[CRSW22]{cifuentes2022voronoi}
D.~Cifuentes, K.~Ranestad, B.~Sturmfels, and M.~Weinstein.
\newblock Voronoi cells of varieties.
\newblock {\em J. Symbolic Comput.}, 109:351--366, 2022.

\bibitem[Cul13]{cullis1913matrices}
C.~E. Cullis.
\newblock {\em Matrices and Determinoids: Volume 1}.
\newblock Calcutta University Readership Lectures. Cambridge University Press,
  1913.

\bibitem[DHMN24]{depaul2024degrees}
G.~DePaul, S.~Ho{\c{s}}ten, N.~Metya, and I.~Nometa.
\newblock Degrees of the {W}asserstein distance to small toric models.
\newblock {\em \arxiv{2402.09626}}, 2024.

\bibitem[DHO{\etalchar{+}}16]{DHOST}
J.~Draisma, E.~Horobe\c{t}, G.~Ottaviani, B.~Sturmfels, and R.~R. Thomas.
\newblock The {E}uclidean distance degree of an algebraic variety.
\newblock {\em Found. Comput. Math.}, 16(1):99--149, 2016.

\bibitem[DREW20]{dirocco2020bottleneck}
S.~Di~Rocco, D.~Eklund, and M.~Weinstein.
\newblock The {B}ottleneck {D}egree of {A}lgebraic {V}arieties.
\newblock {\em SIAM Journal on Applied Algebra and Geometry}, 4(1):227--253,
  2020.

\bibitem[FO14]{friedland2014number}
S.~Friedland and G.~Ottaviani.
\newblock The number of singular vector tuples and uniqueness of best rank-one
  approximation of tensors.
\newblock {\em Found. Comput. Math.}, 14(6):1209--1242, 2014.

\bibitem[Ful98]{fulton1998intersection}
W.~Fulton.
\newblock {\em Intersection theory}, volume~2 of {\em Ergebnisse der Mathematik
  und ihrer Grenzgebiete. 3. Folge. A Series of Modern Surveys in Mathematics
  [Results in Mathematics and Related Areas. 3rd Series. A Series of Modern
  Surveys in Mathematics]}.
\newblock Springer-Verlag, Berlin, second edition, 1998.

\bibitem[GH78]{griffiths1978principles}
P.~Griffiths and J.~Harris.
\newblock {\em Principles of algebraic geometry}.
\newblock Pure and Applied Mathematics. Wiley-Interscience [John Wiley \&
  Sons], New York, 1978.

\bibitem[GKZ94]{gelfand1994discriminants}
I.~M. Gel'fand, M.~M. Kapranov, and A.~V. Zelevinsky.
\newblock {\em Discriminants, resultants, and multidimensional determinants}.
\newblock Mathematics: Theory \& Applications. Birkh\"{a}user Boston, Inc.,
  Boston, MA, 1994.

\bibitem[GS97]{GS}
D.~Grayson and M.~Stillman.
\newblock Macaulay 2--a system for computation in algebraic geometry and
  commutative algebra, 1997.

\bibitem[Har77]{hartshorne1977algebraic}
R.~Hartshorne.
\newblock {\em Algebraic geometry}.
\newblock Graduate Texts in Mathematics, No. 52. Springer-Verlag, New
  York-Heidelberg, 1977.

\bibitem[Har92]{harris1992algebraic}
J.~Harris.
\newblock {\em Algebraic geometry}, volume 133 of {\em Graduate Texts in
  Mathematics}.
\newblock Springer-Verlag, New York, 1992.
\newblock A first course.

\bibitem[Hol88]{holme1988geometric}
A.~Holme.
\newblock The geometric and numerical properties of duality in projective
  algebraic geometry.
\newblock {\em Manuscripta Math.}, 61(2):145--162, 1988.

\bibitem[Hor17]{horobet2017data}
E.~Horobe{\c{t}}.
\newblock The data singular and the data isotropic loci for affine cones.
\newblock {\em Communications in Algebra}, 45(3):1177--1186, 2017.

\bibitem[HR22]{horobet2022data}
E.~Horobe\c{t} and J.~I. Rodriguez.
\newblock Data loci in algebraic optimization.
\newblock {\em J. Pure Appl. Algebra}, 226(12):Paper No. 107144, 15, 2022.

\bibitem[HT23]{horobet2023does}
E.~Horobe\c{t} and E.~T. Turatti.
\newblock When does subtracting a rank-one approximation decrease tensor rank?
\newblock {\em \arxiv{2303.14985}}, 2023.

\bibitem[KKS21]{kubjas2021algebraic}
K.~Kubjas, O.~Kuznetsova, and L.~Sodomaco.
\newblock Algebraic degree of optimization over a variety with an application
  to $p$-norm distance degree.
\newblock {\em \arxiv{2105.07785}}, 2021.
\newblock To appear on Acta Univ. Sapientiae Math. - Special Issue on Metric
  Algebraic Geometry.

\bibitem[Kle86]{kleiman1986tangency}
S.~L. Kleiman.
\newblock Tangency and duality.
\newblock In {\em Proceedings of the 1984 {V}ancouver conference in algebraic
  geometry}, volume~6 of {\em CMS Conf. Proc.}, pages 163--225. Amer. Math.
  Soc., Providence, RI, 1986.

\bibitem[KMQS23]{kozhasov2023minimal}
K.~Kozhasov, A.~Muniz, Y.~Qi, and L.~Sodomaco.
\newblock On the minimal algebraic complexity of the rank-one approximation
  problem for general inner products.
\newblock {\em \arxiv{2309.15105}}, 2023.

\bibitem[Lan12]{landsberg2012tensors}
J.~M. Landsberg.
\newblock {\em Tensors: geometry and applications}, volume 128 of {\em Graduate
  Studies in Mathematics}.
\newblock American Mathematical Society, Providence, RI, 2012.

\bibitem[MRW24]{meroni2024algebraic}
C.~Meroni, B.~Reinke, and K.~Wang.
\newblock The algebraic degree of the {W}asserstein distance.
\newblock {\em \arxiv{2401.12735}}, 2024.

\bibitem[NLC11]{nicolau2011topology}
M.~Nicolau, A.~J. Levine, and G.~Carlsson.
\newblock Topology based data analysis identifies a subgroup of breast cancers
  with a unique mutational profile and excellent survival.
\newblock {\em Proceedings of the National Academy of Sciences},
  108(17):7265--7270, 2011.

\bibitem[NR09]{nie2009algebraic}
J.~Nie and K.~Ranestad.
\newblock Algebraic degree of polynomial optimization.
\newblock {\em SIAM J. Optim.}, 20(1):485--502, 2009.

\bibitem[OS13]{ottaviani2013matrices}
G.~Ottaviani and B.~Sturmfels.
\newblock Matrices with eigenvectors in a given subspace.
\newblock {\em Proc. Amer. Math. Soc.}, 141(4):1219--1232, 2013.

\bibitem[OS21]{OSh}
G.~Ottaviani and Z.~Shahidi.
\newblock Tensors with eigenvectors in a given subspace.
\newblock {\em Rend. Circ. Mat. Palermo, II. Ser}, pages 1--12, 2021.

\bibitem[PFG02]{prells2002compound}
U.~Prells, M.~I. Friswell, and S.~D. Garvey.
\newblock Compound matrices and {P}faffians: a representation of geometric
  algebra.
\newblock In {\em Applications of geometric algebra in computer science and
  engineering ({C}ambridge, 2001)}, pages 109--118. Birkh\"{a}user Boston,
  Boston, MA, 2002.

\bibitem[Pie78]{piene1978polar}
R.~Piene.
\newblock Polar classes of singular varieties.
\newblock {\em Ann. Sci. \'{E}cole Norm. Sup. (4)}, 11(2):247--276, 1978.

\bibitem[Pie15]{piene2015polar}
R.~Piene.
\newblock Polar varieties revisited.
\newblock In {\em Computer algebra and polynomials}, volume 8942 of {\em
  Lecture Notes in Comput. Sci.}, pages 139--150. Springer, Cham, 2015.

\bibitem[RS13]{rostalski2013dualities}
P.~Rostalski and B.~Sturmfels.
\newblock Dualities.
\newblock In {\em Semidefinite optimization and convex algebraic geometry},
  volume~13 of {\em MOS-SIAM Ser. Optim.}, pages 203--249. SIAM, Philadelphia,
  PA, 2013.

\bibitem[SSV23]{shahidi2021degrees}
Z.~Shahidi, L.~Sodomaco, and E.~Ventura.
\newblock Degrees of {K}alman varieties of tensors.
\newblock {\em J. Symbolic Comput.}, 114:74--98, 2023.

\bibitem[SU10]{sturmfels2010multivariate}
B.~Sturmfels and C.~Uhler.
\newblock Multivariate {G}aussian, semidefinite matrix completion, and convex
  algebraic geometry.
\newblock {\em Ann. Inst. Statist. Math.}, 62(4):603--638, 2010.

\bibitem[{The}22]{sagemath}
{The Sage Developers}.
\newblock {\em {S}ageMath, the {S}age {M}athematics {S}oftware {S}ystem
  ({V}ersion 9.5)}, 2022.
\newblock \href{https://www.sagemath.org}{www.sagemath.org}.

\end{thebibliography}

\end{document}